\documentclass[12pt]{article}

\usepackage[margin=1in]{geometry}

\usepackage{natbib}
\bibpunct[, ]{(}{)}{,}{a}{}{,}%
\usepackage[table]{xcolor}
\usepackage{graphicx,subfigure}
\usepackage{float}
\usepackage{multirow}%
\usepackage{hhline}%
\usepackage{amsmath,amssymb,amsthm}
\usepackage{bm}
\usepackage{epstopdf}
\usepackage{makecell}
\usepackage{comment}
\usepackage{algorithm}
\usepackage{algpseudocode}
\usepackage{longtable}
\usepackage{enumitem}
\usepackage{diagbox}
\usepackage{lscape}
\usepackage{url}

\usepackage{thmtools}
\usepackage{thm-restate}
\usepackage{setspace}

\usepackage{booktabs}
\usepackage{threeparttable}
\renewcommand{\qed}{\hfill \rule{1.5ex}{1.8ex}\vskip5pt}
\newtheorem{thm}{Theorem}[section]

\newtheorem{theorem}[thm]{Theorem}
\newtheorem{lemma}[thm]{Lemma}
\newtheorem{proposition}[thm]{Proposition}
\newtheorem{corollary}[thm]{Corollary}
\newtheorem{remark}[thm]{Remark}

\definecolor{cadmiumgreen}{rgb}{0.0, 0.52, 0.24}

\usepackage{tikz}
\newcommand{\UCB}{\operatorname{UCB}}
\newcommand{\LCB}{\operatorname{LCB}}
\newcommand{\TimeToReplace}{\textsc{TimeToReplace}}
\newcommand{\ChooseTarget}{\textsc{ChooseTarget}}
\newcommand{\ConstructPairing}{\textsc{ConstructPairing}}
\newcommand{\DRUCB}{\textsc{DR-UCB}\,}
\newcommand{\epoch}{\operatorname{epoch}}
\newcommand{\polylog}{\operatorname{polylog}}
\newcommand{\Regret}{\operatorname{Regret}}
\newcommand{\kl}{\operatorname{kl}}
\DeclareMathOperator*{\argmin}{arg\,min}
\DeclareMathOperator*{\argmax}{arg\,max}

\newcommand{\ind}[1]{\mathbb{I}\left\{#1\right\}}

\newcommand{\bE}[1]{{\mathbb{E}\left[#1\right]}}
\newcommand{\bbE}[2]{{{\mathbb{E}_{#1}\left[#2\right]}}}

\newcommand{\bP}[1]{{\mathbb{P}\left(#1\right)}}

\usepackage{tcolorbox}
\usetikzlibrary{decorations.pathreplacing} 

\begin{document}

\title{Sequential Hiring of Contingent Workers \\
Through Learning-Based Optimization}

\author{
  Chris Lee\\
  Department of Industrial and Operations Engineering,\\
  University of Michigan\\
  \texttt{chrisor@umich.edu}
  \and
  Xiuli Chao\\
  Department of Industrial and Operations Engineering,\\
  University of Michigan\\
  \texttt{chao@umich.edu}
  \and
  Izak Duenyas\\
  Ross School of Business,\\
  University of Michigan\\
  \texttt{duenyas@umich.edu}
}

\date{}
\maketitle

\begin{abstract}
In this paper, we study a sequential workforce management problem in a contingent labor setting with uncertainty in both worker production and labor supply. A firm seeks to maximize cumulative profit by maintaining an active team of fixed size while learning worker productivity over time. We emphasize two critical operational frictions in this problem: replacing workers is costly, and workers may not be available immediately for hiring because of, for example, prior job commitments, scheduling constraints, or onboarding procedures. Thus, hiring decisions take effect only after a random delay. We formulate this problem as a stochastic multi-play bandit with costly switching and delayed actions, and develop a learning-based hiring policy, DR-UCB (DelayedReplacement-UCB), that makes replacement and hiring decisions sequentially through learning cycles. In each cycle, the policy uses real-time production data to determine when to initiate workforce changes and which workers to replace and hire. We show that the leading-order regret of the proposed policy matches its lower bound in its dependence on the time horizon. Our numerical experiments show that {\DRUCB} outperforms benchmark policies.
\end{abstract}

\noindent {\it Keywords:} hiring, workforce, worker availability, learning, bandit, sequential decision, regret analysis.

\medskip


\section{Introduction}
Contingent work, including temporary and gig work, has become a defining feature of modern labor markets \citep{wood2019good, wu2024gig}. Millions of employers now find short-term workers through staffing intermediaries and digital platforms. In these settings, the ability to select effective workers is especially important, since workforce management has long been recognized as a source of competitive advantage \citep{bartlett2002building}. At the same time, firms in the contingent hiring setting face persistent challenges in assessing worker productivity before hiring, making worker selection an inherently difficult operational decision.

In recent years, technological advances have transformed firms' ability to observe and measure performance. In many operational settings, granular performance data can be collected in near real time and used to estimate the performance of each worker. For example, Amazon reportedly tracks, down to the second, how long it takes a warehouse associate to retrieve a product \citep{Vincent2015}. Such data create an opportunity for firms to update their assessments of worker productivity over time and to improve hiring and retention decisions through adaptive, algorithmic decision tools.

To illustrate, suppose a firm receives a pool of pre-screened candidates from a staffing platform such as Instawork, Upwork, or Fiverr, and needs to staff a certain number, say \(m\), of workers (e.g., warehouse stations, delivery routes, etc.). The firm keeps internal records of who has performed well in the past, using measures such as throughput, error rates, and sales, and uses those records to decide whom to bring back and whom to replace, paying a placement fee for each new worker it books.

However, replacing workers is not instantaneous. Even after the firm decides to hire a new worker, that worker may not be available immediately and can start only several periods later. Thus, although contingent staffing gives firms the flexibility to observe performance and make repeated short-term booking decisions, this flexibility is limited by worker availability and replacement costs. This tension motivates the central question of this paper: how can firms optimize worker replacement and hiring when workers' productivity is uncertain and unknown a priori, replacements are costly, and new workers are not always available?

In this paper, we develop a learning-based approach to this sequential workforce management problem. The firm's objective is to maximize the cumulative profit of its workforce while adapting to worker productivity uncertainty and new worker unavailability.

\vspace{-.1in}

\subsection{Empirical Motivation for the Modeling Framework}\label{sec:hiring-background}

This section describes the empirical and operational features of contingent hiring that motivate our model. Three features are central: workers differ substantially in productivity, pre-hire information is often limited, and contingent staffing decisions must be made quickly under availability and replacement frictions.

{\bf Hiring performance in practice.}
Empirical and organizational research documents substantial heterogeneity in worker performance, even within narrowly defined roles. For example, \cite{schmidt1984selection} report that, for one job class, the standard deviation of dollar-value output is roughly 40\% of annual salary. This variation makes worker selection economically meaningful. However, traditional pre-hire assessments often provide limited information about subsequent job performance. Resumes, cover letters, and background characteristics frequently explain little of the variation in later outcomes \citep{gordon2006identifying, sackett2022revisiting, risavy2022resumes, murphy2009predictive}. These findings motivate a model in which worker productivity is initially uncertain and must be learned from observed performance data.

{\bf Challenges in temporary and contingent labor markets.}
Temporary and contingent hiring differs from traditional employment in both pace and information availability. Firms often need to fill roles quickly and at relatively low cost, which limits the scope for extensive screening before each assignment. Moreover, because these assignments are short-term and typically involve routine or well-defined tasks, differences in observed performance are likely to depend mainly on the workers assigned to the positions. On the other hand, there is always noise in a worker's daily productivity. We therefore model each worker's productivity as random, with an unknown probability distribution. This abstraction captures the operational setting in which firm production output depends primarily on which workers are selected, while the firm must learn worker productivity over time from their realized performance data.

In this paper, we formulate the sequential hiring of contingent workers as a stochastic multi-play bandit problem with replacement costs and random arm availability. However, any arm that is played in a period is assumed to be available in the following period, reflecting in the hiring problem that the worker will continue to be on the payroll. The \emph{replacement costs} capture financial frictions such as platform placement fees and administrative hiring costs, and \emph{random availability} captures worker supply uncertainty, recruiting procedures, or scheduling delays. We develop an algorithmic approach based on bandit learning, which provides a natural mechanism for balancing exploration (learning about candidate potential) and exploitation (selecting the best-performing workers), to maximize the firm's total profit over any horizon, defined as the total revenue generated by the workforce minus the total cost.

\vspace{-.1in}

\subsection{Main contributions}
The main contributions of this paper are summarized as follows.

{\bf Modeling contributions.}
We formulate contingent hiring as a sequential learning problem in which a firm seeks to maximize profit by learning worker productivity from observed performance data. This setting naturally suggests a multi-play bandit formulation, but standard bandit models abstract away from operational frictions that are central to contingent labor markets: workforce adjustments are non-instantaneous, and hiring and replacing contingent workers are costly.
We capture these features by introducing a multi-play bandit model with random action delays and switching costs. In each period, the firm employs a fixed number of workers. Workers currently employed by the firm remain available in the next period, while workers outside the current workforce may be unavailable for random lengths of time. Thus, unlike standard models with exogenous arm availability, workers' future availability in our model depends on the firm's past staffing decisions. To the best of our knowledge, this multi-play bandit model with endogenous random availability and switching costs has not been studied previously. The model, theoretical results, and regret analysis are therefore new and of independent interest beyond the contingent labor hiring application.

{\bf Methodological contributions.}
We develop a learning-based hiring policy, {\DRUCB}, that makes hiring decisions sequentially via learning cycles. We show that the leading-order regret of the proposed policy matches its lower bound in its dependence on the time horizon. The new features in our model present challenges absent in the analysis of standard bandit models. In particular, random action delay disrupts the structural alignment implicitly relied upon in standard bandit analyses. This misalignment appears in three ways: the replacement schedule must be adaptive because the time required to gather data depends on realized delays; suboptimal workers may be identified statistically before they can actually be removed, generating additional regret; and, because a hired worker cannot enter the workforce until their unavailability delay is resolved, each replacement can separate into two decisions: whom to hire now, and whom to remove later when the incoming worker becomes available. Our algorithm design addresses each of these challenges induced by random availability.

First, we design a dynamic policy with an \emph{adaptive} switching schedule that responds to realized unavailability. Such adaptivity lies outside the scope of existing analyses for switching-cost bandits. This differs significantly from the deterministic switching schedules in \cite{agrawal1990switching, gao2019batched, perchet2016batched}, which all determine the switching times at time 0. This adaptivity is a driving force for achieving minimal regret in delayed environments. To control its regret, we bound the number of switches through a combinatorial packing argument. The key idea is to derive a sample-path lower bound on the growth of the minimum play count within the target workforce. By construction, the frequency of switching decreases as this quantity grows, so lower bounding its growth yields an upper bound on switching regret. Because the argument is pathwise, it is agnostic to the coupling between switching decisions and realized rewards and thus applies directly to adaptive policies. We believe this approach can be useful more broadly in analyzing settings where switching is governed by endogenous, data-dependent rules.

Second, we develop new sampling-regret arguments to handle the mismatch between identification and removal induced by unavailability in a multi-play setting. In our model, a suboptimal worker may be statistically identified before the policy is able to replace that worker, creating an additional source of regret absent from standard analyses. We address this through a stage-based decomposition that separates regret due to statistical uncertainty from regret caused by the delay between a replacement decision and its realization. The former can be controlled with standard tools, but the latter requires new ideas. We show that the regret terms admit a telescoping structure, which yields a sharp bound on their cumulative contribution.

Third, worker unavailability creates a richer decision space than in standard bandit models. A firm may decide to hire an incoming worker before that worker is available, while the corresponding removal decision need not be made until the incoming worker can actually join the workforce. The effects of these decisions appear through \emph{intermediate workforce} configurations, a feature unique to our model (see Remark~\ref{ex:intermediate-set}). For analytical tractability, we introduce a pairing formulation in which the firm commits the incoming worker with an outgoing worker. Within this formulation, we propose a simple Rank-Matching rule that pairs outgoing and incoming workers according to their lower-confidence bounds (LCBs). We show that this commitment policy achieves the best possible leading-order performance bound. Moreover, among pairing rules, Rank-Matching minimizes the variability of a surrogate measure of intermediate workforce production.

The work closest to ours is \cite{agrawal1990switching}, but the differences are substantial. Most importantly, the model in \cite{agrawal1990switching} does not include arm unavailability, a central aspect of our setting. This additional model detail materially affects both policy design and regret analysis. This leads to several important algorithmic differences. First, their switching schedule is fixed in advance at time 0 and is deterministic, whereas our policy uses an adaptive schedule based on realized observations. Second, in each iteration of their algorithm, \(m-1\) arms are selected according to empirical means and only the \(m\)th arm is chosen optimistically (using UCB), whereas our policy selects the target workforce by solving a constrained optimization problem with a fully optimistic objective. Third, \cite{agrawal1990switching} relies on a round-robin exploration mechanism, whereas our approach uses optimism throughout. 

The rest of the paper is organized as follows. The next section reviews the related literature. Section~\ref{formulation} presents the problem formulation, and Sections~\ref{sec:alg-desc} and \ref{theoretical} develop the learning-based {\DRUCB} policy and establish its theoretical performance guarantees. Section~\ref{numerical} reports the results of the numerical study. Section~\ref{sec:conclusion} concludes with directions for future research. All detailed technical proofs are provided in the online appendices.

\section{Literature Review}
The present work builds on two bodies of literature: dynamic workforce management and multi-armed bandit learning. We therefore review each separately below, focusing only on the most relevant works.

\vspace{-.1in}
\subsection{Learning-Based Hiring and Workforce Management}
Hiring and workforce allocation are core operational activities that have inspired extensive theoretical research. A classical starting point is the secretary problem \citep{dynkin1963optimal, ferguson1989who}, which models sequential hiring under uncertainty, and its extensions \citep{kleinberg2005multiple, disser2020hiring, kesselheim2024hiring}. While these models capture fundamental aspects of information arrival and timing, they assume that candidate quality is immediately observable, and that decisions are permanent once made.

Other streams of work examine the process of making and managing job offers when candidate acceptance is uncertain but quality distributions are known \citep{purohit2019hiring, epstein2024selection, gans2002managing, du2024sequential} or balancing the amount of permanent and contingent labor \citep{king2016dynamic, dong2020managing, berenguer2024managing, lobel2024frontiers}.

Complementary research explores worker development and learning over time, modeling human-capital accumulation with Bayesian updating of performance for full-time employees \citep{arlotto2014optimal}. These frameworks capture the dynamics of learning and retention but are less suited to short-term or high-churn labor environments, where workers are interchangeable and firm-side learning about quality dominates.

Finally, recent work has applied multi-armed bandit ideas to hiring and related worker selection settings, including models of adaptive interviewing \citep{schumann2019making}, and repeated recruitment and task assignment in crowdsourcing systems \citep{gao2020combinatorial, rangi2018multi}. Building on this emerging stream, we focus on temporary and platform-based hiring environments where screening is limited and performance information is revealed only through repeated short-term engagements. 

\vspace{-.1in}

\subsection{Multi-Armed Bandit Learning}\label{sec:bandit-related-work}

Since the seminal work of \cite{lai1985asymptotically}, the study of {multi-armed bandit} (MAB) problems has expanded across many domains \citep{bouneffouf2020survey}. Within this literature, we highlight four strands most relevant to our setting: (i) \emph{combinatorial semi-bandits}, which generalize classical bandits to settings where the decision maker selects sets of actions rather than a single one; (ii) \emph{bandits with replacement costs}, which model the frictions and penalties associated with changing actions over time; (iii) \emph{bandits with delays}, which account for the delays in actions, rewards, and/or feedback that occur in practice; and (iv) \emph{bandits with policy-dependent unavailability}, where the set of available arms depends on the policy's past decisions. We review representative works from each of these four strands below. A table highlighting similarities and differences among related bandit models can be found in Appendix \ref{sec:related_works_table}.

{\bf Combinatorial semi-bandits.}
The combinatorial semi-bandit model \cite{cesa2012combinatorial, chen2013combinatorial} extends the classical multi-armed bandit framework by allowing the decision maker to select subsets of arms at each round rather than a single one. These subsets represent feasible combinations of actions, often constrained by combinatorial structure. For instance, in our hiring model, each decision corresponds to choosing \(m\) workers from a pool of \(k\), which can be viewed as selecting an \(m\)-element subset of arms.

For stochastic, bounded rewards, regret bounds that scale as \(\tilde{O}(\sqrt{mkT})\) can be achieved with a UCB-type algorithm \citep{kveton2015tight}. Throughout the paper, \(O(\cdot)\) captures leading-order dependence, with \(\tilde{O}(\cdot)\) used to suppress logarithmic factors. The \(m\)-play bandit is a special case of the combinatorial semi-bandit model. However, our model incorporates \emph{replacement costs} and \emph{delayed replacements}, which are not included in these works. 

{\bf Bandits with replacement costs.}
In many operational decision problems, changing actions is costly. In the single-play, stochastic-reward setting, \cite{agrawal1988asymptswithcing} established an asymptotically optimal, instance-dependent regret bound of order \(O(\log T / \Delta)\), where \(\Delta\) denotes the suboptimality gap. Subsequent research extended this line of work to more general cost and reward models, including nonlinear switching penalties \citep{guha2009multi}, adversarial rewards \citep{cesa2013online}, and ``best-of-both-worlds" algorithms that adapt to either stochastic or adversarial regimes \citep{dekel2014bandits, rouyer2021algorithm, amir2022better}. 

A unifying algorithmic principle across these works is the use of \emph{block-based strategies}. Time is partitioned into blocks of increasing length, and a single arm is played for the duration of each block. By carefully choosing the growth rate of block lengths, replacement costs can be controlled so that they remain asymptotically dominated by sampling regret. This design shares close connections with the literature on \emph{batched bandits}, where the learner is constrained to update policies only a limited number of times \citep{agrawal1990switching, gao2019batched, esfandiari2021regret, jin2021almost}. However, in our setting, delayed actions (which we refer to as replacements) decouple update decisions from the accumulation of effective observations. Because replacements complete asynchronously, the number of periods required for a newly selected workforce to generate a desired amount of data is not known at the time the update is initiated. As a result, policies based on fixed, predetermined block lengths may fail to collect the intended amount of information. 

{\bf Bandits with delays.}
Delayed feedback models consider settings in which actions take effect immediately, but the resulting rewards are observed only after a delay. Feedback delays may be stochastic, bounded, or adversarial, and primarily complicate inference by slowing the rate at which information becomes available. \cite{chick2017bayesian} study feedback delays in a Bayesian adaptive experimentation setting, where the goal is to reduce the number of costly samples needed to reach a conclusion when observations arrive after a delay. By contrast, we consider a frequentist regret-minimization framework in which the firm operates over the entire horizon and incurs costs from switching between actions and choosing sub-optimal actions during the learning process. As a result, the objectives, algorithms, and analytical techniques differ substantially. Nevertheless, this line of work, together with subsequent applications to clinical trials \citep{forster2021cost, alban2023value}, highlights the practical importance of delayed observations in sequential decision problems. In the regret-minimization literature, delayed feedback has been studied beginning with \cite{joulani2013online}. In \cite{van2023unified}, an adversarial combinatorial semi-bandit with feedback delays is studied. To the best of our knowledge, stochastic combinatorial semi-bandits with feedback delays have not been studied. More importantly, there are subtle but meaningful differences between delayed feedback and the random delayed replacement model presented here. Delayed feedback does not restrict the learner's ability to enact decisions at each period: control remains instantaneous, and past actions do not constrain future choices.

By contrast, in the model studied here, when a worker is selected for hiring, that worker may not be available, leading to delays in \emph{actions}, not just feedback. This distinction becomes particularly consequential in the multi-play setting. Under delayed feedback, the learner always implements exactly the action set it selects; only the observation arrives later. Here, when multiple replacements are in progress, the realized active set can temporarily differ from the intended set because some replacements may complete earlier than others, producing intermediate workforces (see Remark \ref{ex:intermediate-set}). This phenomenon does not arise in delayed-feedback models.

{\bf Bandits with policy-dependent action unavailability.}
Another related class of models studies action availability constraints induced by past plays. The most similar are blocking bandits \citep{basu2019blocking, atsidakou2021combinatorial}. In this setting, selecting an action renders it temporarily unavailable, forcing the learner to diversify or substitute among remaining options. While these models and ours both feature policy-dependent action availability, they operate in opposite directions: blocking prevents consecutive selections, whereas the random availability studied here prevents frequent change by delaying the realization of replacements. Consequently, although both models feature policy-dependent action availability, the learning challenges they induce are fundamentally different from those in our setting, and results for blocking bandits do not directly apply.

\section{Problem Formulation}
\label{formulation}
In this section, we first formulate the contingent workforce hiring problem and then introduce our modeling approach, the \emph{delayed replacement} bandit model.

\subsection{Problem Statement: Contingent Hiring}
We consider a firm that requires \(m\) workers for its operations in each period \(t=1, 2,\ldots\). The firm hires the workforce from a pool of \(k \ge m\) workers provided by a staffing agency or platform. This pool is assumed to be of moderate size, as is typical in the contingent staffing setting. A worker from the pool, when hired in a period, can generate production throughput that follows a probability distribution unknown to the firm. We normalize the unit value (e.g., selling price) of the production output to 1; hence, this production output is also the revenue it generates. There are two distinct features in the model. First, workers outside the current workforce are not always available for booking. Second, workforce changes are costly: each time the firm replaces an employed worker, it incurs a fixed replacement cost. The firm's objective is to maximize the expected cumulative profit over any \(T\) periods.

\subsection{Delayed Replacement Bandit Model}
We formulate the contingent worker hiring problem as a bandit problem with action delays and costly replacement. It captures the sequential nature of the staffing problem when worker replacement is costly, and a worker who is not currently hired by the firm may not be available for immediate employment.

In the delayed replacement bandit model, there are \(k \ge 2\) arms, denoted by \(i \in [k] := \{1,2,\dots,k\}\) (hereafter, \(``:="\) means ``defined as''). In each period, \(m<k\) distinct arms are played. When arm \(i\) is played in a period \(t\), it generates production output \(X_t^{(i)}\), which is drawn from a distribution with bounded support, taken without loss of generality as \([0,1]\). The firm has no information about these distributions beyond what can be inferred from observed production data. For clarity of exposition, we order the arms by their mean output per period \(\mu_1 \geq \mu_2 \geq \cdots \geq \mu_k\), with ties broken arbitrarily, but consistently. We emphasize that this ordering is unknown to the firm, which must learn the best arms through repeated plays and observations. For each \(t = 1, \dots, T\), we let \(A_t\) denote the set of \(m\) arms played in the period.

For easy connection with the contingent hiring problem, in what follows we refer to the arms as workers, and call the set of workers \(A_t\) the \emph{active workforce} in period \(t\).

\vspace{-.1in}
\subsection{Worker Availability, Replacement Delays, and Replacement Frictions}
We model worker unavailability as a delay between a hiring decision and the worker's first available period, and allow this time to follow any distribution with finite mean. This general formulation accommodates a range of availability processes, as well as firm-side delays. A motivating special case is one in which, in each period, a worker $i$ who is not currently employed by the firm is available for hire in the next period independently with probability \(p_i\). In that case, the induced replacement delays are geometrically distributed.

The sequence of events in each period \(t\) proceeds as follows: First, the firm may initiate a replacement to hire a worker, say \(j\), which costs \(c\). Second, the environment draws a delay that determines when worker \(j\) will become available (the firm may or may not observe the realized delay at the time of initiating the replacement). Third, all pending hires whose delays have elapsed become available. Fourth, the firm observes these newly available workers and decides which current workers to replace, forming the new active workforce \(A_t\). We assume these four steps occur at the beginning of the period. Then, fifth, during period \(t\), the active workforce generates output, \(\sum_{i\in A_t} X_t^{(i)}\). Next, we formalize this process and introduce some necessary notation for the analysis.

{\bf Worker unavailability and delayed hiring.}
When the firm initiates a hire of worker \(j\), it need not decide which current worker to remove when \(j\) becomes available. For the remainder of the paper, we impose the simplifying restriction that the firm chooses the outgoing worker at the time the hire is initiated. We show that this restriction can be used to obtain a tractable policy without meaningfully weakening performance. Section~\ref{theoretical} shows that optimal leading-order performance remains attainable, while Section~\ref{numerical} shows empirically that committing to the outgoing worker at hire time does not degrade empirical performance.

Thus, we model each delayed hiring as a pair \((i,j)\), where \(i\) is the outgoing worker and \(j\) is the incoming worker. The delay associated with replacement \((i,j)\) initiated in period \(t\) is modeled as an independent, nonnegative-integer-valued random variable \(\omega_{i,j}^{(t)}\). We assume that \(\bE{\omega_{i,j}^{(t)}} \le \bar{\omega} < \infty\) for all \(i,j\) and \(t\), where \(\bar{\omega}\) is known to the firm. The event \(\omega_{i,j}^{(t)}=0\) corresponds to the no-delay case in which the replacement is realized immediately. We denote the period in which this replacement becomes effective by \(C_{i,j}^{(t)} := t + \omega_{i,j}^{(t)}\).

{\bf Replacement decisions.}
At the beginning of each period \(t\), the firm determines a set of replacements, denoted by 
\[
    R_t := \{(i,j) \;:\; \text{replacement of worker \(i\) with \(j\)}\}.
\]
We refer to the workforce implied by the replacements as the \emph{target workforce} and denote it
\[
    U_t:= A_t \setminus \{i: (i,\cdot) \in R_t\} \cup \{j: (\cdot, j) \in R_t\}
\]
Each initiated replacement proceeds through the delay process and takes effect in the period determined by its completion time. We define
\[
    \mathcal{E}_t := \{(i,j) \;:\; C_{i,j}^{(s)} = t \text{ for some initiation period } s \le t\}
\]
as the set of replacements that complete in period \(t\). 
\begin{remark}[Intermediate workforce]\label{ex:intermediate-set}
    When multiple replacements are initiated, their completion times may differ. Consequently, the sequence of active workforces \(\{A_t\}\) depends on the order in which individual replacements complete. For example, with \(m=3\), suppose \(A_t=\{1,2, 5\}\) and the firm intends to change its assignments to \(\{3,4, 5\}\) by replacing \((1\to 3)\) and \((2\to 4)\) in period \(t\). If \((1\to 3)\) completes first while \((2\to 4)\) remains pending, then in the next period the active workforce is \(\{2,3,5\}\), an intermediate roster that may perform worse than both \(\{1,2, 5\}\) and \(\{3,4, 5\}\) if, for example, \(\mu_5>\mu_1 > \mu_4 > \mu_2 > \mu_3\). The order and structure of pending replacements therefore affects production output during transitions, not only after the transition.
\end{remark}

To ensure a logically consistent evolution of the workforce state, we impose that each worker can be involved in at most one pending replacement at any given time. Allowing overlapping replacements would create conflicting commitments, for example when the same worker is simultaneously scheduled to exit multiple times or when the same incoming worker is assigned to multiple future replacements. This restriction yields a well-defined replacement pipeline in which staffing changes resolve sequentially. Formally, this requirement is captured by the following coherence constraints: for any period \(\tau\), any currently employed worker \(i \in A_\tau\),
\begin{equation}\label{eq:coherence-1}
    \left|\bigcup_{s=0}^{\tau}\{(i,\cdot)\in R_s \mid C_{i,\cdot}^{(s)}>\tau\}\right|\;\leq\; 1,
\end{equation}
and any worker \(j \notin A_\tau\),
\begin{equation}\label{eq:coherence-2}
    \left|\bigcup_{s=0}^{\tau}\{(\cdot,j)\in R_s \mid C_{\cdot,j}^{(s)}>\tau\}\right|
    \;\leq\; 1.
\end{equation}
We denote by \(\mathcal{R}_t\) the set of feasible replacements, that is, maps from \(A_t\) to \([k]\setminus A_t\) satisfying \eqref{eq:coherence-1} and \eqref{eq:coherence-2}. 

{\bf Replacement costs.} Each initiated replacement incurs a fixed replacement cost \(c > 0\). For convenience, we assume that the cost is incurred when the replacement is initiated. However, the analysis is unchanged if this cost is incurred at the completion time of the delay, or generally, at any point between the initiation and completion of the delay. The total replacement cost in period \(t\) is therefore \(c \cdot |R_t|\).

\vspace{-.1in}

\subsection{Objective}
For any replacement \((i,j)\) initiated in period \(s\) with completion time \(C_{i,j}^{(s)} = t\), worker \(i\) leaves the roster and worker \(j\) enters it in period \(t\). The workforce update is therefore
\[
    A_t
    = \big( A_{t-1} \setminus \{  i : (i,j) \in \mathcal{E}_t  \} \big)
      \;\cup\;
      \{  j : (i,j) \in \mathcal{E}_t  \}.
\]
A worker who exits in this manner can return if selected in a future replacement action, in which case they are again subject to the usual replacement cost and delay process.

Once the active workforce for period \(t\) is determined, the firm observes the output produced by each of the active workers. The total output of the collective workforce in period \(t\) is then
\[
    X_t = \sum_{i \in A_t} X_t^{(i)},
\]
where each \(X_t^{(i)}\) is independent. These observations form the information the firm uses to update assessments of worker productivity and guide future staffing choices.

\begin{table}[t]
    \centering
    \small
    \setlength{\tabcolsep}{4pt}
    \begin{tabular*}{\linewidth}{@{\extracolsep{\fill}} l p{0.72\linewidth}}
    \toprule
    Symbol & Meaning \\
    \midrule
    \(A_t\) & Reward-generating active workforce in period \(t\), after all completions in \(\mathcal{E}_t\) are realized \\
    \(R_t\) & Replacement requests initiated at the start of period \(t\) \\
    \(\mathcal{E}_t\) & Accepted replacement requests whose completion time equals \(t\) \\
    \(U_t\) & Workforce implied by the period-\(t\) request set \(R_t\) \\
    \bottomrule
    \end{tabular*}
    \caption{Core timing and workforce notation.}
    \label{tab:timing-notation}
\end{table}

\begin{remark}[Timing example]\label{rem:timing-example}
Suppose the roster observed at the start of period \(t\) is \(\{1,4\}\) and the firm requests the replacement \((4,3)\). If the realized delay is \(\omega_{4,3}^{(t)}=2\), then \(\mathcal{E}_t=\mathcal{E}_{t+1}=\emptyset\), so periods \(t\) and \(t+1\) still generate rewards from workers \(\{1,4\}\). If the request completes in period \(t+2\), then \(\mathcal{E}_{t+2}=\{(4,3)\}\) and the reward-generating workforce becomes \(A_{t+2}=\{1,3\}\).
\end{remark}

We evaluate a policy over a time horizon of \(T\) periods, where the horizon is known to the firm. As a benchmark, consider an idealized firm that faces no uncertainty. Let \(A^* := \{1,\dots,m\}\) denote the set of the \(m\) highest-mean workers. If these workers were known from the outset, the firm's expected output over \(T\) periods would be \(T \cdot \mu(A^*)\) where \(\mu(A^*) := \sum_{i \in A^*} \mu_i\).

In practice, the firm does not know the \(\mu_i\) and must infer them from the output of the workers who happen to be active at each point in time. Delays and replacement costs prevent the firm from immediately achieving the benchmark. To quantify this gap, we use the standard criterion of expected cumulative (pseudo-)regret, hereafter simply called \emph{regret}, which captures both the production loss due to learning and the costs incurred from workforce adjustments.

Formally, the regret after \(T\) periods is
\begin{eqnarray}\label{def:regret}
    \Regret(T):=
    \underbrace{T \cdot \mu(A^*) - \mathbb{E}\!\left[\sum_{t=1}^T X_t\right]}_{\text{Loss due to learning}}
    + \underbrace{\mathbb{E}\!\left[\sum_{t=1}^T c \cdot |R_t|\right]}_{\text{Cost of replacements}}
\end{eqnarray}
The expectation is taken over the policy's internal randomness and the workers' productivity realizations. A policy with smaller regret is therefore closer, in expectation, to the profit of an oracle that knows all worker qualities a priori. Note that an oracle policy incurs at most \(O(1)\) replacement cost since it makes at most \(m\) replacements to employ workforce \(A^*\).

\section{DR-UCB: A Learning-Based Policy}\label{sec:alg-desc}
We develop a learning-based hiring policy for the delayed replacement setting, which we call \textsc{DelayedReplace-UCB} (\DRUCB), in which staffing changes take effect after a delay and each replacement incurs a cost. Operationally, the policy has three decision modules.

First, the policy must determine \emph{when} to initiate a replacement. The replacement schedule controls a fundamental tradeoff: frequent updates accelerate learning but incur higher cumulative replacement costs, while infrequent updates reduce costs but slow adaptation. Second, the policy must decide \emph{which} workers to employ. This is a combinatorial exploration–exploitation problem over subsets of size \(m\), where the firm balances learning about uncertain workers against exploiting those currently believed to be the best performers. Third, once a target workforce has been chosen, the policy must translate that target into concrete replacement requests. Because replacements are delayed and realized over time, this translation affects the sequence of intermediate workforces encountered during the transition. We encode these decisions through three components:
\[\textsc{TimeToReplace}(\cdot),\quad \textsc{ChooseTarget}(\cdot),\quad \textsc{ConstructPairing}(\cdot),\]
corresponding to the replacement rule, selection rule, and pairing rule, respectively. All three components rely on estimates and confidence bounds of workers' productivity, which we introduce next.

\vspace{-.1in}
\subsection{Productivity Estimates and Confidence Bounds}
We begin by defining the statistical quantities that drive all subsequent decisions in the policy. To guide hiring decisions, the firm maintains an estimate of each worker's average production output based on the periods in which that worker was in the workforce. Let \(\bar{\mu}_i(t)\) denote the empirical mean output of worker \(i\) up to period \(t\), computed from all observations the firm has received while \(i\) was active. The corresponding number of observations is
\[
    T_i(t) := \sum_{s=1}^t \mathbb{I}\{  i \in A_s  \},
\]
which simply counts the total number of periods worker \(i\) has been hired. Because these estimates are based on realized production quantities, the policy constructs confidence bounds that capture the plausible range of each worker's average productivity in a period. For each worker \(i\), we define
\[
    \UCB_i(t) = \bar{\mu}_i(t) + \sqrt{\tfrac{\log T}{T_i(t)}},\qquad 
    \LCB_i(t) = \bar{\mu}_i(t) - \sqrt{\tfrac{\log T}{T_i(t)}}.
\]
The upper confidence bound \(\UCB_i(t)\) is an upper bound on worker \(i\)'s true productivity per period with high probability, while the lower bound \(\LCB_i(t)\) provides a high-probability lower bound. In the hiring setting, firms often have some initial information on worker output, either provided by a staffing agency or collected from historical encounters, which can be used to initialize the firm's confidence bounds.Because the workforce evolves through replacements, we define the ``optimistic gain'' of a replacement \((i,j)\) as
\[\delta_{i,j}(t):= \UCB_j(t) - \LCB_i(t).\]
This represents a high-probability bound on the increased output that could be achieved by replacing worker \(i\) with worker \(j\). These quantities form the primitives used by the core components of the policy, which we now use to describe the overall structure of {\DRUCB}.

\vspace{-.1in}

\subsection{Algorithm Design and Learning Cycles}
We now present the high-level structure of {\DRUCB}. The policy operates sequentially over periods \(t=1,2,\dots,\). At the start of each period \(t\), the firm observes the carried-over workforce state and its current estimates based on feedback through period \(t-1\), and then decides whether to initiate a new round of replacements. If a replacement epoch is triggered, the policy selects a new \emph{target workforce} and constructs a set of replacements that moves the current workforce toward that target. The environment then samples delays for accepted requests, realizes all replacements whose completion time equals \(t\), and generates the period-\(t\) rewards from the resulting active workforce. Algorithm~\ref{alg:DelayedReplace-UCB} summarizes this procedure at a high level. For readability, Algorithm~\ref{alg:DelayedReplace-UCB} omits the bookkeeping variables maintained internally by its subroutines. These are introduced as part of the internal state in the discussion of each subroutine.

\begin{algorithm}
    \caption{DelayedReplace-UCB (\DRUCB)}
    \label{alg:DelayedReplace-UCB}
    \begin{algorithmic}[1]
    \Require Initial workforce \(A_0\in \{A\subset [k]: |A|=m\}\) 
    \Require Replacement parameter \(\gamma>0\)
    \Require Selection rule \textsc{ChooseTarget}(\(\cdot\))
    \Require Replacement rule \textsc{TimeToReplace}(\(\cdot\))
    \Require Pairing rule \textsc{ConstructPairing}(\(\cdot\))
    \State \textbf{initialize:} Target workforce \(U_1\gets A_0\)
    \State \textbf{initialize:} Replacement counter \(\ell\gets 1\) 
    \State \textbf{initialize:} Count vector \(\mathbf{n}_1=(T_1(1), T_2(1),\dots,T_k(1))\)
    \State \textbf{initialize:} Empirical mean vector \(\boldsymbol{\mu}_1=(\bar{\mu}_1(1), \bar{\mu}_2(1),\dots,\bar{\mu}_k(1))\)
    \For{time \(t = 1,2,\dots,T\)}
        \State Observe the carried-over roster state (\(A_{t-1}\)) and current statistics \((\mathbf{n}_{t},\boldsymbol{\mu}_{t})\)
        \If{\textsc{TimeToReplace}(\(A_{t-1}, U_{\ell},\mathbf{n}_t, \gamma\))}
            \State Update replacement counter: \(\ell \gets \ell+1\)
            \State Pick target: \(U_\ell \gets \textsc{ChooseTarget}(A_{t-1}, \mathbf{n}_t, \boldsymbol{\mu}_t)\)
            \State Construct replacement pairs: \(R_t \gets \textsc{ConstructPairing}(A_{t-1},U_\ell, \mathbf{n}_t, \boldsymbol{\mu}_t)\)
            \State Initiate replacements \(R_t\)
        \EndIf
        \State The environment samples delays for accepted requests and realizes \(\mathcal{E}_t\)
        \State The resulting active workforce generates rewards \(X_t^{(i)}\) for each active \(i\)
        \State Update the statistics used in the next decision from the realized period-\(t\) feedback
        \State Update the active workforce: \( A_{t}\gets\big( A_{t-1} \setminus \{  i : (i,j) \in \mathcal{E}_t  \} \big) \cup \{  j : (i,j) \in \mathcal{E}_t  \} \)
    \EndFor
    \end{algorithmic}
\end{algorithm}

Algorithm~\ref{alg:DelayedReplace-UCB} decomposes the policy design problem into an adaptive replacement rule, a constrained selection rule, and a pairing rule. The replacement and selection components are the parts used in the main regret proof. The pairing step is included because delayed replacements create intermediate workforces, but its analysis is auxiliary and is used to explain transition quality rather than the leading-order regret rate.

Our algorithm design naturally induces a cyclic structure on the timeline. This is shown in Figure~\ref{fig:intro-timeline}. Indexed by cycle $\ell=1, 2, \ldots$, each cycle $\ell$ begins with the determination of a new target workforce. Due to replacement delays, the workforce passes through a transition phase, during which the active workforce differs from the target workforce. Second, the algorithm enters an exploration phase where the firm gathers productivity data from its target workforce. At the end of the exploration phase, which is also the end of cycle $\ell$, a new target workforce is determined and the next cycle $\ell+1$ starts. Then the process repeats itself. 

\begin{remark}[Worked learning-cycle example]
    Consider \(m=2\) and suppose the active workforce at time \(t\) is \(A_t=\{1,4\}\). If worker \(4\) is the least-observed worker in the current target workforce, then \TimeToReplace~waits until worker \(4\) has accumulated a sufficient amount of additional observations. Once that threshold is reached, \ChooseTarget~may replace worker \(4\) with an uncertain outside worker \(3\) only if the optimistic gain \(\delta_{4,3}(t)\) is large enough to justify the replacement cost over the remaining horizon. If this replacement is requested and completes two periods later, then the active workforce remains \(\{1,4\}\) during the delay, transitions to \(\{1,3\}\) once the request completes, and the algorithm then holds \(\{1,3\}\) long enough to collect the next block of observations. This illustrates the roles of the replacement rule, the selection rule, and the distinction between the active and target workforces.
\end{remark}

The following subsections define each component in detail, beginning with the replacement rule. Because the selection screen uses the rank-matching operator, we briefly discuss the pairing rule before presenting the selection rule.

\vspace{-.1in}

\subsection{Adaptive Replacement Rule}
The replacement rule \TimeToReplace~governs when 
to initiate a new round of replacements. This rule aims to balance two competing objectives: replacing too frequently accelerates adaptation but incurs excessive replacement costs, while replacing too infrequently reduces costs but slows learning. A natural way to manage this tradeoff is to require the policy to continue collecting data from the current workforce until a prescribed observation threshold is reached, and to permit further replacements only after that point. If these thresholds are chosen appropriately, the policy can remain adaptive while limiting the frequency of workforce updates.

Because the target workforce remains fixed between replacement decisions, it is convenient in this section to index it by the replacement counter and write \(U_\ell\) for the workforce targeted during the \(\ell\)-th learning cycle. Let \(d_\ell\) denote the period in which that target workforce was most recently selected. Because the productivity of the current workforce is only as well understood as its least-observed worker, it is natural to define the replacement threshold in terms of the minimum number of observations among workers currently targeted for employment. Motivated by this principle, we regulate the pace of workforce updates through an observation-based threshold \(N(\ell)\), which specifies how much additional information must be collected before the next replacement epoch is allowed. Specifically, we set
\begin{equation}\label{eq:block-threshold}
    N(\ell)
    :=
    \min_{j \in U_\ell}
    \left\lceil 2\sqrt{\gamma\, T_j(d_\ell)} + \gamma \right\rceil,
\end{equation}
where \(\gamma > 0\) is a parameter controlling how aggressively the policy updates the workforce, and \(\lceil x \rceil\) denotes the smallest integer greater than or equal to \(x\).

The subroutine \TimeToReplace~is described in Algorithm \ref{alg:TimeToReplace}. It implements this rule by waiting until the least-observed worker in that workforce has accumulated at least \(N(\ell)\) additional observations since the last replacement epoch time \(d_\ell\).

\begin{algorithm}
    \caption{TimeToReplace}
    \label{alg:TimeToReplace}
    \begin{algorithmic}[1]
        \Require Current active workforce \(A_t\)
        \Require Current target workforce \(U_{\ell}\)
        \Require Current count vector \(\mathbf{n}_t=(T_1(t),\dots,T_k(t))\)
        \Require Replacement parameter \(\gamma>0\)
        \State \textbf{internal state:} Most recent replacement epoch time \(d_\ell\)
        \State \textbf{internal state:} Counts at last replacement epoch time \(\mathbf{n}_{d_{\ell}}\)
        \State Compute least-sampled worker:
        \[
            i^- \in \argmin_{j \in U_\ell} T_j(d_\ell)
        \]
        \State Compute threshold:
        \[
            N(\ell)\gets \left\lceil 2\sqrt{\gamma\, T_{i^-}(d_\ell)} + \gamma \right\rceil
        \]
        \If{\(T_{i^{-}}(t)\geq T_{i^{-}}(d_{\ell}) + N(\ell)\) and \(A_t=U_\ell\)}
            \State Update internal state: \(d_{\ell}\gets t\)
            \State Update internal state: \(\mathbf{n}_{d_{\ell}}\gets \mathbf{n}_t\)
            \State \Return True \Comment{Time to replace!}
        \Else
            \State \Return False \Comment{Continue waiting}
        \EndIf
    \end{algorithmic}
\end{algorithm}

We make three remarks. First, Algorithm~\ref{alg:TimeToReplace} recomputes the least-sampled worker and the threshold each time the subroutine is called. This can be implemented more efficiently, but we present it this way for clarity. Second, a replacement epoch is permitted only when the current target workforce \(U_\ell\) is fully active. This ensures that replacements satisfy the coherence constraints \eqref{eq:coherence-1} and \eqref{eq:coherence-2}. Third, the threshold \(N(\ell)\) is defined in terms of accumulated employment periods rather than calendar time. Because replacements are completed asynchronously, the number of calendar periods required to collect these observations is not known in advance. The count-based rule guarantees that sufficient data are collected even when delays postpone the onset of learning for newly hired workers. This differs from the fixed calendar-time schedules in \cite{agrawal1990switching, gao2019batched, perchet2016batched}, where update opportunities are fixed in advance in calendar time. As shown in Section~\ref{numerical}, this adaptivity is a key source of robustness in delayed environments.

\begin{figure}[H]
    \centering
    \begin{tikzpicture}[xscale=1, yscale=1]
    
    \draw[thick, ->] (0,0) -- (14,0) node[right] {\(t\)};
    
    
    \draw[fill=blue!15] (0,0.3) rectangle (4,0.8);
    \node at (2,0.55) {\(A_t=U_{\ell-1}\)};
    
    \draw[fill=red!15] (4,0.3) rectangle (6,0.8);
    \node at (5,0.55) {\(A_t\neq U_\ell\)};
    \draw[dashed, thick] (4,-0.1) -- (4,1.1);
    \node[draw, dashed] at (4,1.4) {\textsc{TimeToReplace}(\(\cdot\)) = \textsc{True}};
    
    \draw[fill=blue!15] (6,0.3) rectangle (10,0.8);
    \node at (8,0.55) {\(A_t=U_{\ell}\)};
    
    \draw[fill=red!15] (10,0.3) rectangle (12,0.8);
    \node at (11,0.55) {\(A_t\neq U_{\ell+1}\)};
    \draw[dashed, thick] (10,-0.1) -- (10,1.1);
    \node[draw, dashed] at (10,1.4) {\textsc{TimeToReplace}(\(\cdot\)) = \textsc{True}};
    
    \draw[fill=blue!15] (12,0.3) rectangle (14,0.8);
    \node at (13,0.55) {\(A_t=U_{\ell+1}\)};
    
    \end{tikzpicture}
    \caption{Timeline of Algorithm \ref{alg:DelayedReplace-UCB}}
    \label{fig:intro-timeline}
\end{figure}
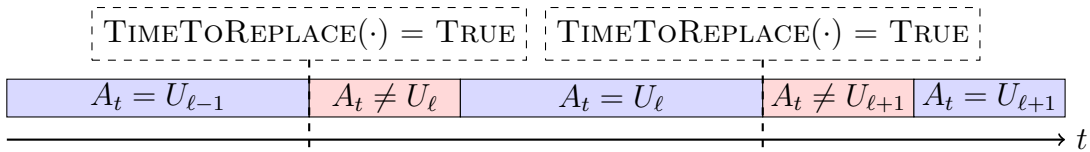

\vspace{-.2in}

The learning cycle with a two-phase structure is depicted in Figure~\ref{fig:intro-timeline}. Each cycle $\ell$ consists of two phases. At the start of the cycle, \TimeToReplace~returns true, and the policy selects a new target workforce \(U_\ell\). The transition phase is highlighted in red, during which the process waits for the completion of replacement requests. Thus, during the transition phase, the active workforce differs from \(U_\ell\), and the firm operates with intermediate workforces. The composition of these workforces depends on how incoming and outgoing workers are paired. Once all the replacements are completed, the algorithm enters an exploration phase with active workforce \(U_\ell\), which is highlighted in blue. The policy remains in this phase for $N(\ell)$ periods specified by \eqref{eq:block-threshold}; then it triggers the start of the next cycle $\ell+1$ with a new target workforce \(U_{\ell+1}\). It is worth noting that some phases may be empty.

\vspace{-.1in}
\subsection{Rank-Matching Pairing Rule}
Once a target workforce \(U_{\ell}\) has been determined at the start of cycle $\ell$, the policy must convert it into concrete replacement requests. Given a target workforce \(U_{\ell}\) and an active workforce \(A_t\), we use a simple \emph{rank-matching} rule designed to make intermediate workforces performance less variable during the transition. The key idea is the following bijection.

Let \(|\cdot|\) denote cardinality and \((i_1,\dots,i_r)\) the elements of \(A_{t}\setminus U_\ell\), with \(r := |A_{t}\setminus U_\ell| = |U_\ell\setminus A_{t}|\) ordered by their lower confidence bounds (i.e.  \(\LCB_{i_1}(t) \ge \LCB_{i_2}(t) \ge \cdots \ge \LCB_{i_r}(t)\)). Similarly, let \((j_1,\dots,j_r)\) denote the elements of \(U_\ell\setminus A_{t}\), ordered so that \(\LCB_{j_1}(t) \ge \LCB_{j_2}(t) \ge \cdots \ge \LCB_{j_r}(t)\). We define the \emph{rank-matching} replacement bijection \(\pi^{\mathrm R}\) by matching workers according to their ranks, i.e., 
\begin{equation}\label{eq:rm-bijection}
    \pi^{\mathrm R}(i_s) = j_s, \qquad s=1,\dots,r.
\end{equation}
That is, workers to be removed and workers to be added are paired according to their relative lower-confidence rankings. The subroutine \ConstructPairing~constructs this pairing between the target and active workforce.

\begin{algorithm}
    \caption{ConstructPairing}
    \label{alg:ConstructPairing}
    \begin{algorithmic}[1]
        \Require Current active workforce \(A_t\)
        \Require Target workforce \(U_{\ell}\)
        \Require Current count vector \(\mathbf{n}_t=(T_1(t),\dots,T_k(t))\)
        \Require Current empirical mean vector \(\boldsymbol{\mu}_t=(\bar{\mu}_1(t),\dots,\bar{\mu}_k(t))\)
        \State Compute \(\LCB_i(t)\) for all workers in the active workforce \(A_t\) and target workforce \(U_{\ell}\).
        \State Use LCBs to construct the rank-matching bijection \(R_t = \{  (i,\pi^{\mathrm{R}}(i)) : i \in  A_{t}\setminus U_{\ell}  \}\).
        \State \Return  \(R_t\).
    \end{algorithmic}
\end{algorithm}

Because replacements are completed after heterogeneous delays, the pairing rule affects the sequence of intermediate workforces encountered during the transition phase. The rank-matching bijection is appealing because it satisfies several desired properties with respect to the regret during the transition phase. See Appendix \ref{sec:bijection_analysis}, in particular Propositions \ref{lem:iid-rank-matching-E} and \ref{lem:iid-rank-matching-Var}, for details. 

\subsection{Selection Rule}
We now turn to the second core component of the algorithm: the selection rule \ChooseTarget, which determines the target workforce \(U_\ell\). A natural starting point is an optimism-based rule that selects the subset with the largest aggregate upper confidence bound. In our setting, however, replacement costs and a known time horizon \(T\) require a refinement: near the end of the horizon, the remaining time may be too short for a replacement to recover its cost unless the potential improvement is sufficiently large. The subroutine \ChooseTarget~incorporates this constrained replacement criterion into an otherwise optimistic selection rule.

\begin{algorithm}
    \caption{ChooseTarget}
    \label{alg:ChooseTarget}
    \begin{algorithmic}[1]
        \Require Current active workforce \(A_t\)
        \Require Current count vector \(\mathbf{n}_t=(T_1(t),\dots,T_k(t))\)
        \Require Current empirical mean vector \(\boldsymbol{\mu}_t=(\bar{\mu}_1(t),\dots,\bar{\mu}_k(t))\) and their LCB and UCB values 
        \(\textbf{LCB}_t=(\LCB_1(t),\dots,\LCB_k(t))\), \(\textbf{UCB}_t=(\UCB_1(t),\dots,\UCB_k(t))\)
        \State \Return Target workforce \(U\) which solves the optimization problem:
        \begin{align}\label{eq:opt}
            \max_{U} \quad & \sum_{i\in U}\UCB_i(t)\\
            \text{subject to}\quad &\sum_{i\in A_t\setminus U}\delta_{i,\pi^\mathrm{R}(i)}(t)\geq c\cdot \frac{| A_t\setminus U|}{T-t}
        \end{align}
    \end{algorithmic}
\end{algorithm}
The objective in Algorithm~\ref{alg:ChooseTarget} is optimistic, favoring target workforces with large upper-confidence values. The constraint tempers this optimism by requiring that the new workforce can plausibly recover its replacement cost over the remaining horizon. If this constraint is dropped, the optimization reduces to the unconstrained optimistic rule that simply selects the \(m\) workers with the largest UCB values.

\begin{remark}[Computation of \textsc{ChooseTarget}]
    Because of the constraint in the selection problem, the standard greedy approach is not optimal. In Appendix \ref{sec:comp_of_selection_rule}, we describe a dynamic programming (DP)-based approach. We discuss the computational time for the DP algorithm in Section~\ref{numerical} and in Appendix~\ref{sec:runtime}.
\end{remark}

\section{Theoretical Guarantee}
\label{theoretical}
The following results show that adapting to delayed replacements need not come at the expense of learning efficiency. Before establishing our main result, we provide a lower bound for the multi-play bandit problem with delayed and costly replacements. The lower bounds are proved for the setting in which a firm may wait to decide which workers exit the workforce until the incoming workers become available.

\begin{proposition} [Minimax lower bound]
\label{888}
Consider the stochastic delayed-replacement bandit model of Section~\ref{formulation} where outgoing workers are decided when incoming workers become available. For any \(m\) and \(k\), and any horizon \(T > 0\), there exists a delayed and costly multi-play bandit instance such that the regret of any algorithm is
\begin{equation*}
    \Regret(T) \geq \min \left\{\Omega\left(\sqrt{mkT}\right), \Omega\left(mT\right) \right\}.
\end{equation*}
\end{proposition}

The proof appears in Appendix \ref{sec:lower_bounds}. We now present our main theoretical result. When equipped with a suitable choice of the parameter \(\gamma\), our hiring policy achieves the same dominant \(\sqrt{T}\)-dependence as the lower bound in Proposition \ref{888}, up to logarithmic factors and lower-order friction terms. Delays and replacement costs enter through additive terms that can be controlled via \(\gamma\), which governs how aggressively the policy is permitted to adjust the workforce. 

\begin{theorem}[Instance-independent regret bound]\label{thm:main}
    There exist universal positive constants $C_1$, $C_2$, and $C_3$, independent of \((k,m,T,c,\bar{\omega})\), such that for all \(k\ge 2\), the regret of DR-UCB satisfies 
    \begin{eqnarray}
    \label{099}
        \Regret(T) &\le&  C_1 \sqrt{m(k-m)T\log T} 
         + C_2(k-m)\sqrt{\gamma \polylog(T)} \\
        &&\qquad\qquad\qquad\qquad + C_3 k\left(\gamma + (\bar{\omega}+c)m\right)
        + (\bar{\omega}+c)m\sqrt{\frac{kT}{\gamma}}
        \nonumber
    \end{eqnarray}
    In particular, when \(\gamma = (\bar{\omega}+c)^2m\), there are universal positive constants \(C_4, C_5\), independent of the problem parameters, such that
  \begin{eqnarray}
      \Regret(T) \leq C_4 \sqrt{mkT\log T} 
        +\; C_5km(c+\bar{\omega})^2\sqrt{ \polylog(T)}.
        \label{088}
        \end{eqnarray}
    The exact constants are given in Appendix \ref{sec:thm-all-constants}.
\end{theorem}

{\bf Explanation of regret.} The regret bound has four distinct contributions that can be interpreted as follows. The term \(C_1 \sqrt{mkT\log T}\) has the same form as in frictionless models. It represents the intrinsic statistical difficulty of learning which workers perform well, independent of the operational constraints. The lower-order term, \(C_2\,(k-m)\sqrt{\gamma \polylog(T)}\), captures the additional learning cost induced by limiting the rate of replacements. Increasing the parameter \(\gamma\) lengthens the periods over which the workforce configuration is held fixed, reducing the frequency of updates but slowing the rate at which suboptimal workers can be identified and replaced. As a result, the learning regret grows with \(\sqrt{\gamma}\) and scales linearly with the number of suboptimal workers, \(k-m\). The remaining dependence on the horizon \(T\) enters only through polylogarithmic factors. The third term \(C_3\,k\left(\gamma + (\bar{\omega}+c)m\right)\) represents a horizon-independent ``start-up'' or ``burn-in'' cost that is standard in bandit learning algorithms. To make statistically meaningful decisions, the algorithm must obtain at least one observation from each worker; otherwise, it cannot distinguish unobserved workers from potentially optimal ones. (Therefore, if the firm has prior data points about some or all of the workers, it can actually speed up the learning process. Indeed, all our results hold true for the case with initial partial data on the workers in the pool.) The remaining term, \((\bar{\omega}+c)\,m\sqrt{{kT}/{\gamma}}\), isolates the penalty of delayed replacements and replacement costs over time. It scales like \(\sqrt{T}\) and decreases in \(\gamma\). More conservative updates (larger \(\gamma\)) reduce the cumulative impact of frictions by limiting how often the workforce is changed. In contrast, smaller \(\gamma\) makes the policy more agile but amplifies this friction-driven component. 

The second part of Theorem \ref{thm:main}, which chooses \(\gamma = (\bar{\omega}+c)^2 m\), shows that this trade-off can be resolved in favor of learning. Once \(\gamma\) is chosen large enough relative to the magnitude of the frictions, the \((\bar{\omega}+c)m\sqrt{kT/\gamma}\) term becomes \(\tilde{O}(\sqrt{kmT}\,)\), the same order as the frictionless lower bound, and the leading \(\sqrt{T}\) term of the bound becomes independent of \(c\) and \(\bar{\omega}\).

{\bf High replacement cost.} In the bound in Theorem \ref{thm:main}, the replacement cost $c$ is a constant independent of $T$. In situations where replacement cost is relatively high, we can model $c$ as a function of $T$, e.g., $c=\Theta(\log T)$ or $c=\Theta(T^{1/4})$. From (\ref{088}), we see that as long as the replacement cost grows at rate  $$c\;\le \; O\big((T/(mk))^{1/4}-\bar{\omega}\big)^+,$$ where $x^+=\max\{x, 0\}$ for any real $x$, then the regret of the {\DRUCB} policy is $\tilde O(\sqrt{mkT})$. 

{\bf Instance-dependent regret.} Along with our main result, we prove an asymptotic instance-dependent bound on the regret. Before presenting the result, we introduce the concept of consistent policies so that we may state the accompanying lower bound. 

We call a policy \emph{consistent} (sometimes ``uniformly good'') if, for all replacement bandit policies, \(\limsup_{T\to \infty}\Regret(T)/T^\alpha=0\) for all \(\alpha>0\) \citep{lattimore2020bandit}. In this setting, the following lower bound holds. Its proof is given in Appendix \ref{sec:lower_bounds}.

\begin{proposition}[Instance-dependent regret lower bound]\label{prop:asymptoptic-lower}
    For the stochastic delayed-replacement model of Section~\ref{formulation} where outgoing workers are decided when incoming workers become available, consider the instance where the firm staffs \(m\) identical positions and rewards are Bernoulli with means \(\mu_i\). Fix any \(\Delta \in (0,1/2)\), and let the Bernoulli means be
    \[\mu_i =\begin{cases}
        1/2, & i = 1,\dots,m,\\
        1/2-\Delta, & i = m+1,\dots,k.
    \end{cases}\]
    Then any consistent algorithm satisfies
    \[
        \liminf_{T\to\infty}\frac{\Regret(T)}{\log T}\geq \Omega\left(\frac{k-m}{\Delta}\right).
    \]
\end{proposition}

We are now ready to present the instance-dependent upper bound for the {\DRUCB} algorithm. As before, all constants \(C_i\) are available in Appendix \ref{sec:delta-all-constants}.
 
\begin{theorem}[Instance-dependent regret bound]\label{thm:delta-main}
    Suppose that there are \(m\) unique best arms (i.e., \(\mu_1\geq \cdots \geq \mu_m >\mu_{m+1}\geq \mu_k\)). Define \(\Delta_{i,j}:=\mu_i-\mu_j\). Then there exist universal constants \(C_1,C_2,C_3\), and \(C_4\), independent of problem parameters, such that when \(T>k\),
    \begin{align*}
        \Regret(T)&\leq C_1\sum_{j>m}\frac{1}{\Delta_{m,j}} \left( \log T+(\bar{\omega}+c)m\sqrt{\frac{k\log T}{\gamma }}\right)\\
        &\quad +~ C_2\sum_{j>m} \left[\sqrt{\gamma \log T\cdot \polylog\left(\frac{\Delta_{1,j}}{\Delta_{m,j}}\right)}+c\log\left(\frac{c}{\Delta_{m,j}}\right)\right]\\
        & \quad + ~C_3\sum_{j>m} \Delta_{1,j} \left[(\gamma + 1) + \frac{m}{T}\right]\\
        &\quad +~C_4(\bar{\omega}+c)mk
    \end{align*}
    Additionally, when \(\gamma=(\bar{\omega}+c)m\), there exist constants \(C_5, C_6\), independent of the problem parameters, such that
    \begin{align*}
        \Regret(T)&\leq C_5\sum_{j>m}\left[\frac{\log T}{\Delta_{m,j}}+\left(\polylog\left(\frac{\Delta_{1,j}}{\Delta_{m,j}}\right)+\frac{1}{\Delta_{m,j}}\right)\sqrt{(\bar{\omega}+c)mk\log T}\right]\ \\
        &\quad +~ C_6\sum_{j>m} \left[c\log\left(\frac{c}{\Delta_{m,j}}\right)+m\Delta_{1,j} \left(\bar{\omega}+c + \frac{1}{T}\right)\right].
    \end{align*}
\end{theorem}

This upper bound shows that the scaling of the instance-dependent bound is optimal. 

\begin{corollary}
    Suppose that there are \(m\) unique best arms (i.e. \(\mu_1\geq \cdots \geq \mu_m >\mu_{m+1}\geq \mu_k\)). Then, 
  \begin{eqnarray*}
  \limsup_{T\to\infty}\frac{\Regret(T)}{\log T}&\leq& \sum_{j>m}\frac{8}{\Delta_{m,j}}.
  \end{eqnarray*}
    In particular, for \(\Delta:=\min_{i\leq m,j>m}\mu_i-\mu_j\), we have
   \begin{eqnarray*}
   \limsup_{T\to\infty}\frac{\Regret(T)}{\log T}&=& O\!\left(\frac{k-m}{\Delta}\right).
\end{eqnarray*}
\end{corollary}

\section{Numerical Experiments}
\label{numerical}

This section presents numerical experiments. We begin with a benchmark study that compares \DRUCB with other algorithms adapted from the bandit literature and two operationally motivated heuristics. We then investigate the impact of replacement delays and replacement costs on the effectiveness of the algorithm. Additional numerical experiments are included in the appendices, which examine the sensitivity of the algorithm to the tuning parameter \(\gamma\), the tractability of the \ChooseTarget~subproblem, and the contributions of the different components of the {\DRUCB} algorithm to the performance gains.

\vspace{-.15in}

\subsection{Benchmarking}\label{numerical:benchmarks}

{\bf Problem instances.}
To evaluate the empirical performance of {\DRUCB}, we simulate a setting where the firm has access to a pool of \(k=25\) temporary workers and employs \(m=10\) workers in each period. The facility operates year-round, and we run each algorithm for \(24\) months or \(730\) days (periods), and decisions are made on a daily basis. We assume the firm has no prior knowledge about any of the workers. Each worker who is currently not in the workforce is available in any future period with independent probability \(p=1/3\). This induces a geometrically distributed delay with a mean of 3 days (\(\bar{\omega}=3\)). We measure profit in thousands of dollars. The replacement cost is set to be five thousand dollars (\(c=5\)). We let the output of worker \(i\) in period \(t\) be \(X_i^{(t)}=\mu_i+\epsilon_t\), where \(\epsilon_t\) is uniformly distributed on the interval \([\mu_i-\alpha_i, \mu_i+\alpha_i]\) and \(\alpha_i=\min(0.1,\mu_i, 1-\mu_i)\). For the intrinsic mean value \(\mu_i\) of each worker \(i\), we draw it from a truncated normal distribution \(\mathcal{N}(0.5,0.3)\) on \([0,1]\). The parameter \(\gamma\) in the DR-UCB algorithm is tuned for the parameter values \(c=5\), \(\bar{\omega}=3\), \(k=25\), \(m=10\), and \(T=730\) days (see Appendix \ref{sec:thm-all-constants}).

\begin{remark}[Computational Tractability]
    As previously mentioned, in Algorithm~\ref{alg:ChooseTarget}, the {\ChooseTarget} exact solver is a dynamic program whose complexity is polynomial in the size of its frontier. In our local Python implementation of the benchmark with \((k=25, m=10, T=730)\), a full {\DRUCB} simulation run takes about 15 milliseconds on average, the mean time per {\ChooseTarget} call is about 0.49 milliseconds with a maximum of 3.37 milliseconds, and the largest frontier we observed in the frontier-size study was 13 states. Appendix~\ref{sec:comp_of_selection_rule} provides the algorithmic details and Figure~\ref{fig:frontier-size} shows that in the benchmark study the maintained frontier remained very small throughout. 
\end{remark}

{\bf Adaptive comparison policies.}
The setting most closely related to ours is the \(m\)-play bandit with switching costs and no action delays. Accordingly, we benchmark against the standard algorithm from this literature, adapted to our setting, \cite{agrawal1990switching}, which we call \textsc{Adapted-AHT} (\textsc{A-AHT}). We also consider an adapted Optimistic Matroid Maximization algorithm (\textsc{A-OMM}) from \cite{kveton2014matroid}, which employs an optimism-based selection rule similar to ours and achieves \(O(\sqrt{kmT})\) regret in the absence of both delays and replacement costs. This choice reflects the canonical regret-optimal approach for stochastic matroid bandits, a class that includes the \(m\)-play setting as a special case.

Because our setting involves delayed replacements, both algorithms must be adapted before they can be implemented. First, the original algorithms prescribe workforce selections, whereas in our setting decisions are implemented through replacements. At each decision epoch, we query the algorithm for its desired workforce and then translate that workforce into a set of replacement requests. In the main benchmark, we use a random pairing for these delayed-action adaptations. Proposition \ref{lem:iid-rank-matching-E} suggests this should not materially weaken the expected regret. Second, the original algorithms assume instantaneous actions and therefore do not account for the feasibility constraints induced by delays. In particular, proposed replacements may violate the coherence constraints \eqref{eq:coherence-1}--\eqref{eq:coherence-2}, for example by attempting to remove or add a worker who is already involved in a pending replacement. To enforce feasibility, we discard any conflicting replacement requests. For \textsc{A-OMM}, our numerical results in Section \ref{numerical:delay-cost} suggest that our adaptation \emph{improves} the policy's performance. Because \textsc{A-OMM} is not designed to account for replacement costs, discarding some replacement decisions (thereby reducing the rate of replacements) improves its performance. For \textsc{A-AHT}, more sophisticated delay-aware adaptations may be possible, but they would require substantial modifications to the original policies. For example, it is not clear how to preserve its calendar-time-based update schedule without such modifications, since delayed replacements can prevent the requested workforce changes from being realized on schedule. In Section~\ref{numerical:delay-cost}, we also compare {\DRUCB} with the benchmark policies in the no-delay regime, where delay-related adaptations are unnecessary and any remaining performance differences cannot be attributed to those modifications.

Additional numerical analysis is reported in Appendix~\ref{sec:dr-ucb-internal} that examines how individual design choices affect the performance of {\DRUCB}. Specifically, we compare adaptive and fixed-calendar replacement timing, constrained selection and the unconstrained version of \ChooseTarget~(namely the top-\(m\) UCB target), and rank-matching and random pairing.

{\bf Screening-based comparison policies.}
We also include two screening-based heuristics intended to reflect common employment practices. The first is \textsc{InterviewScreen}. This policy pays an initial cost to interview each of the \(k\) candidates. Based on an interview score, the policy selects \(m\) candidates to be employed for the year. The second is \textsc{WorkTrial}. This policy is based on guidance from Indeed, a leading job posting company \citep{indeed_job_trial}. The policy first conducts an interview screen of each candidate. Based on the interview scores, it selects the top \(2m\) workers by mean. The firm employs the first \(m\) for 90 days, then replaces them with the second \(m\) for 90 days. This corresponds to a ``work trial.'' Afterward, the firm permanently retains the \(m\) best-performing workers. An in-depth description of how we model the interview-based signal using empirical research can be found in Appendix \ref{sec:interview-model}.

{\bf Results.}
We evaluate performance using two metrics. The first is the cumulative regret, defined in \eqref{def:regret}. The second is the normalized loss,
\[
\mathrm{Normalized\ Loss}(T)
    := \frac{\Regret(T)}{T\,\mu(A^*)} \times 100\%,
\]
which measures regret as a fraction of the oracle benchmark's total reward over the same horizon.

We run the main benchmark on \(250\) independently generated instances. Figure~\ref{fig:benchmark} plots the mean trajectory for each policy, with shaded bands showing \(\pm 1\) standard deviation across runs.

\begin{figure}[t]
    \centering
    \begin{minipage}[t]{0.47\linewidth}
        \centering
        \includegraphics[width=\linewidth]{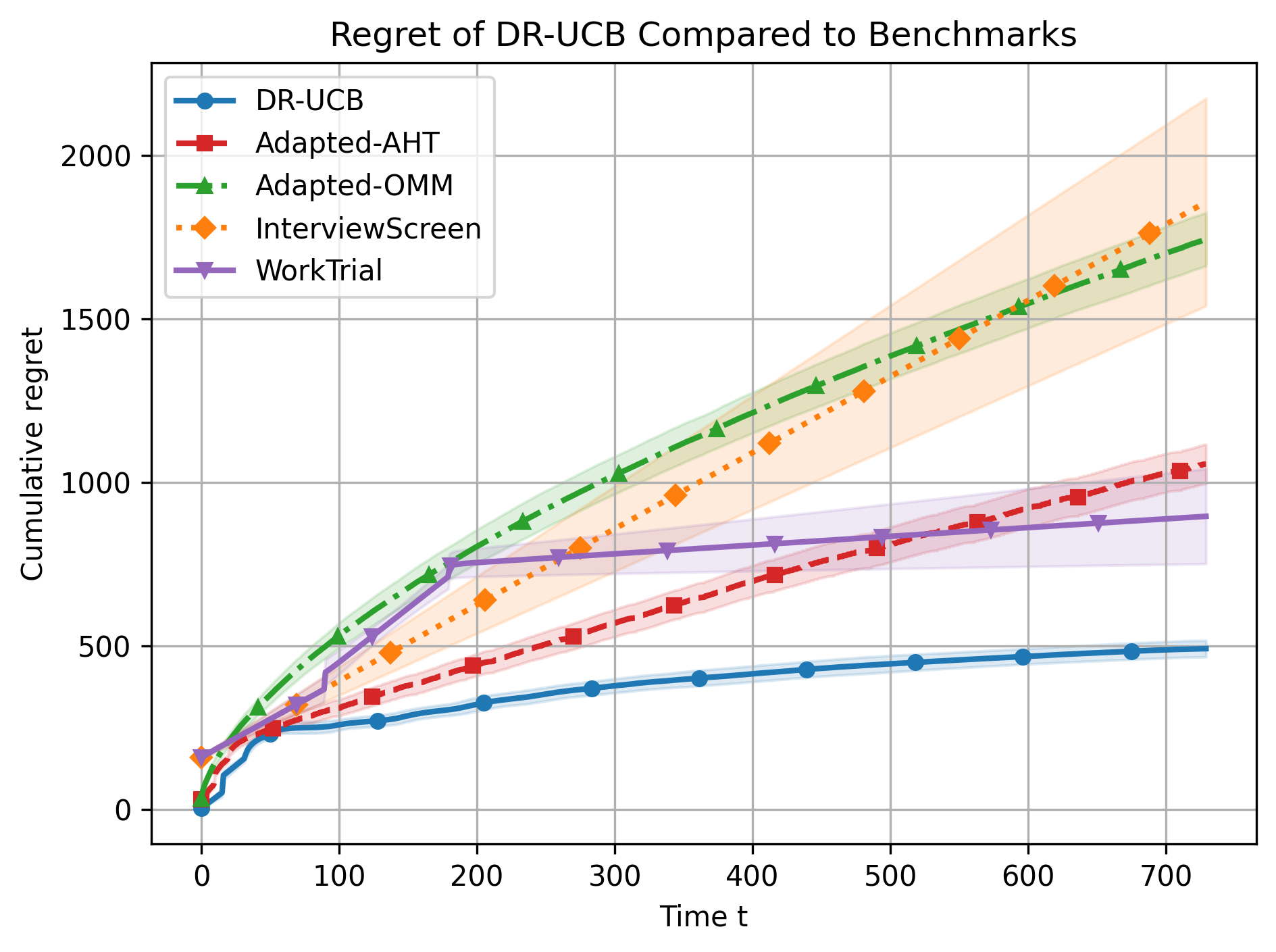}
    \end{minipage}
    \hfill
    \begin{minipage}[t]{0.47\linewidth}
        \centering
        \includegraphics[width=\linewidth]{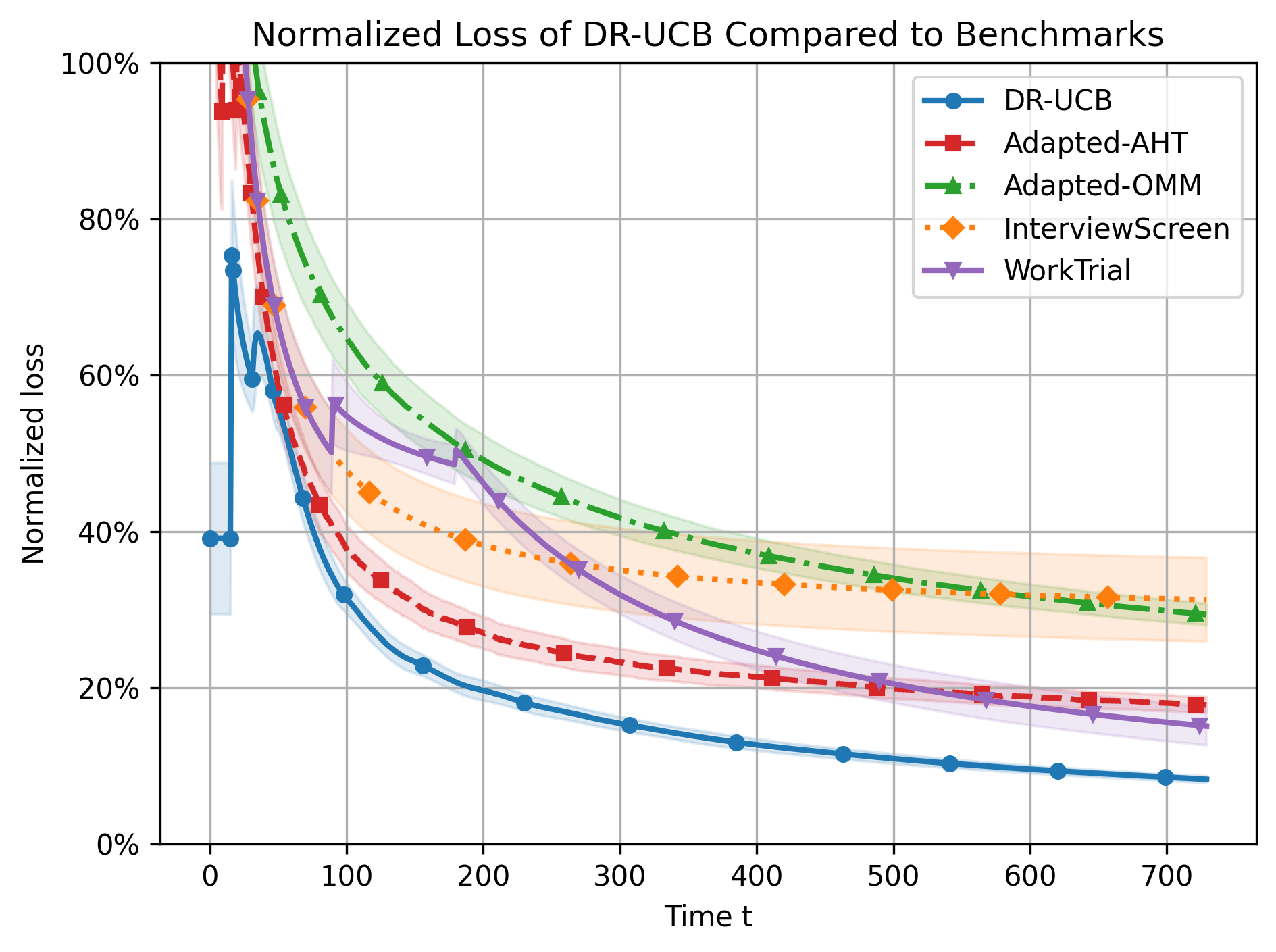}
    \end{minipage}

    \caption{Simulated performance of {\DRUCB} compared with adaptive benchmarks (\(m=10\)).}
    \label{fig:benchmark}
\end{figure}

\begin{table}[t]
    \centering

    \small
    \setlength{\tabcolsep}{4pt}
    
    \begin{tabular*}{\linewidth}{@{\extracolsep{\fill}} l r r r r r}
    \toprule
    Metric & DR-UCB & A-AHT & A-OMM & InterviewScreen & WorkTrial \\
    \midrule
    Regret & $402 \pm 24$ & $649 \pm 46$ & $1{,}145 \pm 59$ & $1{,}007 \pm 158$ & $799 \pm 73$ \\
    Normalized Loss & $13.6\% \pm 0.8\%$ & $21.9\% \pm 1.5\%$ & $38.6\% \pm 2.0\%$ & $33.9\% \pm 5.3\%$ & $26.9\% \pm 2.4\%$ \\
    \bottomrule
    \end{tabular*}
    \caption{Cumulative regret and normalized loss of policies at \(12\) months (250 simulated instances).}
    \label{fig:12-month-benchmark}
\end{table}

The main message is that {\DRUCB} provides the strongest overall performance. Although some benchmark policies perform well over particular portions of the horizon, {\DRUCB} is the only method that remains consistently near the frontier throughout. In particular, at one year it achieves the smallest cumulative regret and normalized loss among all policies tested.

At 12 months, the cumulative regret of {\DRUCB}
is about \(38\%\) lower than \textsc{A-AHT}, about \(65\%\) lower than \textsc{A-OMM}, about \(60\%\) lower than \textsc{InterviewScreen}, and about \(50\%\) lower than \textsc{WorkTrial}. It is also the most stable policy across instances. Additionally, the variability of the regret of {\DRUCB} is about \(48\%\) lower than \textsc{A-AHT}, about \(59\%\) lower than \textsc{A-OMM}, about \(85\%\) lower than \textsc{InterviewScreen}, and about \(67\%\) lower than \textsc{WorkTrial}. Thus, within this benchmark set, {\DRUCB} not only achieves the best mean performance at the target horizon, but does so substantially more reliably across heterogeneous instances. Table \ref{fig:12-month-benchmark} provides the numerical results on cumulative regret and normalized loss for the various algorithms at \(12\) months (\(T=365\)). The results for other time horizons (\(T=\)3, 6, 12, and 24 months) can be found in Appendix \ref{sec:benchmark_tables}.

Figure~\ref{fig:benchmark_alt} reports the regret and normalized loss at 12 months for two extreme worker-utilization regimes, \(m=5\) and \(m=20\).
More results, for time horizons 3, 6, 12, and 24 months, are given in 
Appendix \ref{sec:benchmark_tables}. The qualitative pattern is unchanged and the message remains the same: {\DRUCB} performs best over all horizons.

\begin{figure}[t]
    \centering
    \begin{minipage}[t]{0.45\linewidth}
        \centering
        \includegraphics[width=\linewidth]{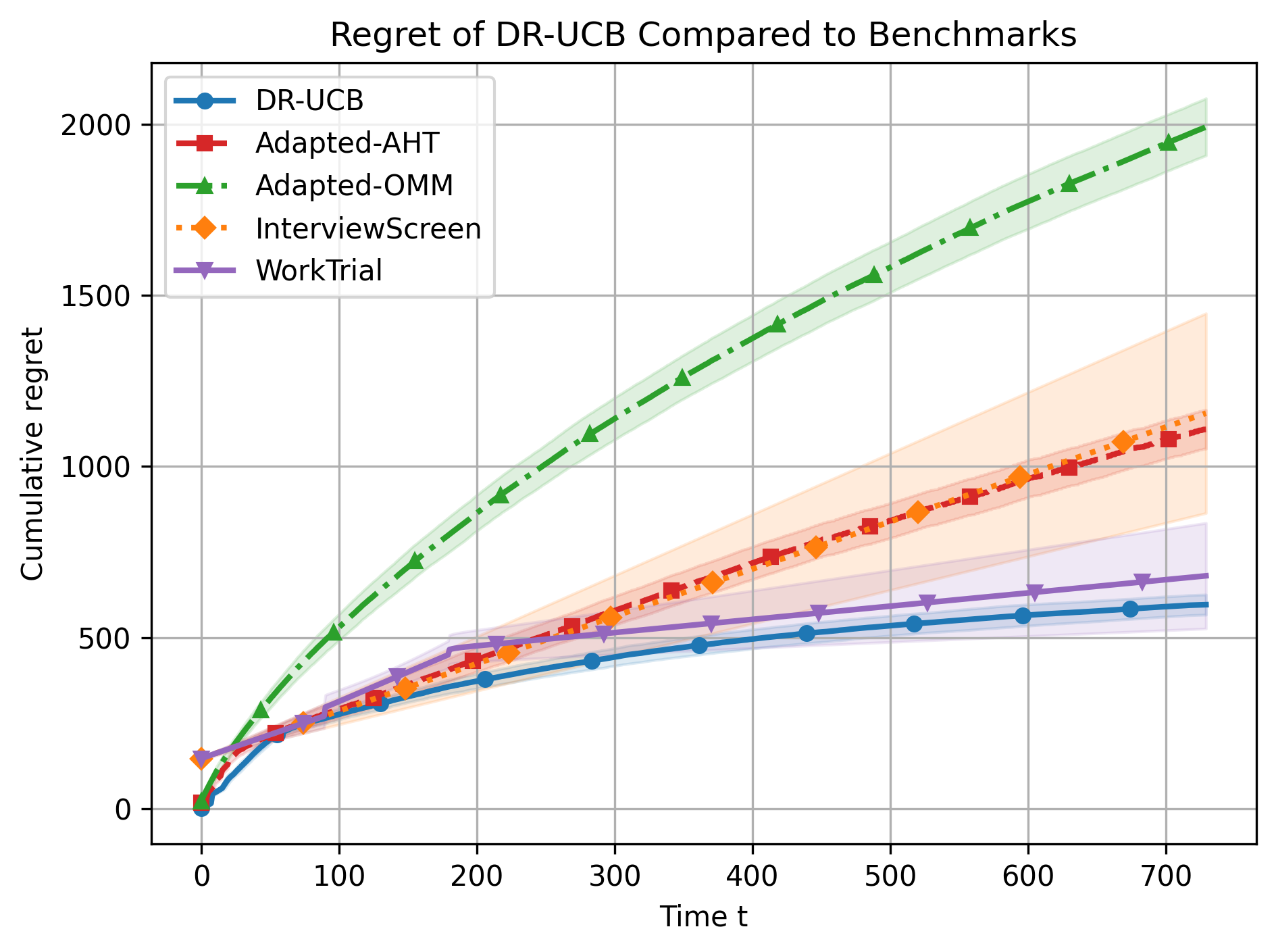}
    \end{minipage}
    \hfill
    \begin{minipage}[t]{0.45\linewidth}
        \centering
        \includegraphics[width=\linewidth]{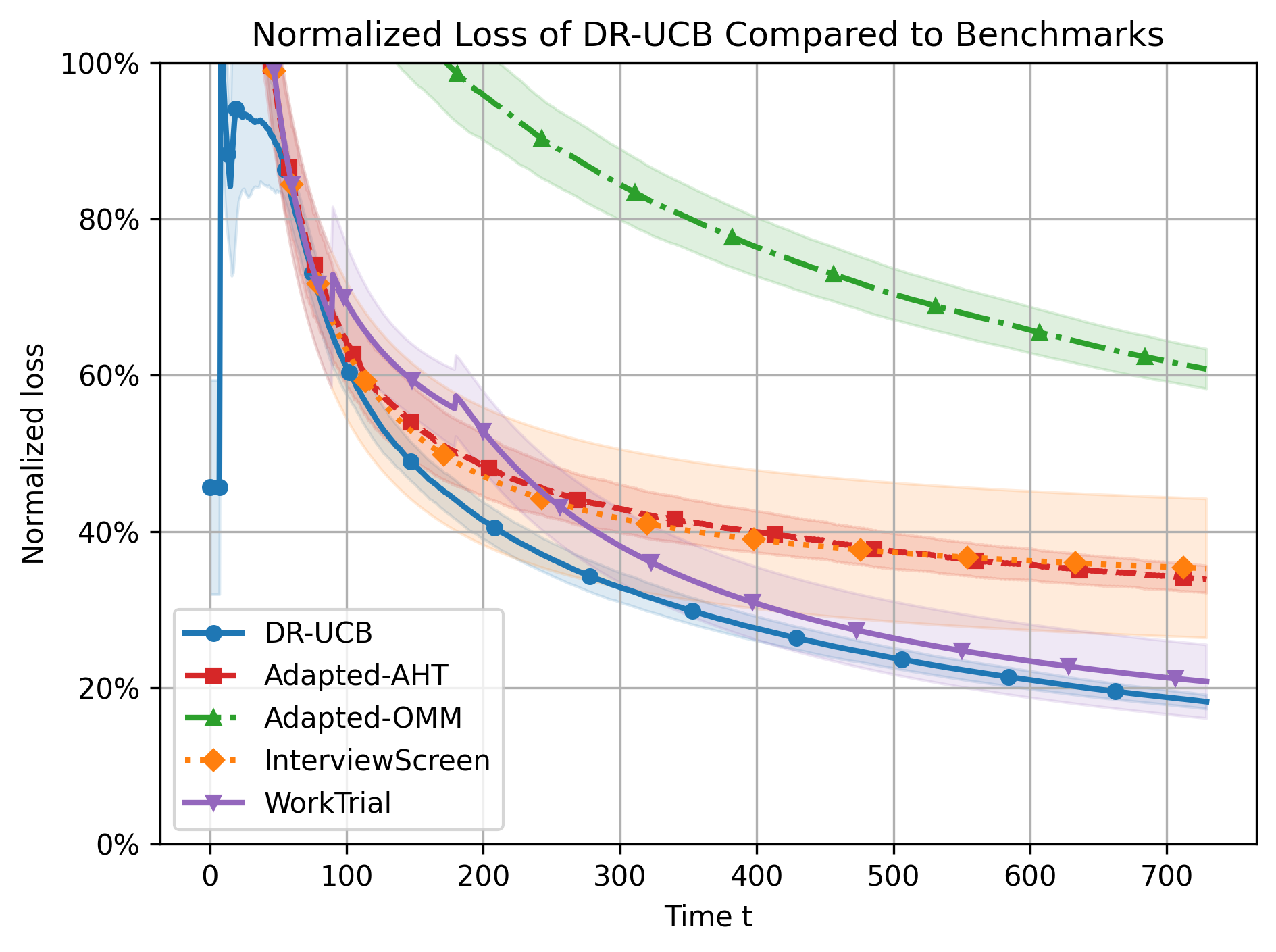}
    \end{minipage}
    \vspace{0.5em}
    \begin{minipage}[t]{0.45\linewidth}
        \centering
        \includegraphics[width=\linewidth]{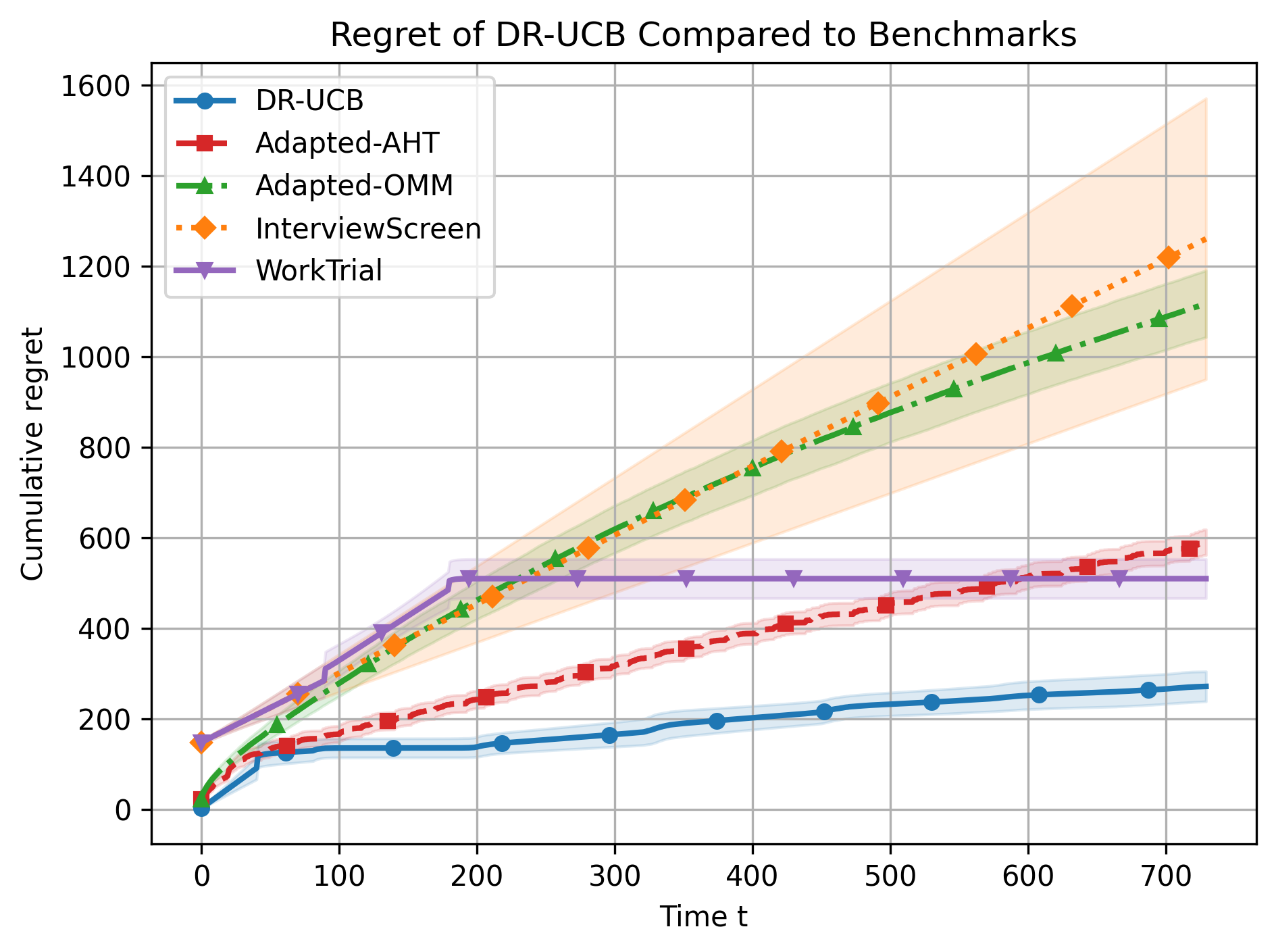}
    \end{minipage}
    \hfill
    \begin{minipage}[t]{0.45\linewidth}
        \centering
        \includegraphics[width=\linewidth]{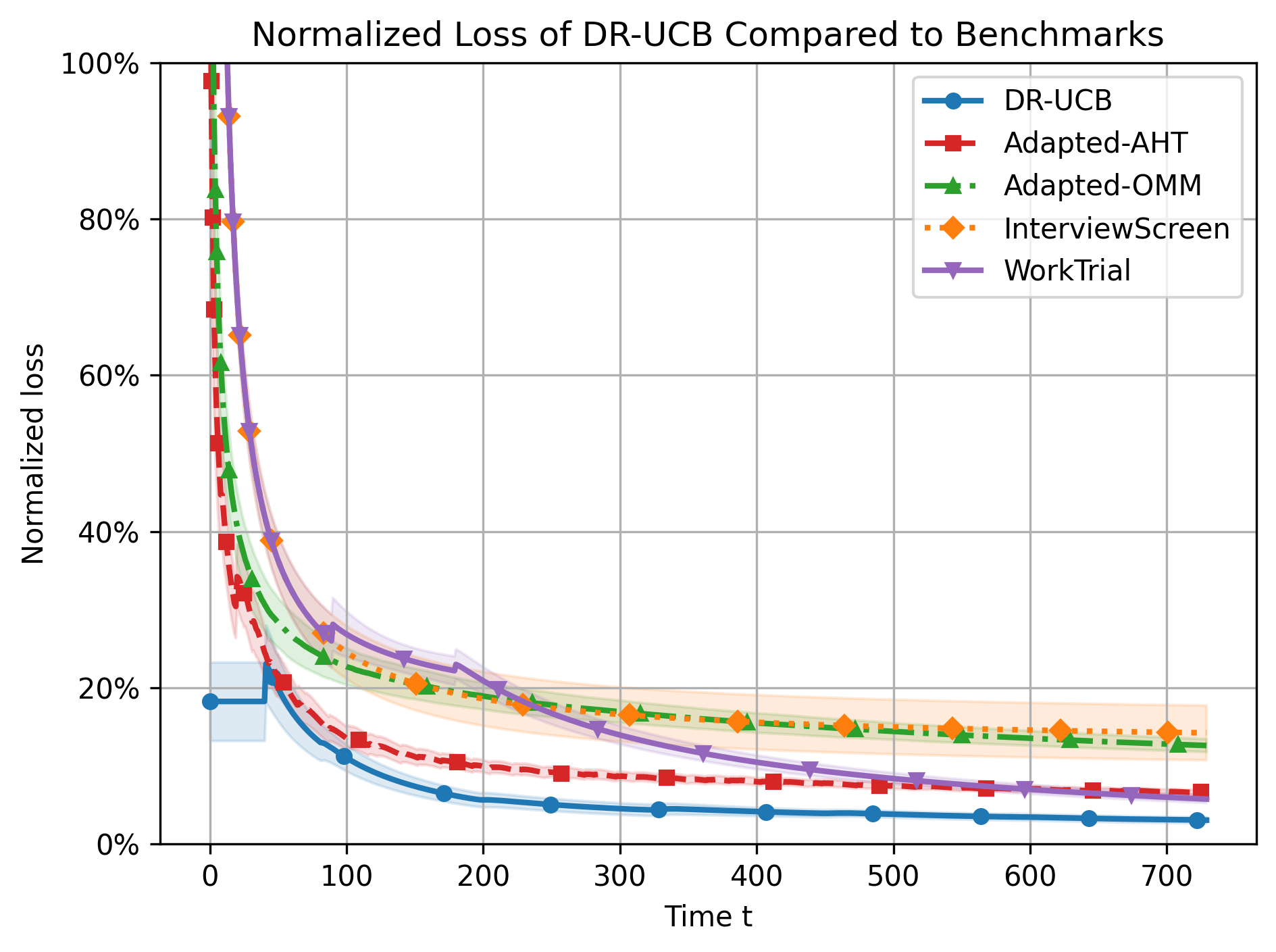}
    \end{minipage}

    \caption{Simulated performance of {\DRUCB} compared with adaptive benchmarks with \(m=5\) (top) and \(m=20\) (bottom).}
    \label{fig:benchmark_alt}
\end{figure}

\begin{table}[t]
    \centering

    \small
    \setlength{\tabcolsep}{4pt}
    
    \begin{tabular*}{\linewidth}{@{\extracolsep{\fill}}l l r r r r r}
    \toprule
    $m$& Metric & DR-UCB & A-AHT & A-OMM & InterviewScreen & WorkTrial \\
    \midrule
    5& Regret& $479 \pm 27$ & $669 \pm 46$ & $1{,}294 \pm 65$ & $651 \pm 145$ & $540 \pm 77$ \\
    5& Normalized Loss & $29.3\% \pm 1.7\%$ & $40.9\% \pm 2.8\%$ & $79.1\% \pm 4.0\%$ & $39.8\% \pm 8.9\%$ & $33.0\% \pm 4.7\%$ \\
    \midrule
    20& Regret& $193 \pm 28$ & $364 \pm 22$ & $707 \pm 57$ & $703 \pm 154$ & $509 \pm 43$ \\
    20& Normalized Loss& $4.4\% \pm 0.6\%$ & $8.2\% \pm 0.5\%$ & $16.0\% \pm 1.3\%$ & $15.9\% \pm 3.5\%$ & $11.5\% \pm 1.0\%$ \\
    \bottomrule
    \end{tabular*}
    \caption{Cumulative regret and normalized loss for \(m=5\) and \(m=20\) at \(12\) months (250 simulated instances).}
    \label{tab:12-month-benchmark_alt}
\end{table}

More numerical results are reported in Appendix~\ref{sec:dr-ucb-internal} in which we analyze the separate contributions of the main design choices. In the benchmark calibration, the adaptive count-based replacement rule provides the largest improvement. Relative to a fixed calendar schedule, it reduces 3-month regret from \(297\) to \(238\) and 12-month regret from \(485\) to \(407\) in a 250-run simulation. The constrained selection rule in {\ChooseTarget} also improves performance, relative to the unconstrained top-\(m\) UCB target, reducing 3-month regret from \(267\) to \(238\) and 12-month regret from \(431\) {to \(407\)}. Finally, the last simulation shows that the rank-matching rule preserves performance even though it commits to the incoming/outgoing worker pairing at the time of hire. Its performance is comparable to the more flexible benchmark that delays the removal decision until the incoming worker becomes available.

\vspace{-.1in}

\subsection{Effects of Delays and Replacement Costs 
}\label{numerical:delay-cost}
In Figure~\ref{fig:sweep_delay}, we examine the effect of delayed replacements by varying the mean delay \(\bar\omega\). Keeping the other parameters the same as in Section~\ref{numerical:benchmarks}, we test the values of \(\bar{\omega}\) corresponding to a zero-delay case (requests may be realized in the same period before rewards are drawn), a mean delay of 1 day, a mean delay of 3 days, and a mean delay of 7 days.

\begin{figure}[h]
    \centering
    \begin{minipage}[t]{0.45\linewidth}
        \centering
        \includegraphics[width=\linewidth]{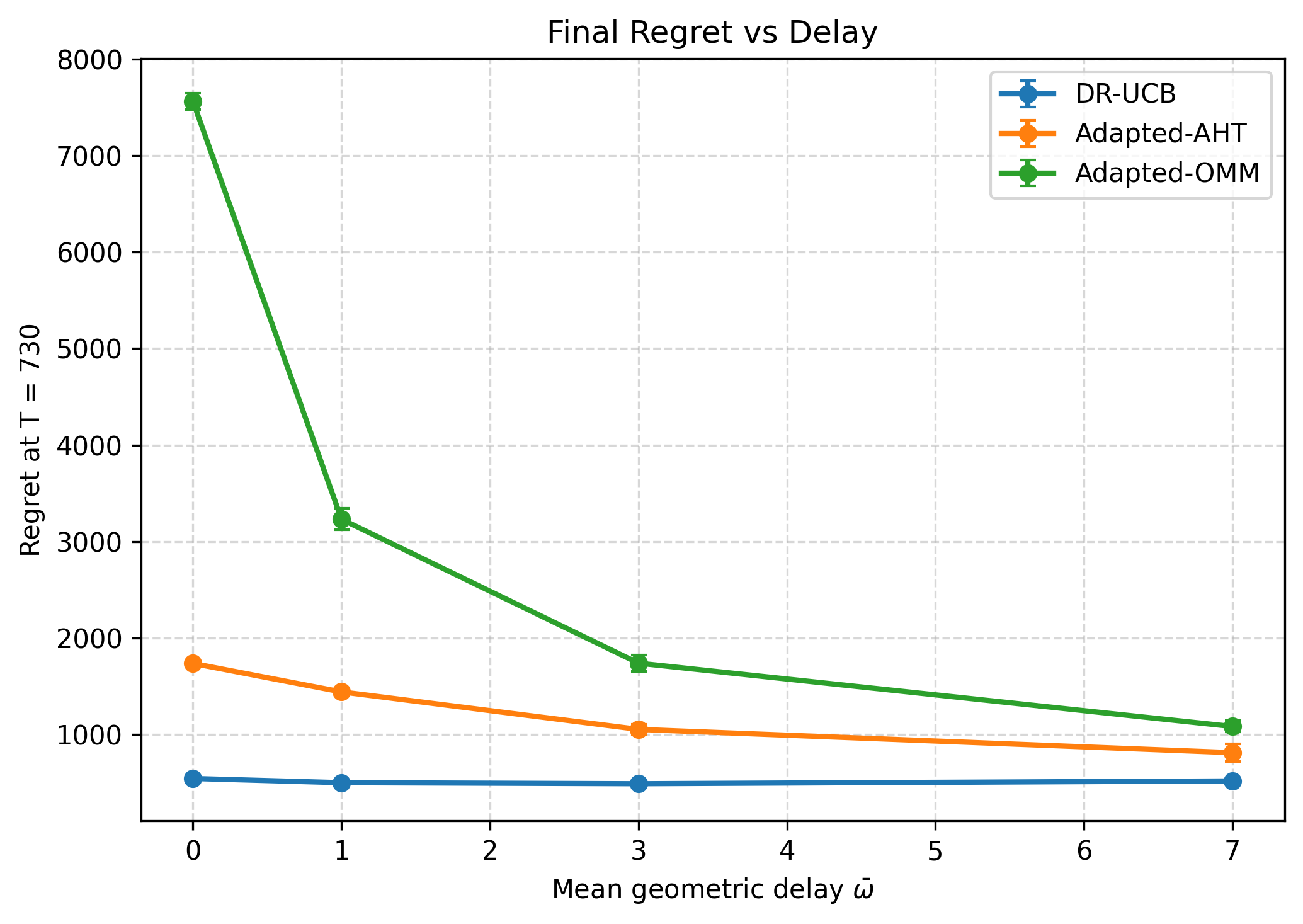}
    \end{minipage}
    \hfill
    \begin{minipage}[t]{0.45\linewidth}
        \centering
        \includegraphics[width=\linewidth]{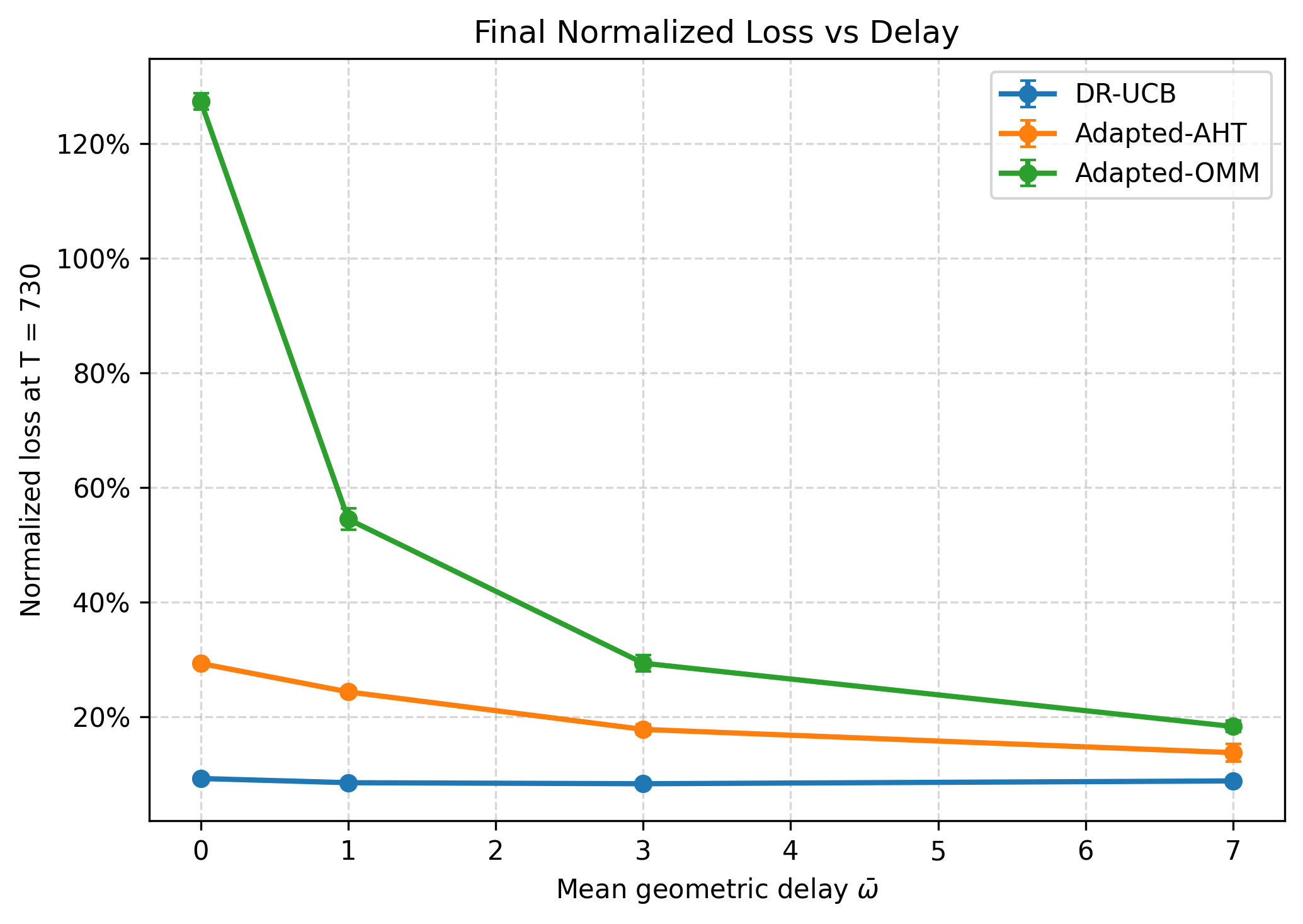}
    \end{minipage}

    \caption{Simulated performance of \textsc{A-AHT}, \textsc{A-OMM}, and {\DRUCB} under several delay settings.}
    \label{fig:sweep_delay}
\end{figure}

Figure~\ref{fig:sweep_delay} yields two main observations. First, {\DRUCB} outperforms \textsc{A-AHT} even in the zero-delay regime \((\bar{\omega}=0)\). In this case, replacements are realized in the same environment step before period-\(t\) rewards are generated, so no delay-related adaptations are needed and the policy is a faithful implementation of the original policy from \cite{agrawal1990switching}. The performance gap therefore cannot be attributed solely to the way \textsc{A-AHT} is modified for delayed replacements, and instead suggests that {\DRUCB} more effectively balances exploration and replacement timing even when requests can be implemented immediately.

Second, the regret of \textsc{A-OMM} and \textsc{A-AHT} \emph{decreases} as delays increase. The reason is that \textsc{A-OMM} and \textsc{A-AHT} both initiate too many replacements, which generates excessive replacement costs. Larger delays cause more of these intended replacements to be discarded, effectively regularizing the policy's replacement frequency. To illustrate this effect, Figure~\ref{fig:sweep-delay-repl} plots the number of completed replacements by delay.

\begin{figure}[h]
    \centering
    \includegraphics[width=0.5\linewidth]{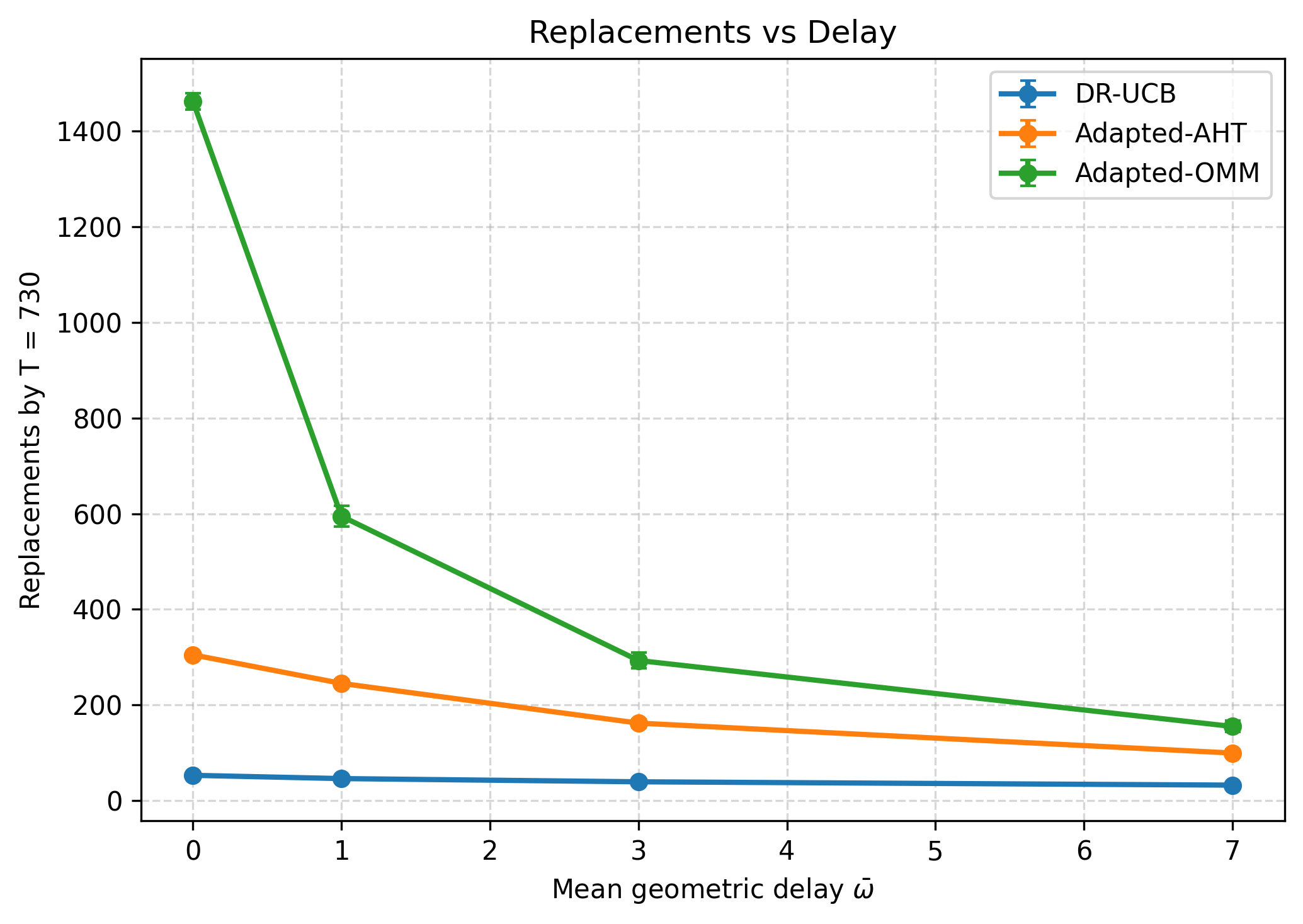}
    \caption{Number of replacements completed by \textsc{A-AHT}, \textsc{A-OMM}, and {\DRUCB} under several delay settings.}
    \label{fig:sweep-delay-repl}
\end{figure}

This highlights an important operational caution when replacement costs exist: delays may not always be harmful, and what matters operationally is the replacement rate. If the firm is already replacing at a near-optimal rate, reducing delays can increase profit. But if the firm is already replacing too frequently, reducing delays may counterintuitively \emph{decrease} profit. In contrast, \DRUCB maintains low regret across the different delay regimes tested.

We next investigate the effect of replacement costs on learning algorithms. 
In Figure~\ref{fig:sweep_c}, we examine the performance of \textsc{A-AHT}, \textsc{A-OMM}, and {\DRUCB} under the same setup as Section~\ref{numerical:benchmarks}, with \(k=25\), \(m=10\), and \(\bar{\omega}=3\). The costs tested are \(c\in \{3,7,30,60,90\}\).

\begin{figure}[h]
    \centering
    \begin{minipage}[t]{0.45\linewidth}
        \centering
        \includegraphics[width=\linewidth]{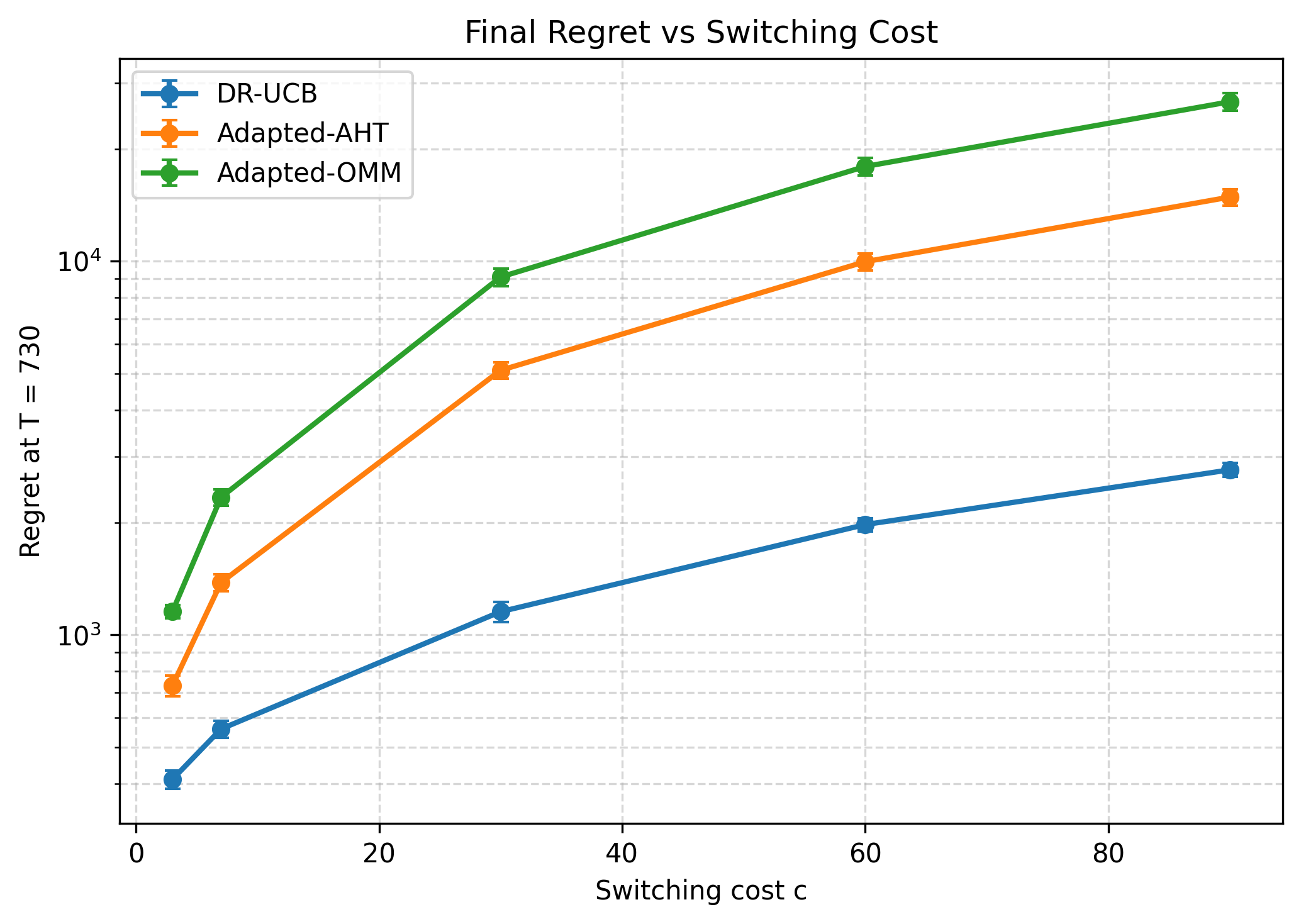}
    \end{minipage}
    \hfill
    \begin{minipage}[t]{0.45\linewidth}
        \centering
        \includegraphics[width=\linewidth]{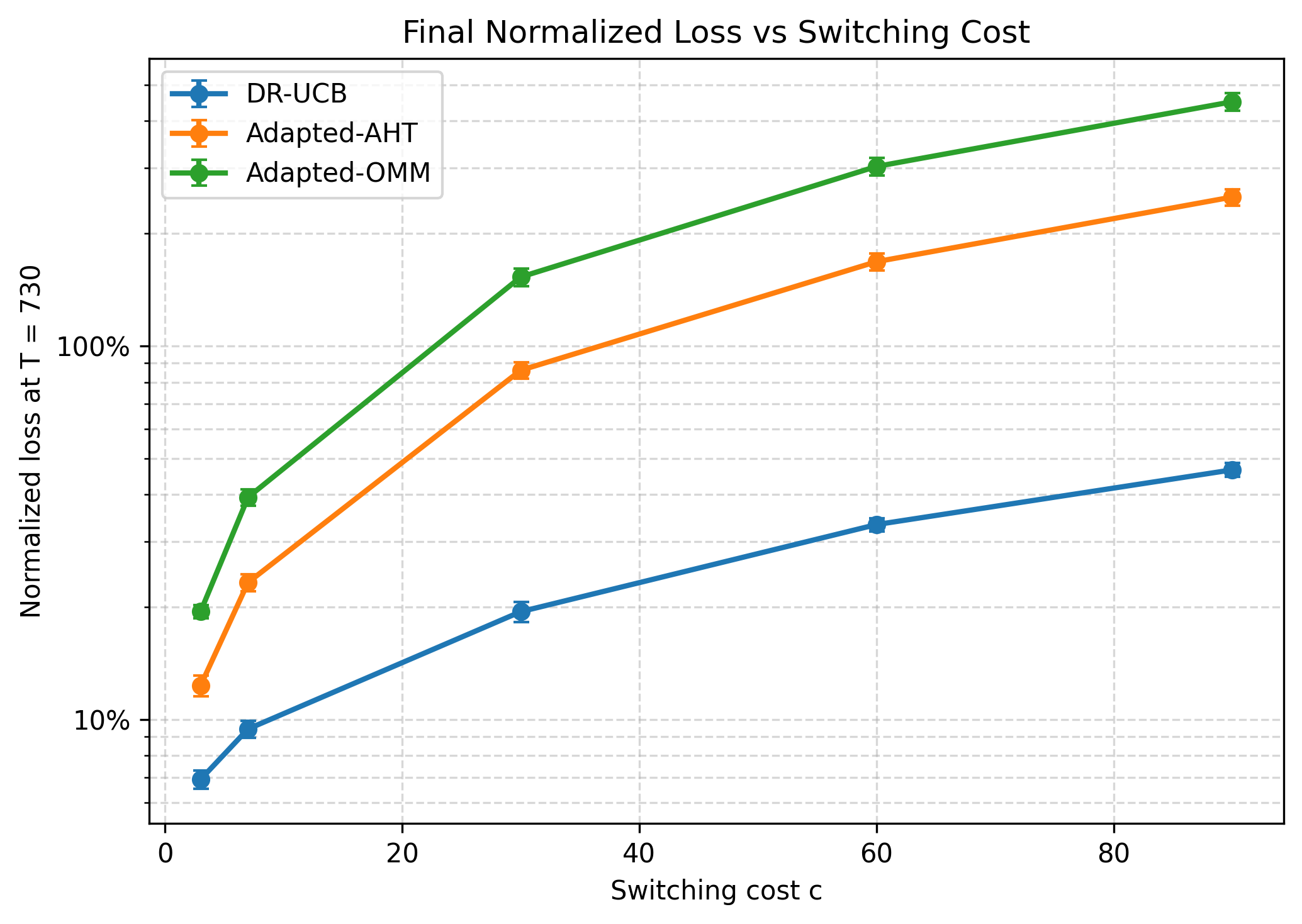}
    \end{minipage}

    \caption{Simulated performance of \textsc{A-AHT}, \textsc{A-OMM}, and {\DRUCB} under several cost settings.}
    \label{fig:sweep_c}
\end{figure}

Across the instances tested, {\DRUCB} outperforms \textsc{A-AHT} and \textsc{A-OMM}. Its advantage also grows as replacements become more costly. This occurs because {\DRUCB} adapts its replacement frequency to the replacement cost \(c\). When \(c\) is small, frequent replacements can be useful because they allow the firm to learn quickly. When \(c\) is large, frequent replacements become expensive, so the policy benefits from replacing workers less often.

Figure~\ref{fig:sweep-c-repl} shows the number of completed replacements by each policy at \(T=730\). The number of switches made by \textsc{A-AHT} and \textsc{A-OMM} is constant across values of \(c\), since these policies do not adjust their replacement frequency to the replacement cost. As a result, their regret increases roughly in proportion to \(c\) and to the number of replacements they make. In contrast, {\DRUCB} makes fewer replacements as \(c\) increases. This allows it to better balance learning benefits against the ensuing replacement costs.

\begin{figure}[h]
    \centering
    \includegraphics[width=0.55\linewidth]{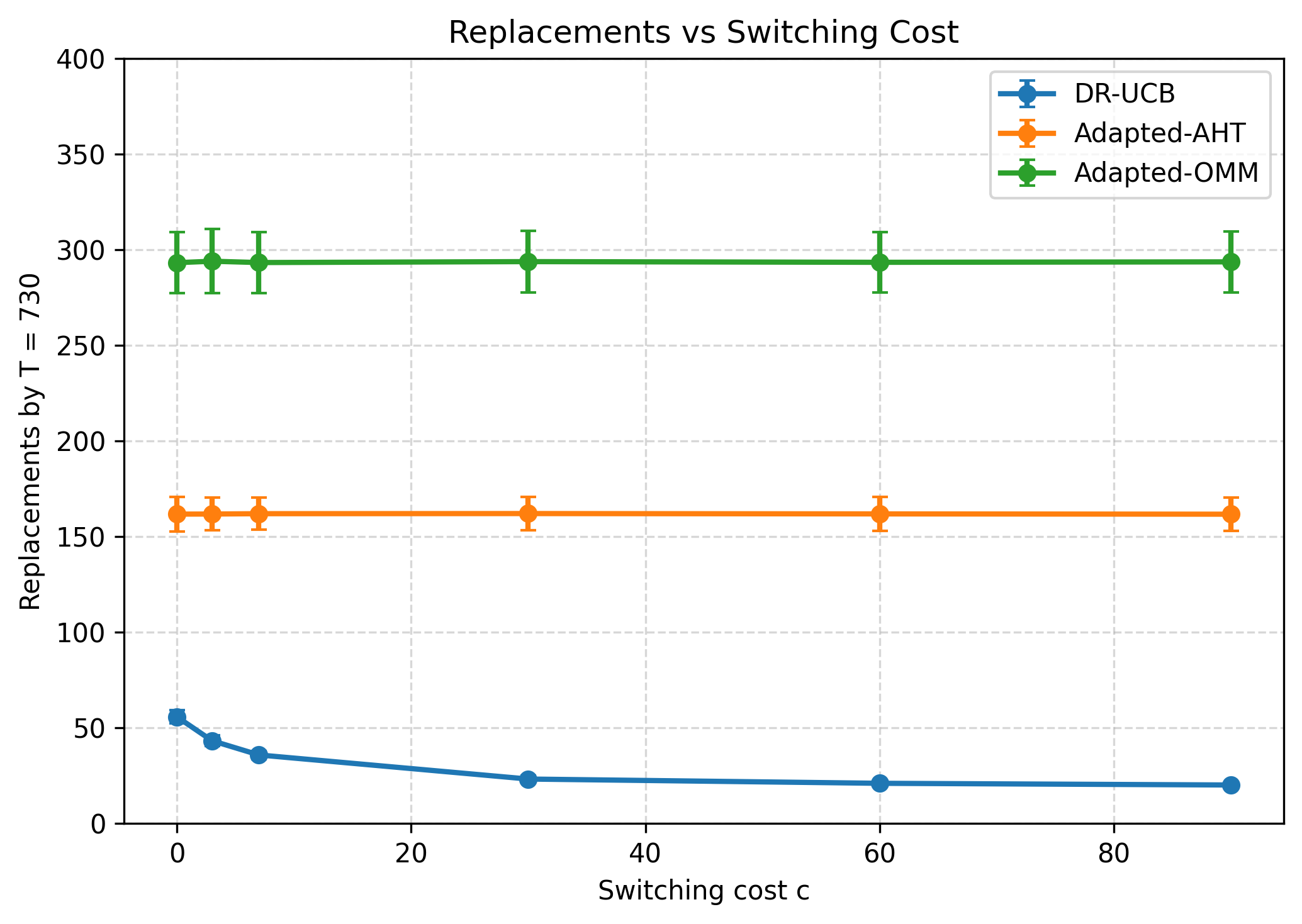}
    \caption{Number of replacements completed by \textsc{A-AHT}, \textsc{A-OMM}, and {\DRUCB} under several cost settings.}
    \label{fig:sweep-c-repl}
\end{figure}

\section{Concluding Remarks}\label{sec:conclusion}

This paper studies a sequential hiring problem in a contingent labor environment, where a firm must learn worker productivity while managing costly and delayed workforce adjustments. We formulate this problem as a multi-play bandit model with replacement costs and worker-availability uncertainty. We develop a learning-based hiring policy, {\DRUCB}, that consists of three core components: adaptive replacement timing, constrained target selection, and a pairing step. 

Our theoretical analysis shows that {\DRUCB} attains a regret bound whose leading-order dependence on the horizon \(T\) matches the lower bound up to logarithmic terms. Our numerical experiments support the theoretical results. Across the benchmark settings, {\DRUCB} compares favorably with adapted bandit algorithms and operationally motivated heuristics. Methodologically, our analysis develops a novel approach for handling features absent from standard bandit models, including adaptive switching schedules, the mismatch between statistical identification and physical replacement, and intermediate workforce configurations induced by delayed availability. We believe this approach can also be useful for studying other related learning problems. The model developed in this paper is most applicable when the set of potential workers is moderate in size, as may occur when a platform provides a limited set of recommendations, or when the firm has some initial data on candidate workers. 

When the worker pool is very large or effectively unbounded, and the firm has little or no initial information about most workers, a different modeling approach may be needed. One possible direction is to combine our framework with methods for reducing large candidate sets. For example, one could adapt approaches such as those in \cite{wang2008algorithms}, which first identify a finite subset that contains a sufficiently good worker with high probability and then apply a finite-armed bandit algorithm to that restricted set. Another direction is to use a contextual bandit model, in which each worker is represented by observable features and the expected production output of a worker is a function of those features. In linear or generalized linear bandit models, for example, learning reduces to estimating the coefficients of the reward model rather than estimating a separate mean for each worker. This formulation would also allow the candidate pool to change arbitrarily over time, since new workers could be evaluated through their observable characteristics. The adaptive switching and constrained selection ideas developed in this paper could be extended to such contextual settings. 

Finally, an interesting direction for future research is to incorporate strategic worker behavior. Workers may accept or reject assignments based on outside opportunities, wages, or market conditions. Incorporating these features would necessitate modeling the interaction between firm decisions and worker responses, likely using tools from dynamic games or mechanism design.


\section{Code and Data Disclosure}\label{sec:code-data-disclosure}
The code used in this paper is available at \url{https://github.com/leechr001/hiring-bandit}. The repository includes the simulation environment, implementations of all policies reported in the paper, and scripts to reproduce the figures and tables reported in the appendices. 

\vspace{.25in}

\bibliographystyle{plainnat}
\bibliography{hiring-bandit}

\newpage
\appendix

\noindent {\Large Appendices: Proofs and Simulation Details}
\vspace{.1in}

\section{Statement of the Main Theorem With Exact Constants}\label{sec:thm-all-constants}
Let \(k\) denote the number of arms, \(m\) the size of each solution, and \(\gamma>0\) a policy parameter. Then the expected regret of Algorithm 1 over \(T\) periods satisfies
\begin{eqnarray*}
    \Regret(T) &\leq& 5\sqrt{m(k-m)T\log T}\\
    &&+\;(c+1)(k-m) +(k-m)c\log\left(\frac{cmT}{(k-m)\log T}\right)\\
    &&+ k(\gamma + 1) + \frac{4mk}{T}\\
    &&+2(k-m)\sqrt{\gamma \cdot 4\log T}
    \left(1+\frac{1}{2}\log\!\left(\frac{mT}{(k-m)\log T}\right)\right)\\
    &&+\;(\bar{\omega}+c) m\sqrt{\frac{kT}{\gamma}} + (\bar{\omega}+c)mk.
\end{eqnarray*}

\section{Language, Notation, and Structure for Proofs}
To streamline the presentation of the proofs, we describe the workforce-selection problem using the standard language of bandit algorithms. In this formulation, workers are treated as ``arms,'' the \(m\)-worker roster corresponds to a ``solution,'' observed output is interpreted as a ``reward,'' and employing a worker is viewed as ``playing an arm.'' We also introduce structural notation describing the algorithm's control flow, decision instants, and target solutions, which will be used consistently throughout the analysis.

Each time Algorithm~\ref{alg:DelayedReplace-UCB} initiates replacements, the counter \(\ell\) increases. This counter also indexes the sequence of target solutions \(U_{\ell}\). Accordingly, we refer to each \(\ell = 1,2,\dots\) as an epoch. The target solution \(U_\ell\) associated with epoch \(\ell\) is computed in the first period of that epoch. We call this initial time the \emph{decision instant} and denote it by
\begin{equation}\label{eq:c-def}
    d_\ell := \text{the first period of epoch \(\ell\)}.
\end{equation}
The algorithm may not be able to play \(U_\ell\) immediately because of arm unavailability. Instead, it gradually transitions toward \(U_\ell\) and commits to playing it once all arms in \(U_\ell\) are available. We denote this first play time by the \emph{play instant},
\begin{equation}\label{p-def}
    p_\ell := \text{the first period in which \(U_\ell\) is available}.
\end{equation}
This naturally partitions each epoch into two phases:
\[
    \mathcal{T}_\ell = \{ t : d_\ell \le t < p_\ell \},
    \quad
    \mathcal{B}_\ell = \{ t : p_\ell \le t < d_{\ell+1} \}.
\]
We call the first phase the \emph{transition phase}, since the active set moves from the previous target solution to the new one. We refer to the second phase as the \emph{block phase}, since it plays a role analogous to the blocks in block-based algorithms. Unlike in those algorithms, however, the length of our block phases is adaptive: the length of \(\mathcal{B}_{\ell}\) depends on the realized delays. For example, it is possible that \(\mathcal{B}_\ell = \emptyset\) if the realized delays are sufficiently long. Likewise, it is possible that \(\mathcal{T}_\ell = \emptyset\) if \(U_\ell\) is available immediately at \(d_\ell\). Figure~\ref{fig:timeline} illustrates these intervals.

A crucial feature of this setup is that although each epoch is structured around a single target solution \(U_{\ell}\), the actual plays of an arm \(j \in U_\ell\) need not occur exclusively within that epoch. Conversely, during epoch \(\ell\), arms from \(U_{\ell-1}\) may still be played regardless of whether they belong to \(U_\ell\). For convenience, we define
\[
    \epoch(t) := \text{the unique \(\ell\) such that \(t \in \mathcal{T}_\ell \cup \mathcal{B}_\ell\)},
\]
and let \(L := \epoch(T)\) denote the index of the last epoch before horizon \(T\).

\begin{figure}[t] 
    \centering
    \begin{tikzpicture}[xscale=1, yscale=1]
    
    \draw[thick, ->] (0,0) -- (14,0) node[right] {time};
    
    \node at (1,1.2) {\dots Epoch \(\ell-1\)};
    \node at (6,1.2) {Epoch \(\ell\)};
    \node at (12,1.2) {Epoch \(\ell+1\) \dots};
    
    \draw[fill=blue!15] (0,0.3) rectangle (4,0.8);
    \node at (2,0.55) {\(\mathcal{B}_{\ell-1}\)};
    
    \draw[fill=red!15] (4,0.3) rectangle (6,0.8);
    \node at (5,0.55) {\(\mathcal{T}_\ell\)};
    
    \draw[fill=blue!15] (6,0.3) rectangle (10,0.8);
    \node at (8,0.55) {\(\mathcal{B}_{\ell}\)};
    
    \draw[fill=red!15] (10,0.3) rectangle (12,0.8);
    \node at (11,0.55) {\(\mathcal{T}_{\ell+1}\)};
    
    \draw[fill=blue!15] (12,0.3) rectangle (14,0.8);
    \node at (13,0.55) {\(\mathcal{B}_{\ell+1}\)};
    
    \draw[dashed] (4,0) -- (4,1);
    \node[below] at (4,0) {\(d_\ell\)};
    
    \draw[dashed] (6,0) -- (6,1);
    \node[below] at (6,0) {\(p_{\ell}\)};
    
    \draw[dashed] (10,0) -- (10,1);
    \node[below] at (10,0) {\(d_{\ell+1}\)};
    
    \draw[dashed] (12,0) -- (12,1);
    \node[below] at (12,0) {\(p_{\ell+1}\)};
    
    \end{tikzpicture}
    \caption{Timeline of consecutive epochs and their phases.}
    \label{fig:timeline}
\end{figure}
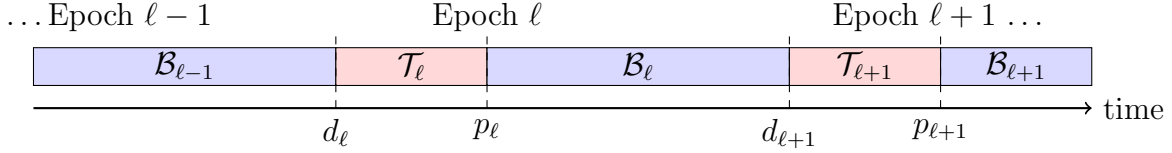

\vspace{-.1in}

\section{Proof of the Regret Bound}\label{sec:proof-outline}
In this section, we sketch the main argument underlying the regret bound and highlight the key properties on which the analysis depends. Proofs of the supporting lemmas are given in the subsequent sections.

The regret naturally consists of two components. Let \(S_T\) denote the \emph{sampling regret}, corresponding to the regret incurred from selecting suboptimal workers, and let \(J_T\) denote the \emph{replacement cost}. Then
\[
\Regret(T) = S_T + J_T.
\]

{\bf Bounding replacement costs.}
The main difficulty in controlling replacement-cost regret is that epoch lengths are both \emph{endogenous} and \emph{history dependent}: the duration of an epoch depends on the realized delays and on the number of samples accumulated by the least-observed arm. Consequently, epoch lengths are coupled to both the learning trajectory and the unknown delay process. Moreover, the block thresholds \(N(\ell)\) need not be monotone in the epoch index \(\ell=1,\dots,L\), which further complicates any attempt to directly bound the total number of policy replacement epochs.

Our approach avoids tracking which specific arm attains the minimum sample count in each epoch. Instead, we analyze the evolution of the minimum count itself. By decoupling the count from the identity of the arm that achieves it, we can exploit simple combinatorial properties to obtain deterministic lower bounds on how small this minimum can be at any given period. Showing that the minimum sample count must increase sufficiently quickly yields a lower bound on the number of periods required to complete an epoch. This, in turn, implies an upper bound on the total number of epochs---and hence on the total number of replacement epochs---over a horizon of \(T\) periods.

The analysis therefore hinges on understanding how the minimum sample count evolves over time under the blocking rule~\eqref{eq:block-threshold}. In particular, we need a mechanism that converts lower bounds on the growth of this minimum into upper bounds on the number of epochs the policy can execute. The first step is to show that the minimum cannot remain within any fixed interval of sample-count space for too many epochs.

The key idea is a monotonicity-driven counting argument. Although the evolution of play counts is random, an arm's count never decreases. Thus, once an arm has been the minimizer at some level \(a\), the block threshold~\eqref{eq:block-threshold} forces its count above \(a+g(a)\), where \(g(x)=2\sqrt{\gamma x}+\gamma\). Consequently, that arm cannot be the minimizer again while the minimum remains in the interval \([a, a+g(a))\). Since at most \(k\) distinct arms can serve as minimizers, no more than \(k\) epochs can have minimum sample counts in any such interval (Lemma~\ref{lem:local-packing}). This ``local packing'' property is the structural core of the proof: it converts the adaptive blocking rule into a deterministic limit on how many epochs can occur at comparable minimum sample levels.

To lift this local bound to a global bound on the total number of epochs, we introduce a deterministic comparison sequence defined by \(a_{n+1}=a_n+g(a_n)\). This recursion yields quadratic growth, with \(a_n=\gamma n^2\). The induced intervals \(I_n=[a_n, a_n+g(a_n))\) partition the space of minimum sample counts into successive levels. By the packing lemma, each level \(I_n\) can contain at most \(k\) epochs. Moreover, any epoch whose minimum sample count lies in \(I_n\) must last at least \(g(a_n)\) periods. To upper bound the total number of epochs, we therefore consider a worst-case trajectory in which, subject to the packing constraint, epochs occur at the smallest admissible minimum sample levels. In such a trajectory, each epoch incurs the minimal duration permitted by the blocking rule, yielding the slowest possible growth in elapsed time. This construction gives
\[
    d_{\ell+1}
    \;\gtrsim\;
    k \sum_{n=0}^{\lfloor \ell/k\rfloor-1} g(a_n)
    \;=\;
    \gamma k \left(\lfloor \ell/k\rfloor\right)^2,
\]
which shows that the epoch timeline must grow quadratically. This is formalized in Lemma~\ref{lem:ell-t-bound}, where we derive the bound
\[
    d_{\ell+1}
    \;\ge\;
    \gamma k \left(\frac{\ell}{k}-1\right)^2.
\]

Finally, since each epoch can change up to \(m\) arms, the replacement-cost regret is at most \(cm\) times the number of epochs completed by time \(T\). Solving \(d_{\ell+1} \leq T\) yields \(\ell \leq \sqrt{kT/\gamma} + k\). Therefore, we obtain the replacement-cost bound stated in Lemma~\ref{lem:sc-bound-final},
\[
    J_T \leq cm\left(\sqrt{\frac{kT}{\gamma}} + k\right).
\]

This bound reinforces a central message from the literature on costly action changes: limiting the frequency of workforce adjustments is essential for favorable regret scaling. The key distinction in our setting is that adjustment timing is not imposed exogenously; instead, it is governed by an adaptive, data-dependent threshold rule that depends endogenously on both learning progress and execution delays. This endogenous timing structure creates challenges not present in settings with fixed or externally specified update schedules and therefore requires a new analytical approach.

{\bf Bounding sampling regret.}
Next, we bound the sampling regret \(S_T\). The regret at time \(t\) is the difference between the expected reward of the optimal solution \(A^*\) and that of the chosen solution \(A_t\):
\[
    S_T
    \le \bE{\sum_{t=1}^T \left(\sum_{i\in A^*}\mu_i - \sum_{j\in A_t}\mu_j \right)}
    = \bE{\sum_{t=1}^T \left(\sum_{i\in A^*\setminus A_t}\mu_i - \sum_{j\in A_t\setminus A^*}\mu_j \right)}.
\]
This expression shows that regret is driven entirely by the mismatch between the optimal arms omitted from \(A_t\) and the suboptimal arms included in \(A_t\). Pairing each included suboptimal arm with the best optimal arm that is omitted yields the upper bound
\[
    S_T
    \le \bE{\sum_{t=1}^T \sum_{j\in A_t\setminus A^*}
    \Big( \max_{i\in A^*\setminus A_t}\mu_i - \mu_j \Big)}.
\]
For convenience, let
\[
    i_t^* = \argmax_{i\in A^*\setminus A_t}\mu_i
\]
denote the best optimal arm excluded at time \(t\). Then each term can be expressed in terms of the optimality gap \(\Delta_{i_t^*,j} := \mu_{i_t^*} - \mu_j\), yielding
\[
    S_T = \sum_{j>m} \bE{\sum_{t=1}^T \Delta_{i_t^*,j} \ind{j\in A_t}}.
\]
A key step is to show that we can control \(i_t^*\). We begin by splitting the expectation into two cases:
\begin{eqnarray*}
    S_T
    &\le& \sum_{j>m} \bE{\sum_{t=1}^T \Delta_{i_t^*,j} \ind{j\in A_t, \,\Delta_{i_t^*,j}\leq \frac{c}{T-t}}} \\
    &&\qquad + \sum_{j>m} \bE{\sum_{t=1}^T \Delta_{i_t^*,j} \ind{j\in A_t, \,\Delta_{i_t^*,j}> \frac{c}{T-t}}}.
\end{eqnarray*}
We first bound the first sum. Note that when \(t < T-c/\Delta_{m,j}\), we have
\[
    \frac{c}{T-t}\leq \frac{c}{c/\Delta_{m,j}}=\Delta_{m,j}\leq \Delta_{i_t^*,j},
\]
so the event cannot hold. We therefore bound the first term using harmonic sums:
\begin{align*}
    &\sum_{j>m} \bE{\sum_{t=1}^T \Delta_{i_t^*,j} \ind{j\in A_t, \,\Delta_{i_t^*,j}\leq \frac{c}{T-t}}}
    \\
    &\qquad \le (k-m)+\sum_{j>m}\sum_{t=\lceil T-c/\Delta_{m,j}\rceil}^{T-1} \frac{c}{T-t}
    \\
    &\qquad \le (k-m)+\sum_{j>m} c\bigl(1+\log(c/\Delta_{m,j})\bigr)
    \\
    &\qquad \le (c+1)(k-m) +\sum_{j>m}c\log\!\left(\frac{c}{\Delta_{m,j}}\right).
\end{align*}
For the remainder of the proof, we work on the event \(\{\Delta_{i_t^*,j}> c/(T-t)\}\). To simplify notation, we omit this event from the expressions below and refer to it explicitly only when needed.

It suffices to analyze the regret contribution of a fixed suboptimal arm \(j>m\). Fix such an arm \(j\). We will bound
\[
    \bE{\sum_{t=1}^T \Delta_{i_t^*,j} \ind{j\in A_t}}.
\]
The algorithm proceeds in epochs indexed by \(\ell\), with boundaries \((d_\ell,p_\ell,d_{\ell+1})\). Grouping periods by epoch, we can write
\[
    \bE{\sum_{\ell=1}^L \sum_{t\in[d_\ell,d_{\ell+1})} \Delta_{i_t^*,j} \ind{j\in A_t}}.
\]
Identifying the transition phase \(\mathcal{T}_{\ell}\) and block phase \(\mathcal{B}_{\ell}\), this becomes
\[
    \bE{\sum_{\ell=1}^L \left(\sum_{t\in \mathcal{T}_\ell} \Delta_{i_t^*,j} \ind{j\in A_t}
    + \sum_{t\in \mathcal{B}_{\ell}} \Delta_{i_t^*,j} \ind{j\in A_t}\right)}.
\]

The main obstacle is that the regret incurred by playing arm \(j\) is governed by \(\Delta_{i_t^*,j}\), so the relevant comparator depends on \emph{which} optimal arms are missing from the played set \(A_t\). In a multi-play setting, \(j\in A_t\) does not imply that the most rewarding arm is excluded, so a uniform bound such as \(\Delta_{i_t^*,j} \leq \Delta_{1,j}\) can be too crude to reflect the information carried by the rest of \(A_t\). Delayed replacements further complicate the analysis by making the realized comparison depend on the order in which replacements complete, rather than only on the intended action. Together, these features preclude a direct extension of single-arm elimination arguments and require a more refined treatment. We therefore bound the regret from the block phases and transition phases separately.

To bound the regret during the block phases, we need to control
\[
    \bE{\sum_{\ell=1}^L \sum_{t\in \mathcal{B}_{\ell}} \Delta_{i_t^*,j} \ind{j\in A_t}}.
\]
The first step is to control the optimality gap \(\Delta_{i_t^*,j}\). To do so, we organize the history using a sequence of stopping times that indicate when suboptimal arms can be statistically distinguished from optimal ones. For each optimal arm \(i \leq m\), define
\begin{equation}\label{eq:main-tij}
    t_{i,j}:=\min\Big\{t: T_j(t)>\tfrac{4\log T}{\Delta_{i,j}^2}\Big\}.
\end{equation}
We also identify the epoch containing \(t_{i,j}\) and denote it by
\[
    \ell_{i,j}:=\epoch(t_{i,j}).
\]
For notational convenience, we also set \(t_{0,j}:=1\) and \(\ell_{0,j}:=1\). 
Under a suitable concentration event and on the event \(\{\Delta_{i_t^*,j}> c/(T-t)\}\), we can show that for every epoch after \(\ell_{i,j}\), if the suboptimal arm \(j\) is included in the target solution \(U_\ell\), then the optimal arm \(i\) is included as well. Formally, it follows from Lemma~\ref{lem:gap-bound} that for all \(\ell>\ell_{i,j}\), we have \(\Delta_{i_t^*,j} \le \Delta_{i,j}\). Once \(\ell>\ell_{m,j}\), where \(m\) is the worst optimal arm, we have \(j\notin A_t\) by Lemma~\ref{lem:arm-elimination}. Thus, if we partition the epochs \(\ell=1,\dots,L\) according to these stopping times,
\[
    [1,L]=\left(\bigcup_{n=1}^m [\ell_{n-1,j}+1, \ell_{n,j}]\right)\cup \big\{\ell\geq \ell_{m,j}+1\big\},
\]
then each segment admits the corresponding bound from Lemma~\ref{lem:gap-bound}. Specifically, for each period \(t\) in an epoch \(\ell \in [\ell_{n-1,j}+1, \ell_{n,j}]\), we have \(\Delta_{i_t^*,j} \le \Delta_{n,j}\), while \(j\notin A_t\) for \(\ell\geq \ell_{m,j}+1\). This gives the control we need over \(\Delta_{i_t^*,j}\).

\begin{figure}[t]
    \centering
    \begin{tikzpicture}[xscale=1.2, yscale=1]

    \draw[thick, ->] (0,0) -- (12,0) node[right] {time};

    \node at (1,1) {\dots \(\ell_{i-1,j}\)};
    \node at (3.5,1) {\(\ell_{i-1,j}+1\)};
    \node at (8,1) {\(\ell_{i,j}\)};
    \node at (10.8,1) {\(\ell_{i,j}+1\) \dots};

    \draw[fill=yellow!30] (2.5,0.2) rectangle (10.5,0.7);
    \node at (6,0.45) {\(\dots\)};

    \draw[fill=red!15] (0,0.2) rectangle (1,0.7);
    \node at (0.5,0.45) {\(T\)};
    \draw[fill=blue!15] (1,0.2) rectangle (2.5,0.7);
    \node at (1.7,0.45) {\(B\)};

    \draw[fill=yellow!30] (2.5,0.2) rectangle (3.5,0.7);
    \node at (3,0.45) {\(T\)};
    \draw[fill=yellow!30] (3.5,0.2) rectangle (5,0.7);
    \node at (4.2,0.45) {\(B\)};

    \draw[fill=red!15] (9.5,0.2) rectangle (10.5,0.7);
    \node at (10,0.45) {\(T\)};
    \draw[fill=blue!15] (10.5,0.2) rectangle (12,0.7);
    \node at (11.2,0.45) {\(B\)};

    \draw[fill=yellow!30] (7,0.2) rectangle (8,0.7);
    \node at (7.5,0.45) {\(T\)};
    \draw[fill=yellow!30] (8,0.2) rectangle (9.5,0.7);
    \node at (8.7,0.45) {\(B\)};

    \draw[dashed, thick] (2,0) -- (2,1);
    \node[below] at (2,0) {\(t_{i-1,j}\)};

    \draw[dashed, thick] (7.8,0) -- (7.8,1);
    \node[below] at (7.8,0) {\(t_{i,j}\)};

    \end{tikzpicture}
    \caption{Timeline illustrating the epochs \(\ell_{i-1,j}+1,\dots, \ell_{i,j}\) in yellow. Note that the periods \(t_{i-1,j}\) and \(t_{i,j}\) are ``offset'' from the beginning and end of the epochs they define.}
    \label{fig:main-ell-interval}
\end{figure}
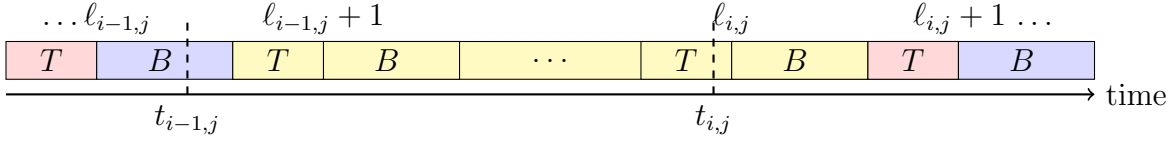

With this control of \(\Delta_{i_t^*,j}\), the remaining task is to determine how many times the suboptimal arm \(j\) is selected during each interval of epochs \([\ell_{n-1,j}+1, \ell_{n,j}]\). By the definition of \(t_{i,j}\) in \eqref{eq:main-tij}, we know, up to integer ambiguity, the number of plays of arm \(j\) at these stopping times. However, because of the delay between the initiation of a replacement and its completion, together with the blocking rule that forces the solution to remain fixed, the periods \(t_{i-1,j}\) and \(t_{i,j}\) do not align with the boundaries of the epochs \(\ell_{i-1,j}+1,\dots, \ell_{i,j}\), where we control \(\Delta_{i_t^*,j}\). Figure~\ref{fig:main-ell-interval} illustrates this offset: the periods \(t_{i-1,j}\) and \(t_{i,j}\) lie to the left of the epoch boundaries that define the highlighted interval.

For the periods between \(t_{i-1,j}\) and \(d_{\ell_{i-1,j}+1}\)---the first period of epoch \(\ell_{i-1,j}+1\); see \eqref{eq:c-def}---the optimality gap \(\Delta_{i_t^*,j}\) cannot be bounded by \(\Delta_{i-1,j}\) via Lemma~\ref{lem:gap-bound}. However, the \emph{number of plays} of arm \(j\) in the periods \(d_{\ell_{i-1,j}+1},\dots, p_{\ell_{i,j}}\), which lie within the epoch interval \(\ell_{i-1,j}+1,\dots, \ell_{i,j}\), cannot exceed the number of plays in the periods \(t_{i-1,j},\dots, p_{\ell_{i,j}}\), since \(t_{i-1,j}\leq d_{\ell_{i-1,j}+1}\) and the play count is nondecreasing over time. We may therefore decompose the relevant time periods as
\[
    \{t:t_{i-1,j}\leq t < t_{i,j}\} \cup  \{t:t_{i,j}\leq t \leq p_{\ell_{i,j}}\}.
\]
The number of plays in the left set can be bounded by the difference
\[
    \frac{4\log T}{\Delta_{i,j}^2} - \frac{4\log T}{\Delta_{i-1,j}^2}
\]
using the definition of \(t_{i,j}\). Combining this play-count bound with the bound on \(\Delta_{i_t^*,j}\), and then telescoping over the index \(i\), yields the leading \(\log T/\Delta_{m,j}\) term. The number of plays in the right set can be bounded using the block threshold \(N(\ell)\). Careful accounting, together with a second telescoping sum, produces the secondary \(\sqrt{\gamma\log T}\left[1+\log(\Delta_{1,j}/\Delta_{m,j})\right]\) contribution and a lower-order remainder of order \(\Delta_{1,j}(\gamma+1)\) (Lemma~\ref{lem:B-final}).

To complete the argument, we bound the regret incurred during the transition phases. During these phases, the arbitrary nature of replacement times obscures the identity of the best omitted optimal arm, so we do not attempt a refined bound on \(\Delta_{i_t^*,j}\). Moreover, suboptimal arms may persist in transition phases even after they can no longer appear in a block phase, for example if their replacements are delayed. Instead, we use a worst-case bound. Each epoch contains at most \(\bar{\omega}\) transition periods, and each such period incurs regret at most \(m\Delta_{1,k} \leq m\), since all gaps lie in \([0,1]\). It therefore suffices to bound the number of epochs completed within \(T\) rounds. The epoch-count bound from the replacement-cost analysis yields at most \(\sqrt{kT/\gamma}+k\) epochs. Putting these pieces together gives the statement of Lemma~\ref{lem:transition-regret},
\[
    (\text{Transition regret})
    \le
    \bar{\omega} m\!\left(\sqrt{\frac{kT}{\gamma}}+k\right).
\]
Combining this with the contribution from the block phases completes the sampling-regret bound. Combining the result further with the replacement-cost bound completes the proof.

\section{Bounding the Switching Regret}
We now turn to bounding the regret due to arm switches. Recall that the length of each epoch is determined by the least played arm in the current target solution. The central challenge is to understand how quickly this minimum play count can grow. Although the process is stochastic, it has a crucial monotonicity property: once an arm has been played a certain number of times, its count can never decrease. This means the minimum can never fall below its past values, and hence it cannot remain close to a fixed level for too many epochs.  

\subsection{Controlling the Growth of the Minimum Play Count}
The idea is as follows. Each epoch length is tied to the current minimum, so controlling its growth directly limits how many epochs can occur over a horizon, and therefore how many switches can occur. Because only \(k\) different arms can serve as the minimizer, once all of them have been ``cycled through," the minimum must increase. This implies a packing-type bound: the number of epochs with blocks governed by minima near one another is finite and controlled. The next lemma makes this statement precise.

\begin{lemma}[Local packing in \(X\)-space]\label{lem:local-packing}
    For each epoch \(\ell\), define
    \[
        X(\ell) = \min_{i \in A_{d_\ell}} T_i(d_\ell),
    \]
    the minimum number of plays among the \(m\) arms selected at time \(d_\ell\). Let
    \[
        g(x) := 2\sqrt{\gamma  x} + \gamma 
        \quad \text{ and } \quad 
        N(\ell) := \bigl\lceil g(X(\ell)) \bigr\rceil.
    \]
    Then for any \(a \ge 0\), the number of epochs whose minimum lies in the
    interval \([a,  a + g(a))\) is at most \(k\):
    \[
        \bigl|\{\ell : a  \leq X(\ell) < a + g(a)\}\bigr|  \leq k.
    \]
\end{lemma}

\noindent {\bf Proof.}
    Fix \(a\ge 0\) and consider the set of epochs
    \[
        \mathcal{L}(a) := \{\ell : a  \leq X(\ell) < a + g(a)\}.
    \]
    If \(\mathcal{L}(a)=\emptyset\), there is nothing to prove. Otherwise, list its
    elements in increasing order:
    \[
        \ell_1 < \ell_2 < \cdots < \ell_q,
        \qquad q := |\mathcal{L}(a)|.
    \]
    For each \(\ell_n \in \mathcal{L}(a)\), select an arm \(i_n \in U_{\ell_n}\) that
    attains the minimum, i.e.
    \[
        T_{i_n}(d_{\ell_n}) = X(\ell_n).
    \]
    Because \(X(\ell_n)\in[a,a+g(a))\), we have in particular \(X(\ell_n)\ge a\).
    Since \(g\) is increasing and \(N(\ell_n)=\lceil g(X(\ell_n))\rceil\), it follows that
    \[
        N(\ell_n) \ge g(X(\ell_n)) \ge g(a).
    \]
    Thus by the end of epoch \(\ell_n\),
    \[
        T_{i_n}(d_{\ell_n + 1})
        = T_{i_n}(d_{\ell_n}) + N(\ell_n)
        \ge
        a + g(a).
    \]
    Play counts are nondecreasing over time, so for any later epoch
    \(\ell > \ell_n\), 
    \[
        T_{i_n}(d_\ell) \ge T_{i_n}(d_{\ell_n+1}) \ge a + g(a).
    \]
    In particular, for any \(\ell > \ell_n\) with \(X(\ell)\in[a,a+g(a))\), arm \(i_n\)
    cannot be the minimum again, because its count at the start of epoch \(\ell\)
    already lies at or above \(a+g(a)\), outside the interval.
    
    We now show that the arms \(i_1,\dots,i_q\) are all distinct. Suppose for
    contradiction that \(i_{n'} = i_n\) for some \(1  \leq n' < n  \leq q\).  Then at the start
    of epoch \(\ell_n\) we would have
    \[
        T_{i_n}(d_{\ell_n}) \ge a + g(a),
    \]
    by the argument above applied to the earlier epoch \(\ell_{n'}\). But since \(\ell_n\in\mathcal{L}(a)\), we also have
    \[
        T_{i_n}(d_{\ell_n}) = X(\ell_n) < a + g(a),
    \]
    a contradiction. Hence all \(i_1,\dots,i_q\) are distinct arms. There are only \(k\) arms in total, so \(q  \leq k\). Since
    \(q = |\mathcal{L}(a)|\), we obtain
    \[
        \bigl|\{\ell : a  \leq X(\ell) < a + g(a)\}\bigr|  \leq k,
    \]
    as claimed.
\qed  

\subsection{Quadratic Growth of the Epoch Timeline}
We now turn to lifting this bound to a sequence of iterations. Let \(\{\ell_{r}\}_{r\geq 1}\) be a sequence of epochs ordered chronologically (i.e., \(\ell_{1}< \ell_{2}<\cdots < \ell_{r}\)). Note that this sequence need not contain all epochs between \(\ell_1\) and \(\ell_r\). As an illustrative example, we will later consider the sequence of epochs \(\{\ell_r :\exists i>m \text{ with } i\in U_{\ell_r}\}\). The quantity of interest is the minimum number of periods, over all possible trajectories, that could be contained in these epochs. Let us denote this
\[K(r)=\sum_{s=1}^r(|\mathcal{T}_{\ell_s}|+|\mathcal{B}_{\ell_s}|)\]
To fix ideas, note that if \(\{\ell_{r}\}_{r\geq 0}\) is the sequence \(1,2,\dots, r\) (the first \(r\) consecutive epochs), then \(K(r)=d_{r+1}\), the timestamp of the first period of the (\(r+1\))th epoch. Our goal is to show that \(K(r)\) cannot grow too slowly. The key observation is that in each group of \(k\) epochs, the minimum play count must increase according to \(g(x)\). Consequently, the sequence of minima cannot stagnate, because it is bounded below by a deterministic recurrence. 

\begin{lemma}\label{lem:ell-t-bound}
    Let \(\{\ell_{r}\}_{r\geq 1}\) be a sequence of epochs and \(X(\ell_1)\) be the minimum count at the beginning of epoch \(\ell_1\). Then for every  \(r > k\)
    \[
        K(r)
        \ge
        k\left(\sqrt{\gamma}\left(\frac{r}{k} - 1\right)+\sqrt{X(\ell_1)}\right)^2-kX(\ell_1).
    \]
    and for every \(r \leq k\), \(K(r) \geq r\left(2\sqrt{X(\ell_1)}+ \gamma\right)\).
\end{lemma}

\noindent {\bf Proof.}
    From the definition of \TimeToReplace~(Algorithm \ref{alg:TimeToReplace}) we have that the length of an epoch is either the block threshold \(N(\ell)\) or the longest delay, whichever comes last. Thus,
    \[ K(r)=\sum_{s=1}^r|\mathcal{T}_{\ell_s}|+|\mathcal{B}_{\ell_s}|=\sum_{s=1}^r\max\left(N(\ell_r), \max_{(i,j)\in R_{d_{\ell_r}}}\omega_{i,j}^{(d_{\ell_r})}\right)\geq \sum_{s=1}^r N(\ell_r)
    \]
    
    We first consider the case that \(r \leq k\). Then, because 
    \[N(\ell_r)\geq r\cdot  \min_{i\in U_{\ell_r}}\left( 2\sqrt{\gamma T_{i}(t)}+ \gamma\right)\geq r\left(2\sqrt{\gamma X(\ell_1)}+ \gamma\right)\]
    we have the trivial bound \(K(r) \geq r\left(2\sqrt{X(\ell_1)}+ \gamma\right)\). We now prove the case for \(r > k\). As in Lemma~\ref{lem:local-packing}, define \(g(x) := 2\sqrt{\gamma x} + \gamma\) as the function that maps the minimum play count to a lower bound on the number of plays in the epoch. A natural way to organize the possible values of the minima is to introduce a deterministic sequence
    \[
        a_0 := X(\ell_1), \quad a_{n+1} := a_n + g(a_n).
    \]
    By construction \(a_n = (\sqrt{X(\ell_1)}+n\sqrt{\gamma})^2\) (Lemma \ref{lem:exact-quadratic}) for all \(n\ge 0\).  The sequence \(\{a_n\}\) induces a natural partition of the nonnegative numbers in the \(X\)-space (the space of minimum counts) into the intervals
    \[
        I_n := [ a_n, a_n + g(a_n) ).
    \]
    These intervals can be viewed as ``levels." Because \(g\) is increasing, an epoch whose minimum lies in \(I_n\) must have epoch length at least \(g(a_n)\). Lemma~\ref{lem:local-packing} says that at most \(k\) epochs can have their minima in any particular level \(I_n\). Let \(q_n\) denote the number of the first \(r\) epochs whose minima fall in \(I_n\).
    The packing property ensures that
    \[
        0  \leq q_n  \leq k
        \qquad\text{and}\qquad
        \sum_{n\ge 0} q_n = r.
    \]
    Since each such epoch contributes at least \(g(a_n)\) periods, the total time after
    \(\ell\) epochs satisfies
    \[
        K(r)
        \ge
        \sum_{n\ge 0} q_n  g(a_n).
    \]
    To obtain the smallest possible value of the right-hand side subject to the above constraints, one must fill low levels before higher ones, because \(g(a_n)\) is increasing in \(n\). Thus the minimizing pattern is
    \[
        q_0 = q_1 = \cdots = q_{R-1} = k,
        \qquad R := \left\lfloor {r}/{k}\right\rfloor,
    \]
    with any remaining epochs (if any) placed at level \(L\) or above.  This yields the
    lower bound
    \[
        K(r)
        \ge
        k \sum_{n=0}^{R-1} g(a_n).
    \]
    Because \(a_n = (\sqrt{a_0}+n\sqrt{\gamma})^2\), the increment satisfies \(g(a_{n-1}) = 2\gamma  (n-1) + \gamma\). Thus
    \begin{align*}
        \sum_{n=0}^{R-1} g(a_n)
        &= \sum_{n=1}^{R} g(a_{n-1})
        = \sum_{n=1}^{R} \left(2\sqrt{\gamma (\sqrt{a_0}+(n-1)\sqrt{\gamma})^2} + \gamma\right)\\
        &= \sum_{n=1}^{R} 2\sqrt{\gamma a_0}+\gamma  \sum_{n=1}^{R}2(n-1)  + 1
        = 2R\sqrt{\gamma a_0}+\gamma  R^2 = (\sqrt{\gamma}R+\sqrt{a_0})^2-a_0.
    \end{align*}
    Substituting this into the previous display gives \(K(r)\ge k [(\sqrt{\gamma}R+\sqrt{a_0})^2-a_0]\). Finally, since \(R = \lfloor r/k\rfloor \ge r/k - 1\) and \(r/k - 1\geq 0\) when \(r \geq k\), we have \(R^2 \ge (r/k - 1)^2\),
    and therefore
    \[
        K(r)
        \ge
        k\left(\sqrt{\gamma}\left(\frac{r}{k} - 1\right)+\sqrt{X(\ell_1)}\right)^2-kX(\ell_1)
    \]
    for the second result.
\qed

We will frequently invoke the following corollary, which follows immediately from a special case of Lemma~\ref{lem:ell-t-bound}. 
\begin{corollary}[Inverse of \(K(r)\) with \(X(\ell_1)=0\)]\label{cor:inv_K}
    For a given sequence \(\{\ell_r\}_{r\geq 1}\) with \(X(\ell_1)=0\) that contains at most \(y\) periods, we can compute the maximum number of epochs in the sequence by inverting \(K\):
    \[\max\{r: K(r)\leq y\}\leq K^{-1}(y)=\sqrt{\frac{ky}{\gamma }}+k\]
\end{corollary}

\subsection{Final Bound on Replacement-Cost Regret}
Finally, we prove a bound on the replacement-cost regret.
\begin{lemma}[Max epochs]\label{lem:max_epochs}
    The maximum number of epochs satisfies
    \[
        L \leq \sqrt{\frac{kT}{\gamma}}+k.
    \]
\end{lemma}
\noindent {\bf Proof.}
    If \(L\leq k\), then we are done. Otherwise, let \(\{\ell_r\}_{r\geq 1}\) be the sequence of epochs \(\ell=1,\dots,r\). This sequence must satisfy
    \[L \leq \max\left\{\ell\geq k: K(r)\leq T\right\}\]
    Since \(X(\ell_1)=0\), Corollary~\ref{cor:inv_K} gives that
    \[L\leq \max\left\{k, K^{-1}(T)\right\}\leq \max\left\{k, \sqrt{\frac{kT}{\gamma }}+k\right\}= \sqrt{\frac{kT}{\gamma}}+k\]
The result thus follows. 
\qed  

\begin{lemma}[Gap-free replacement-cost regret]\label{lem:sc-bound-final}
    The replacement-cost regret satisfies
    \[
        J_T \leq cm\sqrt{\frac{kT}{\gamma}}+cmk.
    \]
\end{lemma}
\noindent {\bf Proof.}
    By construction, the replacement cost by time \(T\) is at most \(cm\) times the number of epochs that can be completed within \(T\) rounds. By Lemma \ref{lem:max_epochs},
    \[J_T \leq cm\cdot \left(\sqrt{\frac{kT}{\gamma }}+k\right)=cm\sqrt{\frac{kT}{\gamma}}+cmk.\]
The proof is then complete. 
\qed  

We can also derive an instance-dependent replacement-cost bound using Lemma~\ref{lem:ell-t-bound} together with Corollary \ref{cor:final-phase} (proved in the next section).
\begin{lemma}[Instance-dependent max epochs]\label{lem:delta-max-epoch}
    Define the minimum optimality gap \(\Delta:=\min_{i\leq m, j>m}\Delta_{i,j}\). If \(\Delta>0\), then the maximum number of epochs is bounded by
    \[\bE{L}\leq 2k\left(\left(\sum_{j>m}\frac{2}{\Delta_{m,j}}\right)\sqrt{\frac{\log T}{k\gamma}}+1+\frac{1}{k}+\frac{k}{\gamma T}+\frac{k^2}{T^2}\right)\]
    so that when \(T> k\),
    \[\bE{L}= O\left(\sum_{j>m}\frac{2}{\Delta_{m,j}}\sqrt{\frac{k\log T}{\gamma}}\right)\]
\end{lemma}
\noindent {\bf Proof.}
    Consider the sequence of epochs that contain at least one suboptimal arm, \(\{\ell_r: U_{\ell_r}\neq A^*\}\). Then \(X(\ell_1)=0\) by definition. Define the event 
    \[\mathcal{G}=\bigcap_{\ell=1}^L\left\{\LCB_i(d_{\ell})\leq \mu_i\leq\UCB_i(d_{\ell}) \text{ for all } i\in[k]\right\}\]
    Additionally, we define 
    \[
    \mathcal{G}_j(t)=\bigg\{\LCB_i(t)<\mu_i<\UCB_i(t) \text{ for all } i\leq m\bigg\} \bigcap ~\bigg\{\LCB_j(t)<\mu_j<\UCB_j(t)\bigg\}
    \]
    Clearly we have that \(\mathcal{G}\) implies \(\mathcal{G}_j(d_{\ell})\) for all \(j\) and \(d_{\ell}\). Therefore, Corollary \ref{cor:final-phase} (see Appendix \ref{sec:sampling_proof}) implies that the number of periods in the epochs where at least one suboptimal arm is chosen to be in the target set is at most
    \[\sum_{j>m}\frac{4\log T}{\Delta_{m,j}^2},\]
    after which \(U_{\ell}=[m]\). Thus, the number of epochs \(r\) that contain at least one suboptimal arm must satisfy 
    \[K(r)\leq \sum_{j>m}\frac{4\log T}{\Delta_{m,j}^2}\]
    and therefore the number of epochs under \(\mathcal{G}\) is bounded by
    \[\bE{\left(\sum_{\ell:U_{\ell}\neq A^*}|\mathcal{T}_{\ell}|+|\mathcal{B}_{\ell}|\right)\cdot \ind{\mathcal{G}}}\leq \max\left\{r: K(r)\leq \sum_{j>m}\frac{4\log T}{\Delta_{m,j}^2}\right\}\]
    Applying Corollary~\ref{cor:inv_K},
    \begin{align*}
        r&\leq K^{-1}\left(\sum_{j>m}\frac{4\log T}{\Delta_{m,j}^2}\right)=\sqrt{\frac{k}{\gamma}\sum_{j>m}\frac{4\log T}{\Delta_{m,j}^2}}+k\\
        &\leq \left(\sum_{j>m}\frac{2}{\Delta_{m,j}}\right)\sqrt{\frac{k\log T}{\gamma}}+k
    \end{align*}
    The last inequality follows from \(\|1/x\|_2\leq \|1/x\|_1\) for all positive \(x\). We conclude that 
    \[\bE{\left(\sum_{\ell:U_{\ell}\neq A^*}|\mathcal{T}_{\ell}|+|\mathcal{B}_{\ell}|\right)\cdot \ind{\mathcal{G}}}\leq \left(\sum_{j>m}\frac{2}{\Delta_{m,j}}\right)\sqrt{\frac{k\log T}{\gamma}}+k\]
    Next we consider the sequence of epochs where \(U_{\ell}=A^*\) under \(\mathcal{G}\), that is, epochs with no suboptimal arm. Recall that the epoch counter increases only when the solution \(U_{\ell}\) changes. Since every epoch in this sequence has the same solution, \(U_{\ell}= A^*\), it must be true that the epoch preceding each epoch in the sequence has \(U_{\ell}\neq A^*\) (except perhaps the initial set \(U_0\)). Thus there is an injective mapping between epochs with \(U_{\ell}=A^*\) and epochs with \(U_{\ell}\neq A^*\). Therefore, under \(\mathcal{G}\), we can also bound the number of epochs with \(U_{\ell}=A^*\) by the number of epochs with \(U_{\ell}\neq A^*\),
    \[\bE{\left(\sum_{\ell:U_{\ell}= A^*}|\mathcal{T}_{\ell}|+|\mathcal{B}_{\ell}|\right)\cdot \ind{\mathcal{G}}}\leq \left(\sum_{j>m}\frac{2}{\Delta_{m,j}}\right)\sqrt{\frac{k\log T}{\gamma}}+k\]
    Thus we have that 
    \[\bE{L\cdot \ind{\mathcal{G}}}\leq \bE{L\mid \mathcal{G}}\leq \left(\sum_{j>m}\frac{4}{\Delta_{m,j}}\right)\sqrt{\frac{k\log T}{\gamma}}+2k+2\]
    Now we count the epochs under the complement event. From Hoeffding's inequality and a union bound,
    \[\bP{\mathcal{G}^c}\leq \bE{\sum_{\ell=1}^L \frac{2k}{T^2}}=\bE{\frac{2Lk}{T^2}}\leq \left(\sqrt{\frac{kT}{\gamma}}+k\right)\frac{2k}{T^2}\]
    The last inequality is Lemma \ref{lem:max_epochs}, which states that the number of epochs that can occur in \(T\) periods is deterministically less than \(\sqrt{kT /\gamma}+k\).
    Again using Lemma \ref{lem:max_epochs} to bound \(\bE{L\mid \mathcal{G}^c}\),
    \[\bE{L\cdot \ind{\mathcal{G}^c}}=\bP{\mathcal{G}^c}\bE{L\mid \mathcal{G}^c}\leq \left(\sqrt{\frac{kT}{\gamma}}+k\right)^2\frac{2k}{T^2}\leq 2k\left(\frac{k}{\gamma T}+\frac{k^2}{T^2}\right)\]
    Thus,
    \[\bE{L}\leq 2k\left(\left(\sum_{j>m}\frac{2}{\Delta_{m,j}}\right)\sqrt{\frac{\log T}{k\gamma}}+1+\frac{1}{k}+\frac{k}{\gamma T}+\frac{k^2}{T^2}\right)\]
    so that when \(T> k\),
    \[\bE{L}= O\left(\left(\sum_{j>m}\frac{1}{\Delta_{m,j}}\right)\sqrt{\frac{k\log T}{\gamma}}\right).\]
The proof is then complete. 
\qed

Similarly, we can derive an instance-dependent replacement-cost regret.
\begin{lemma}[Instance-dependent replacement-cost regret]\label{lem:delta-sc-bound-final}
    Denote the minimum optimality gap \(\Delta:=\min_{i\leq m, j>m}\Delta_{i,j}\). If \(\Delta>0\),
    \begin{align*}
        J_T&\leq 2cmk\left(\left(\sum_{j>m}\frac{2}{\Delta_{m,j}}\right)\sqrt{\frac{\log T}{k\gamma}}+1+\frac{1}{k}+\frac{k}{\gamma T}+\frac{k^2}{T^2}\right)
    \end{align*}
    and when \(T>k\),
    \[J_T= O\left(\left(\sum_{j>m}\frac{1}{\Delta_{m,j}}\right)\cdot cm\sqrt{\frac{k\log T}{\gamma}}\right)\]
\end{lemma}
\noindent {\bf Proof.}
    By construction, the replacement cost by time \(T\) is at most \(cm\) times the number of epochs that can be completed within \(T\) rounds. By Lemma \ref{lem:delta-max-epoch},
    \[J_T\leq \bE{cmL}\leq cmk\left(\left(\sum_{j>m}\frac{1}{\Delta_{m,j}}\right)\sqrt{\frac{\log T}{k\gamma}}+1+\frac{1}{k}+\frac{2k}{\gamma T}+\frac{2k^2}{T^2}\right).\]
This completes the proof. 
\qed

\section{Bounding the Sampling Regret}\label{sec:sampling_proof}
Recall that for a fixed suboptimal arm \(j\), the regret contribution is determined by the best available optimal arm that was not chosen,
$i^*_t = \argmax_{i \in A^* \setminus A_t} \mu_i$.
A naive upper bound of the form \(\Delta_{i^*_t,j} \leq \Delta_{1,j}\) is too coarse to yield tight guarantees, since it ignores the fact that the identity of \(i^*_t\) depends on all of \(A_t\), and not just \(j\). The challenge is that, unlike in the single-arm setting, eliminating \(j\) requires comparing it not just to the best arm overall, but successively to each optimal arm. In other words, regret may persist against weaker optimal arms even after \(j\) has been statistically separated from the strongest one.

This difficulty is compounded by two sources of randomness: (i) the block threshold \eqref{eq:block-threshold} is defined by a \emph{minimum} over play counts, which makes both the start time and the length of each epoch random and history dependent; and (ii) delayed replacements mean that the set of comparisons relevant at time \(t\) is itself subject to the delay process. These features prevent a straightforward elimination argument.

To overcome this, we introduce a sequence of random intervals that align with the successive times at which arm \(j\) can be distinguished from each optimal arm \(i\). Within each interval, the regret from pulling \(j\) can be tightly bounded by the gap \(\Delta_{i,j}\), reflecting that \(j\) can only be confused with optimal arms weaker than \(i\). The problem is thus reduced to showing that the total length of these intervals is well controlled. By carefully bounding their expected duration, we obtain the sharper regret bound stated in Lemma \ref{lem:B-final}.

\subsection{Constructing Elimination Intervals}
The analysis of regret depends on how we compare the chosen suboptimal arm \(j>m\) against the best available alternative, $i_t^*$. To organize the analysis, we introduce what we call \emph{distinguishability intervals}, a sequence of random stopping intervals that partition the horizon according to when each suboptimal arm becomes statistically distinguishable from the optimal set. This construction generalizes standard arm elimination arguments used in classical UCB analysis to the delayed and multi-play setting. The intervals serve as a proof device that filters the history to isolate the moments where meaningful comparisons can be made between the chosen arm and the best available alternative. Within each interval, the gap \(\Delta_{i^*_t,j}\) can be bounded more sharply than by the trivial inequality, and the overall regret can then be controlled by bounding the lengths of these intervals.

Formally, define the stopping time
\[
t_{i,j} = \min\left\{t\in \mathbb{N} : T_j(t) > \frac{4\log T}{\Delta^2_{i,j}}\right\},
\]
with the convention \(t_{i,j}=\infty\) if the set is empty. Intuitively, \(t_{i,j}\) is the first time (under the good event) when the learner has accumulated enough samples of arm \(j\) to statistically separate it from arm \(i\); that is, the point at which arm \(j\) becomes distinguishable from arm \(i\). 

We then define the epoch index $\ell_{i,j} = \epoch(t_{i,j})$,
corresponding to the epoch that contains \(t_{i,j}\). Because the algorithm updates the UCBs once per epoch, the information that \(t_{i,j}\) has occurred will not be incorporated until epoch \(\ell_{i,j} + 1\). The resulting sequence of stopping intervals
\(\{[\ell_{i-1,j}+1, \ell_{i,j}+1)\}\)
forms the distinguishability intervals that structure the analysis throughout the subsequent regret decomposition. See Figure \ref{fig:elimination-timeline} for the timeline.

\begin{figure}[t]
    \centering
    \begin{tikzpicture}[xscale=1.2, yscale=1]
    
    \draw[thick, ->] (0,0) -- (12,0) node[right] {time};
    
    \node at (1,1) {\dots \(\ell_{i-1,j}\)};
    \node at (3.5,1) {\(\ell_{i-1,j}+1\)};
    \node at (6,1.7) {Interval \([\ell_{i-1,j}+1, \ell_{i,j}+1)\)};
    \node at (8,1) {\(\ell_{i,j}\)};
    \node at (10.8,1) {\(\ell_{i,j}+1\) \dots};
    
    \draw[fill=yellow!30] (2.5,0.2) rectangle (10.5,0.7);
    \node at (6,0.45) {\(\dots\)};
    
    \draw[fill=red!15] (0,0.2) rectangle (1,0.7);
    \node at (0.5,0.45) {\(\mathcal{T}\)};
    \draw[fill=blue!15] (1,0.2) rectangle (2.5,0.7);
    \node at (1.7,0.45) {\(\mathcal{B}\)};

    \draw[fill=yellow!30] (2.5,0.2) rectangle (3.5,0.7);
    \node at (3,0.45) {\(\mathcal{T}\)};
    \draw[fill=yellow!30] (3.5,0.2) rectangle (5,0.7);
    \node at (4.2,0.45) {\(\mathcal{B}\)};

    \draw[fill=red!15] (9.5,0.2) rectangle (10.5,0.7);
    \node at (10,0.45) {\(\mathcal{T}\)};
    \draw[fill=blue!15] (10.5,0.2) rectangle (12,0.7);
    \node at (11.2,0.45) {\(\mathcal{B}\)};

    \draw[fill=yellow!30] (7,0.2) rectangle (8,0.7);
    \node at (7.5,0.45) {\(\mathcal{T}\)};
    \draw[fill=yellow!30] (8,0.2) rectangle (9.5,0.7);
    \node at (8.7,0.45) {\(\mathcal{B}\)};
    
    \draw[dashed, thick] (2,0) -- (2,1);
    \node[below] at (2,0) {\(t_{i-1,j}\)};
    
    \draw[dashed, thick] (7.8,0) -- (7.8,1);
    \node[below] at (7.8,0) {\(t_{i,j}\)};
    
    \end{tikzpicture}
    \caption{Timeline illustrating a single distinguishability interval for arm \(j\). The interval \([\ell_{i-1,j}+1, \ell_{i,j}+1)\) is shaded yellow, and the corresponding periods \(t_{i-1,j}\) and \(t_{i,j}\) are marked.}
    \label{fig:elimination-timeline}
\end{figure}
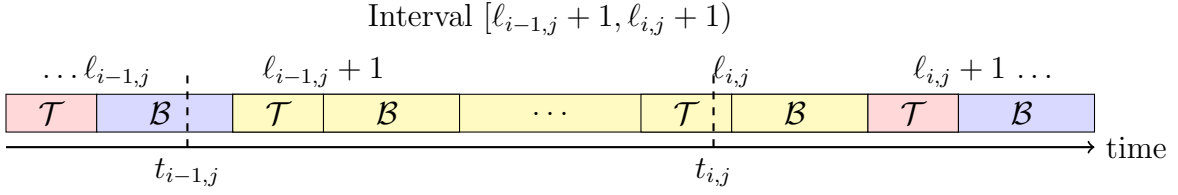

Our first lemma establishes that once sufficient evidence has accumulated to distinguish a suboptimal arm \(j\) from some optimal arm \(i\), the algorithm will never prefer \(j\) over \(i\). This parallels the classical UCB elimination lemma, but with an important caveat: because of the availability constraints and the constrained selection rule, we cannot immediately conclude that whenever \(j\) is played, arm \(i\) must also be played.

Define a ``good event''
\[
    \mathcal{G}_j(t)=\bigg\{\LCB_i(t)<\mu_i<\UCB_i(t) \text{ for all } i\leq m\bigg\} \bigcap ~\bigg\{\LCB_j(t)<\mu_j<\UCB_j(t)\bigg\}
\]
to be the event that \(j\) is well estimated and all optimal arms \(i\leq m\) (\(i\in A^*\)) are well estimated. The next step is to understand how the stopping times \(\ell_{i,j}\) are ordered as \(i\) ranges over the optimal arms. The following lemma shows that these times occur in a natural sequence, beginning at the start of the process and ending with the final elimination of \(j\).

\begin{lemma}\label{lem:ordered-intervals}
    For any suboptimal arm \(j>m\), recalling the convention \(\ell_{0,j}:=1\), we have
    \[\ell_{0,j}\leq \ell_{1,j}\leq \ell_{2,j}\leq \cdots \leq \ell_{m,j}\]
\end{lemma}
\noindent {\bf Proof.}
    Observe that \(\Delta_{1,j}>\Delta_{2,j}>\cdots >\Delta_{m,j}\) implies 
    \[\frac{4\log T}{\Delta^2_{1,j}}<\frac{4\log T}{\Delta^2_{2,j}}<\cdots < \frac{4\log T}{\Delta^2_{m,j}}\]
    Recall that 
    \[t_{i,j} = \min\left\{t\in \mathbb{N}:T_j(t)>\frac{4\log T}{\Delta^2_{i,j}}\right\}\]
    where \(T_j(t)\) is the count of times arm \(j\) has been played by time \(t\). Then we must have that 
    \[T_j(t_{1,j})\leq T_j(t_{2,j})\leq \cdot \leq T_j(t_{m,j})\]
    with equality arising from the restriction \(t\in \mathbb{N}\). Additionally, \(T_j(\cdot)\) is a count, so it is monotonically nondecreasing in its argument. This implies that 
    \[t_{1,j}\leq t_{2,j} \leq \cdots \leq t_{m,j}\]
    By the monotonicity of the \(\epoch(\cdot)\) function we have 
    \[\ell_{1,j}\leq \ell_{2,j}\leq \cdots \leq \ell_{m,j}\]
The proof is complete.
\qed  

This ordering allows us to define the distinguishability intervals. Each interval \([\ell_{i-1,j}+1, \ell_{i,j}+1)\) corresponds to a phase of the algorithm where arm \(j\) is distinguishable from the best \(i-1\) optimal arms. Once \(j\) can be separated from all \(m\) optimal arms, it is eliminated permanently. Formally, the \emph{distinguishability intervals} form a random partition of the following form:
\[
[1,L]=\left(\bigcup_{i=1}^m [\ell_{i-1,j}+1, \ell_{i,j}+1)\right)\cup \{\ell\geq \ell_{m,j}+1\}.
\]
Note that an interval may be empty if two stopping times \(\ell_{i,j}\) and \(\ell_{i+1,j}\) occur in the same epoch. The intervals thus filter the history into exactly the stages where the regret incurred by playing \(j\) can be successively bounded. This gives the needed structure to address the availability constraints. However, we still must address the constraint in the selection rule.

We want to show that if an optimal arm \(i\) has a higher upper confidence bound than \(j\) and arm \(i\) is available, then \(j\) will not be played. However, recall that the optimization problem \eqref{eq:opt} has a constraint 
\[\sum_{j\in A_t\setminus U}\delta_{j,i}(t)\geq |A_t\setminus U|\cdot \frac{c}{T-t}\]

\begin{lemma}
\label{lem:monotonicity-selection}
    Let \(i,j\) be arms such that
    \[
    \UCB_i(t)>\UCB_j(t)
    \qquad\text{and}\qquad
    \delta_{j,i}(t):=\UCB_i(t)-\LCB_j(t)\ge \frac{c}{T-t}.
    \]
    Let \(U\) be an optimal solution to \eqref{eq:opt}. Then
    \[
    j\in U \implies  i\in U.
    \]
\end{lemma}

\noindent\textbf{Proof.}
    We argue by contradiction. Suppose \(j\in U\) but \(i\notin U\). We show that in either case one can construct a feasible target set with strictly larger objective value, contradicting the optimality of \(U\). We consider the target set \(U' := \bigl(U\setminus\{j\}\bigr)\cup\{i\}\). Because \(\UCB_i(t)>\UCB_j(t)\) the objective value of \(U'\) is strictly larger than that of \(U\). Therefore, if \(U'\) is feasible we have derived a contradiction. We prove feasibility in two cases.
    
    \textit{Case 1: \(j\notin A_t\).}
    Then \(j\in U\setminus A_t\), so \(j\) is entering the target set and \(|A_t\setminus U|=|A_t\setminus U'|\). Consider the target set \( U' := \bigl(U\setminus\{j\}\bigr)\cup\{i\}. \)
    Since \(j\in U\setminus A_t\), feasibility of \(U\) implies
    \begin{align*}
        \sum_{s\in A_t\setminus U'} \delta_{s,\pi^{\mathrm R}(s)}(t)&=\sum_{s\in A_t\setminus U'}\bigg[\UCB_{\pi^{\mathrm R}(s)}(t)-\LCB_s(t)\bigg]\\
        &=\sum_{s\in A_t\setminus U}\bigg[\UCB_{\pi^{\mathrm R}(s)}(t)-\LCB_s(t)\bigg]+ \bigg(\UCB_i(t)-\UCB_j(t)\bigg)\\
        &\ge \bigg(\UCB_i(t)-\UCB_j(t)\bigg)+|A_t\setminus U|\frac{c}{T-t}\\
        &=\bigg(\UCB_i(t)-\UCB_j(t)\bigg)+|A_t\setminus U'|\frac{c}{T-t}.
    \end{align*}
    Because \(\UCB_i(t)>\UCB_j(t)\), we conclude that 
    \[\sum_{s\in A_t\setminus U'} \delta_{s,\pi^{\mathrm R}(s)}(t)\ge |A_t\setminus U'|\frac{c}{T-t}\]
    so \(U'\) is feasible. 
    
    \textit{Case 2: \(j\in A_t\).}
    In this case, \(j\in U\cap A_t\), so \(j\) is retained in the workforce and \(|A_t\setminus U|+1=|A_t\setminus U'|\). Therefore,
    \begin{align*}
        \sum_{s\in A_t\setminus U'} \delta_{s,\pi^{\mathrm R}(s)}(t)&=\sum_{s\in A_t\setminus U'}\bigg[\UCB_{\pi^{\mathrm R}(s)}(t)-\LCB_s(t)\bigg]\\
        &=\sum_{s\in A_t\setminus U}\bigg[\UCB_{\pi^{\mathrm R}(s)}(t)-\LCB_s(t)\bigg]+ \bigg(\UCB_i(t)-\LCB_j(t)\bigg)\\
        &\ge \bigg(\UCB_i(t)-\LCB_j(t)\bigg)+\sum_{s\in A_t\setminus U}\frac{c}{T-t}\\
        &\ge \frac{c}{T-t}+\sum_{s\in A_t\setminus U}\frac{c}{T-t}=(|A_t\setminus U|+1)\frac{c}{T-t}=|A_t\setminus U'|\frac{c}{T-t}.
    \end{align*}
    This proves that \(U'\) is feasible when \(j\in A_t\) as well.
    
    In both cases, the assumption \(j\in U\) and \(i\notin U\) leads to a contradiction. Therefore \(j\in U\) implies \(i\in U\).\qed 

With this result in hand, we can prove the elimination lemma.
\begin{lemma} \label{lem:gap-bound}
    For any suboptimal arm \(j>m\) and optimal arm \(i\leq m\), if \(\ell> \ell_{i,j}\) and \(t\in \mathcal{B}_{\ell}\), then under the good event \(\mathcal{G}_j(d_\ell)\) and the horizon event \(\{\Delta_{i^*_t, j}> c/(T-t)\}\), we have 
 \(\Delta_{i^*_t,j}\leq \Delta_{i,j}\). 
\end{lemma}
\noindent {\bf Proof.}
    Suppose for contradiction that \(i^*_t\leq i\). Let \(\ell > \ell_{i,j}\). Then the target solution corresponding to epoch \(\ell\) is computed at decision instant \(d_\ell\). Observe that 
    for every \(\ell\), \( d_\ell \leq  p_\ell \leq  d_{\ell+1}\),
    with equality \(p_\ell=d_{\ell +1}\) possible if the least-played arm is played \(N(\ell)\) times before \(U_{\ell}\) becomes available. Therefore,
    \[d_\ell\geq d_{\ell_{i,j}+1} \geq d_{\ell_{i^*_t,j}+1}\geq t_{i^*_t,j}\]
    Since the play count of arm \(j\), \(T_j(\cdot)\), is monotonically nondecreasing in its argument, we must have \(T_j(d_\ell)\geq T_j(t_{i^*_t,j})\). 
    Plugging in the definition of \(t_{i^*_t,j}\) gives
    \[T_j(d_\ell)\geq T_j(t_{i^*_t,j})>\frac{4\log T}{\Delta^2_{i^*_t,j}}\]
    It then follows from Lemma~\ref{lem:arm-comp} that under the event \(\mathcal{G}_j(d_\ell)\), we have \(\UCB_j(d_\ell)<\UCB_{i^*_t}(d_\ell)\).
    This implies that replacing \(j\) with \(i^*_t\) increases the objective. Moreover, under \(\mathcal{G}_j(d_{\ell})\) we have \(\delta_{i^*_t,j}\geq \Delta_{i^*_t,j}\). Then on the event \(\Delta_{i^*_t, j}> c/(T-t)\) we conclude
    \[(T-t)\delta_{i^*_t,j}\geq (T-t)\Delta_{i^*_t,j}\geq c\]
    Since \(\UCB_j(d_\ell)<\UCB_{i^*_t}(d_\ell)\) and \((T-t)\delta_{i^*_t,j}\geq c\), Lemma \ref{lem:monotonicity-selection} yields that \(i^*_t\) must be in the target set, that is, \(i^*_t\in U_{\ell}\). Moreover, because \(t\in \mathcal{B}_\ell\), the arms in \(U_{\ell}\) are active by definition. However, this implies that the algorithm plays arm \(i^*_t\), contradicting the definition of \(i^*_t\) as the best unplayed arm. We conclude that we must have \(i^*_t>i\). It then follows from the ordering of means \(\mu_1 > \mu_2 > \dots > \mu_k\) that
    \[i^*_t>i\implies \Delta_{i^*_t,j}\leq \Delta_{i,j}\]
The result thus follows. 
\qed  

A direct corollary of Lemma \ref{lem:gap-bound} is that every suboptimal arm is eventually eliminated from the target solution once it becomes distinguishable from the worst optimal arm.
\begin{corollary} \label{cor:final-phase}
    For any suboptimal arm \(j>m\), if \(\ell> \ell_{m,j}\), then under the good event \(\mathcal{G}_j(d_\ell)\), we have \(j\not\in U_{\ell}\).
\end{corollary}
\noindent {\bf Proof.}
    Suppose for contradiction that \(j\in U_\ell\). This implies that there exists \(i\leq m\) that is not in \(U_{\ell}\). This means that \(i^*_t\leq m\), which contradicts Lemma \ref{lem:gap-bound}.
\qed  

\subsection{Bounding the Regret From the Block Phase}
We begin with the contribution to regret arising during the block phases. To do this, we condition on \(\mathcal{G}_j(d_\ell)\) for each \(1\leq \ell \leq L\).
\begin{eqnarray*}
   \bE{\sum_{\ell =1}^L \sum_{t\in \mathcal{B}_\ell}\Delta_{i^*_t, j} \ind{j\in U_{\ell}}}
    &\leq& \bE{\sum_{\ell =1}^L \sum_{t\in \mathcal{B}_\ell}\Delta_{i^*_t, j} \ind{j\in U_{\ell},\mathcal{G}_j(d_\ell)}} \\
    &&+ \bE{\sum_{\ell =1}^L \sum_{t\in \mathcal{B}_\ell}\Delta_{i^*_t, j} \ind{j\in U_{\ell},\mathcal{G}^c_j(d_\ell)}}
\end{eqnarray*}
We begin by bounding the complement event.

\begin{lemma}\label{lem:block-complement-regret}
For any \(j\in[k]\), let \(\mathcal{B}_\ell\) denote the block phase in epoch \(\ell\). Then, for \(k\ge 2\),
\[
    \bE{\sum_{\ell =1}^L \sum_{t\in \mathcal{B}_\ell}\Delta_{i^*_t, j} \ind{j\in U_{\ell},\mathcal{G}^c_j(d_\ell)}} 
    \le
    \frac{4m\Delta_{1,j}}{T}.
\]
In particular, when \(T>m\) the contribution of \(\mathcal{G}_j^c\) is
\(O\!\left(1\right)\).
\end{lemma}
\noindent {\bf Proof.}
    Note that Lemma \ref{lem:Gc} gives \(\bP{\mathcal{G}_j^c(d_{\ell})}\leq 4m/T^2\) for any \(j\) and any realization of \(d_\ell\). Thus,
    \begin{align*}
        \bE{\sum_{\ell =1}^L \sum_{t\in \mathcal{B}_\ell}\Delta_{i^*_t, j} \ind{j\in U_{\ell},\mathcal{G}^c_j(d_\ell)}} 
        &\leq \bE{\sum_{\ell =1}^L \sum_{t\in \mathcal{B}_\ell}\Delta_{i^*_t, j} \ind{\mathcal{G}^c_j(d_\ell)}} \\
        &\leq \Delta_{1,j}\bE{\sum_{\ell =1}^L \sum_{t\in \mathcal{B}_\ell} \ind{\mathcal{G}^c_j(d_\ell)}} \\
        &\leq  \Delta_{1,j}T\left(\frac{4m}{T^2}\right) \leq \frac{4m\Delta_{1,j}}{T}.
    \end{align*}
The proof is complete. 
\qed

Now we bound the sampling regret under the good event. By construction of the distinguishability intervals, we can decompose the total block phase regret as
\begin{align*}
    \bE{\sum_{\ell =1}^L \sum_{t\in \mathcal{B}_\ell}\Delta_{i^*_t, j} \ind{j\in U_{\ell},\mathcal{G}_j(d_\ell)}}
    &= \sum_{i=1}^m \bE{\sum_{\ell\in [\ell_{i-1,j}+1, \ell_{i,j}+1)} \sum_{t\in \mathcal{B}_\ell}\Delta_{i^*_t, j} \ind{j\in U_{\ell},\mathcal{G}_j(d_\ell)}} \\
    &\quad + \bE{\sum_{\ell\geq \ell_{m,j}+1} \sum_{t\in \mathcal{B}_\ell}\Delta_{i^*_t, j} \ind{j\in U_{\ell},\mathcal{G}_j(d_\ell)}}.
\end{align*}
The first term corresponds to the \(m\) distinguishability intervals, while the second term accounts for epochs that occur after the last stopping time \(\ell_{m,j}\). 
We analyze these two parts separately.

\begin{lemma}\label{lem:arm-elimination}
    For any \(\ell \geq \ell_{m,j}+1\),
    \[
        \bE{\sum_{t\in \mathcal{B}_\ell}\Delta_{i^*_t, j} \ind{j\in U_{\ell},\mathcal{G}_j(d_\ell)}}=0.
    \]
\end{lemma}
\noindent {\bf Proof.}
    Once \(\ell_{m,j}\) is reached, arm \(j\) has accumulated enough evidence to be statistically distinguished from all optimal arms. It follows from Corollary \ref{cor:final-phase} that \(\ell \geq \ell_{m,j}+1>\ell_{m,j}\) implies \(j\not\in U_\ell\) under \(\mathcal{G}_j(d_\ell)\). Thus 
    \[\ind{j\in U_{\ell},\mathcal{G}_j(d_\ell)}=0\]
    which proves the result.
\qed  

Thus it suffices to restrict attention to the first \(m\) intervals. 
Within such an interval, the suboptimality gap can be controlled uniformly by Lemma \ref{lem:gap-bound}. This lemma shows that the regret per play of \(j\) during block phases in the \(i\)-th interval is at most \(\Delta_{i,j}\). 
What remains is to bound the \emph{number of such plays}. 

\begin{lemma}\label{lem:B-final}
    For any suboptimal arm \(j > m\), we have
    \begin{eqnarray*}
        &&\bE{\sum_{i=1}^m \sum_{\ell \in [\ell_{i-1,j}+1, \ell_{i,j}+1)}\sum_{t \in \mathcal{B}_\ell}\Delta_{i^*_t,j}\ind{j \in U_{\ell}, \mathcal{G}_j(d_\ell)}}
        \\
        &\leq& \frac{8\log T}{\Delta_{m,j}}+\left(2\sqrt{\gamma \cdot 4\log T}\right)\left[1+\log\left(\frac{\Delta_{1,j}}{\Delta_{m,j}}\right)\right] + \Delta_{1,j}(\gamma + 1)
    \end{eqnarray*}
\end{lemma}
\noindent {\bf Proof.}
    First, we apply the previous lemma to conclude \(\Delta_{i^*_t,j}\leq \Delta_{i,j}\). Next, define the set of all periods \(t\) that belong to an epoch in the distinguishability interval.
    \begin{equation}\label{eq:block-set}
        \mathcal{S}^{\mathcal B}_{i,j}
        :=
        \{t : \ell_{i-1,j}+1 \leq \epoch(t) \leq \ell_{i,j},\ t \in \mathcal{B}_{\epoch(t)},\ j \in A_t\}.
    \end{equation}
    Since \(A_t = U_\ell\) for all \(t \in \mathcal{B}_\ell\), we can rewrite
    \[
        \sum_{i=1}^m\bE{\sum_{\ell \in [\ell_{i-1,j}+1, \ell_{i,j}+1)}\sum_{t \in \mathcal{B}_\ell}\Delta_{i^*_t,j}\ind{j \in U_\ell, \mathcal{G}_j(d_\ell)}}
        = \sum_{i=1}^m\Delta_{i,j}\bE{\sum_{t \in \mathcal{S}^{\mathcal B}_{i,j}}\ind{\mathcal{G}_j(d_{\epoch(t)})}}.
    \]
    Thus, it suffices to bound the contribution from the set \(\mathcal{S}^{\mathcal B}_{i,j}\). Observe that
    \[
        \mathcal{S}^{\mathcal B}_{i,j} = \{t : d_{\ell_{i-1,j}+1} \leq t < d_{\ell_{i,j}+1}, t \in \mathcal{B}_{\epoch(t)}, j\in A_t\}.
    \] 
    Intuitively, the set \(\mathcal{S}^{\mathcal B}_{i,j}\) consists of the distinguishability interval shifted to the right by some ``extra pulls'' determined by the maximum block length (see Figure \ref{fig:elimination-timeline}). This motivates us to define 
    \[\tilde{B}_{i,j}:= \{t\in \mathcal{B}_{\epoch(t)} : t_{i,j} \leq t < d_{\ell_{i,j}+1}, j\in A_t\} \]
    to be the ``extra" plays in the block phase after \(t_{i,j}\) until the epoch completes. Then,
    \begin{align*}
        \mathcal{S}^{\mathcal B}_{i,j}
        &= \Big(\{t\in \mathcal{B}_{\epoch(t)} : t_{i-1,j} \leq t < t_{i,j}, j\in A_t\} \cup \tilde{B}_{i,j}\Big) \setminus \tilde{B}_{i-1,j}
    \end{align*}
    We can then bound the size of this set
    \[
        |\mathcal{S}^{\mathcal B}_{i,j}| = |\{t\in \mathcal{B}_{\epoch(t)} : t_{i-1,j} \leq t < t_{i,j}, j\in A_t\}| + |\tilde{B}_{i,j}| -  |\tilde{B}_{i-1,j}|
    \]
    
    Therefore,
    \begin{align*}
        &\sum_{i=1}^m\bE{\sum_{t \in \mathcal{S}^{\mathcal B}_{i,j}}\ind{\mathcal{G}_j(d_{\epoch(t)})}}\\
        &\quad \leq \sum_{i=1}^m\bE{|\{t\in \mathcal{B}_{\epoch(t)} : t_{i-1,j} \leq t < t_{i,j}, j\in A_t\}|\ind{ \mathcal{G}_j(d_{\epoch(t)})}}\\
        &\qquad + \Delta_{1,j}\bE{|\tilde{B}_{i,j}|\cdot \ind{\mathcal{G}_j(d_{\ell_{i,j}})}}+ \sum_{i=2}^m \Delta_{i,j}\left(\bE{|\tilde{B}_{1,j}|\cdot \ind{\mathcal{G}_j(d_{\ell_{1,j}})}} -  \bE{|\tilde{B}_{i,j}|\cdot \ind{\mathcal{G}_j(d_{\ell_{i-1,j}})}}\right)\\
    \end{align*}
    It follows from the definition of \(t_{i,j}\) that for \(i \geq 2\),
    \[\bE{|\{t\in \mathcal{B}_{\epoch(t)} : t_{i-1,j} \leq t < t_{i,j}, j\in A_t\}|\ind{ \mathcal{G}_j(d_{\epoch(t)})}}\leq \left(\frac{4\log T}{\Delta_{i,j}^2} - \frac{4\log T}{\Delta_{i-1,j}^2}\right)\]
    and 
    \[\bE{|\{t\in \mathcal{B}_{\epoch(t)} : t_{0,j} \leq t < t_{1,j}, j\in A_t\}|\cdot \ind{ \mathcal{G}_j(d_{\epoch(t)})}}\leq \frac{4\log T}{\Delta_{1,j}^2}\]
    From Lemma \ref{lem:technical-1} we have 
    \begin{equation}\label{eq:block-pull-bound-1}
        \frac{4\log T}{\Delta_{1,j}}+\sum_{i=2}^m \Delta_{i,j}\left(\frac{4\log T}{\Delta_{i,j}^2} - \frac{4\log T}{\Delta_{i-1,j}^2}\right)\leq \frac{8\log T}{\Delta_{m,j}}
    \end{equation}
    Next we bound the terms involving \(\tilde{B}_{i,j}\). To reduce clutter, let us denote 
    \[\mathcal{B}_{i,j}:=\bE{|\tilde{B}_{i,j}|\cdot \ind{\mathcal{G}_j(d_{\ell_{i,j}})}}\]
    Then we bound \(\mathcal{B}_{i,j}\). The length of the block phase can be at most \(N(\ell)\). This occurs when the arm \(i_{\min}:=\argmin_{i\in U_\ell} T_i(d_\ell)\) is the last arm to become available. Thus,
    \begin{align*}
        \mathcal{B}_{i,j} 
        &\leq N(\ell_{i,j})  
        = \min_{j' \in U_{\ell_{i,j}}} \Big\lceil 2\sqrt{\gamma \cdot T_{j'}(d_{\ell_{i,j}})} + \gamma \Big\rceil \\
        &\leq \Big\lceil 2\sqrt{\gamma \cdot T_j(d_{\ell_{i,j}})} + \gamma \Big\rceil
        \leq 2\sqrt{\gamma \cdot T_j(d_{\ell_{i,j}})} + \gamma + 1,
    \end{align*}
    where the first inequality follows from $j \in U_{\ell_{i,j}}$. Finally, since \(d_{\ell_{i,j}} \leq t_{i,j}\), by the definition of \(t_{i,j}\) we have
    \[
        T_j(d_{\ell_{i,j}}) \leq \frac{4 \log T}{\Delta_{i,j}^2}.
    \]
    Substituting this bound gives
  $$  \mathcal{B}_{i,j}\leq 2\sqrt{\gamma \cdot T_j(d_{\ell_{i,j}})} + \gamma + 1
        \leq 2\sqrt{\gamma \cdot \frac{4\log T}{\Delta_{i,j}^2}} + \gamma + 1 $$
    Therefore we can derive the regret contribution
    \begin{eqnarray*}
       && \Delta_{1,j}\mathcal{B}_{1,j}+ \sum_{i=2}^m \Delta_{i,j}\left(\mathcal{B}_{i,j} -  \mathcal{B}_{i-1,j}\right)\\
       &=&\Delta_{m,j}\mathcal{B}_{m,j}+ \sum_{i=1}^{m-1} \left(\Delta_{i,j}-\Delta_{i+1,j}\right)\mathcal{B}_{i,j}
        =\Delta_{m,j}\mathcal{B}_{m,j}+ \sum_{i=1}^{m-1} \Delta_{i,i+1}\mathcal{B}_{i,j}\\
        &\leq& \Delta_{m,j}\left(2\sqrt{\gamma \cdot \frac{4\log T}{\Delta_{m,j}^2}} + \gamma + 1\right)+ \sum_{i=1}^{m-1} \Delta_{i,i+1}\left(2\sqrt{\gamma \cdot \frac{4\log T}{\Delta_{i,j}^2}} + \gamma + 1\right)\\
    \end{eqnarray*}
    The square root term of the second term in the sum can be bound
    \begin{align*}
        \sum_{i=1}^{m-1} \Delta_{i,i+1}\left(2\sqrt{\gamma \cdot \frac{4\log T}{\Delta_{i,j}^2}}\right)
        &= \left(2\sqrt{\gamma \cdot 4\log T}\right)\sum_{i=1}^{m-1} \frac{\Delta_{i,i+1}}{\Delta_{i,j}}\\
        &\leq\left(2\sqrt{\gamma \cdot 4\log T}\right)\log\left(\frac{\Delta_{1,j}}{\Delta_{m,j}}\right). \tag{Lemma \ref{lem:sum-of-ratios}}
    \end{align*}
    Additionally, we can bound the \(\gamma + 1\) term in the sum
    \[\sum_{i=1}^{m-1} \Delta_{i,i+1}(\gamma + 1)=(\gamma+1)\sum_{i=1}^{m-1} \Delta_{i,i+1}=\Delta_{1,m}(\gamma+1)\]
    Putting everything together we can bound the regret contribution from the \(\mathcal{B}_{i,j}\) terms
    \begin{align*}
        \Delta_{1,j}\mathcal{B}_{1,j}+ \sum_{i=2}^m \Delta_{i,j}\left(\mathcal{B}_{i,j} -  \mathcal{B}_{i-1,j}\right)&\leq \left(2\sqrt{\gamma \cdot 4\log T}\right)\left[1+\log\left(\frac{\Delta_{1,j}}{\Delta_{m,j}}\right)\right] +(\Delta_{m,j}+\Delta_{1,m})(\gamma + 1)\\
        &= \left(2\sqrt{\gamma \cdot 4\log T}\right)\left[1+\log\left(\frac{\Delta_{1,j}}{\Delta_{m,j}}\right)\right] + \Delta_{1,j}(\gamma + 1).
    \end{align*}
    
    Together with equation (\ref{eq:block-pull-bound-1}), we get the claimed bound:
    \begin{align*}
        &\bE{\sum_{i=1}^m \sum_{\ell \in [\ell_{i-1,j}+1, \ell_{i,j}+1)}\sum_{t \in \mathcal{B}_\ell}\Delta_{i^*_t,j}\ind{j \in U_{\ell}, \mathcal{G}_j(d_\ell)}}
        \\
        &\qquad \leq \frac{8\log T}{\Delta_{m,j}}+\left(2\sqrt{\gamma \cdot 4\log T}\right)\left[1+\log\left(\frac{\Delta_{1,j}}{\Delta_{m,j}}\right)\right] + \Delta_{1,j}(\gamma + 1).
    \end{align*}
The proof is complete. 
\qed  

\subsection{Bounding the Regret From the Transition Phase}
We control the regret from the transition phase by limiting the number of epochs, and thus the number of transition periods. Since the expected number of periods in the transition phase is bounded by \(\bar{\omega}\), and the regret in each period is at most \(m\Delta_{1,k}\leq m\), the same techniques used to bound the replacement costs can be used to control the sampling regret from transition phases. 

\begin{lemma}[Gap-free regret from transitions]\label{lem:transition-regret}
    \[\bE{\sum_{t=1}^T\Delta_{i^*_t, j}\ind{j\in A_t, t\in \mathcal{T}_{\epoch(t)}}}\leq \bar{\omega}m\sqrt{\frac{kT}{\gamma}}+\bar{\omega}mk.\]
\end{lemma}
\noindent {\bf Proof.}
    In each epoch, the expected length of each delay is upper bounded by \(\bar{\omega}\). Therefore, the expected regret incurred in each transition phase is at most \(\bar{\omega} m\) since \(\Delta_{i,j}\in [0,1]\) for all \(i,j\) and \(\omega_{i,j}^{(t)}\) is independent. From Lemma \ref{lem:max_epochs}, we have that the regret from transitions is bounded by
    \[\bar{\omega}mL\leq \bar{\omega}m\left(\sqrt{\frac{kT}{\gamma}}+k\right)\]
    proving the statement.
\qed  

\begin{lemma}[Instance-dependent regret from transitions]\label{lem:delta-transition-regret}
    Let \(\Delta:=\min_{i\leq m, j>m}\Delta_{i,j}\). Then if \(\Delta>0\),
    \[\bE{\sum_{t=1}^T\Delta_{i^*_t, j}\ind{j\in A_t, t\in \mathcal{T}_{\epoch(t)}}}\leq 2\bar{\omega}mk\left(\left(\sum_{j>m}\frac{2}{\Delta_{m,j}}\right)\sqrt{\frac{\log T}{k\gamma}}+1+\frac{1}{k}+\frac{k}{\gamma T}+\frac{k^2}{T^2}\right)\]
    and when \(T>k\),
    \[\bE{\sum_{t=1}^T\Delta_{i^*_t, j}\ind{j\in A_t, t\in \mathcal{T}_{\epoch(t)}}} = O\!\left(\left(\sum_{j>m}\frac{1}{\Delta_{m,j}}\right)\cdot \bar{\omega}m\sqrt{\frac{k\log T}{\gamma}}\right)\]
\end{lemma}
\noindent {\bf Proof.}
    In each epoch, the transition phase can incur regret of at most \(\bar{\omega} m\), since \(\Delta_{i,j}\in [0,1]\) for all \(i,j\). From Lemma \ref{lem:delta-max-epoch}, when \(\Delta>0\) we have that the regret from transitions is bounded by
    \[\bE{\bar{\omega}mL}\leq 2\bar{\omega}mk\left(\left(\sum_{j>m}\frac{2}{\Delta_{m,j}}\right)\sqrt{\frac{\log T}{k\gamma}}+1+\frac{1}{k}+\frac{k}{\gamma T}+\frac{k^2}{T^2}\right)\]
    This proves the statement.
\qed

\subsection{Statement of the Instance-Independent Sampling Regret}
By combining the results from the block and transition phases, we obtain an expression for the total sampling regret that can be transformed into an instance-independent regret bound via a standard argument.

\begin{lemma}\label{lem:instance-independent-regret}
    The instance-independent sampling regret is bounded by
    \begin{align*}
        S_T&\leq 5\sqrt{m(k-m)T\log T}\\
        &\quad +~ (c+1)(k-m) +(k-m)c\log\left(\frac{cmT}{(k-m)\log T}\right)\\
        &\quad\;+\;2(k-m)\sqrt{\gamma \cdot 4\log T}\left(1+\frac{1}{2}\log\left(\frac{mT}{(k-m)\log T}\right)\right)+k(\gamma + 1)\\
        &\quad + ~\bar{\omega} m\sqrt{\frac{kT}{\gamma}} + \bar{\omega}mk
        \quad + ~ k(\gamma + 1) + \frac{4mk}{T}
    \end{align*}
\end{lemma}
\noindent {\bf Proof.}
    By Lemma \ref{lem:B-final} and Lemma \ref{lem:transition-regret} the sampling regret under the events \(\{\Delta_{i^*_t,j}>c/(T-t)\}\) is bounded by
    \begin{align*}
        S_T&\leq \sum_{j>m} \left[\frac{8\log T}{\Delta_{m,j}}+ \left(2\sqrt{\gamma \cdot 4\log T}\right)\left(1+\log\left(\frac{\Delta_{1,j}}{\Delta_{m,j}}\right)\right)\right]\\
        & \quad + ~\sum_{j>m} \left[\Delta_{1,j}(\gamma + 1) + \frac{4m\Delta_{1,j}}{T}\right]\\
        & \quad\;+\; \bar{\omega} m\sqrt{\frac{kT}{\gamma}} + \bar{\omega}mk
    \end{align*}
    Now let us split the arms based on whether 
    \[\Delta_{m,j}< \epsilon:=\sqrt{\frac{(k-m)\log T}{mT}}\]
    Then the sampling regret can be decomposed
    \begin{align*}
        S_T&\leq \sum_{t=1}^T\sum_{j\in A_t\setminus A^*}\Delta_{m,j}\ind{\Delta_{m,j}<\epsilon }\\
        &\quad\;+\;\sum_{j:\Delta_{m,j}\geq\epsilon} \left[\frac{8\log T}{\Delta_{m,j}}+ \left(2\sqrt{\gamma \cdot 4\log T}\right)\left(1+\log\left(\frac{\Delta_{1,j}}{\Delta_{m,j}}\right)\right)\right]\\
        & \quad + ~\sum_{j:\Delta_{m,j}\geq\epsilon} \left[\Delta_{1,j}(\gamma + 1) + \Delta_{1,j}\frac{4m}{T}\right]\\
        & \quad\;+\; \bar{\omega} m\sqrt{\frac{kT}{\gamma}} + \bar{\omega}mk
    \end{align*}
    The first sum is bounded
    \[\sum_{t=1}^T\sum_{j\in A_t\setminus A^*}\Delta_{m,j}\ind{\Delta_{m,j}<\epsilon }\leq mT\sqrt{\frac{(k-m)\log T}{mT}}=\sqrt{m(k-m)T\log T}\]
    The second is bounded
    \begin{align*}
    &\sum_{j:\Delta_{m,j}\geq\epsilon}\left[\frac{4\log T}{\Delta_{m,j}}+2\sqrt{\gamma \cdot 4\log T}\left(1+\log\left(\frac{\Delta_{1,j}}{\Delta_{m,j}}\right)\right)\right]\\
    &\qquad \leq (k-m)\left[\frac{4\log T}{\epsilon}+2\sqrt{\gamma \cdot 4\log T}\left(1+\log\left(\frac{\Delta_{1,j}}{\epsilon}\right)\right)\right]\\
    &\qquad = (k-m)\left[4\sqrt{\frac{mT\log T}{k-m}}+2\sqrt{\gamma \cdot 4\log T}\left(1+\log\left(\sqrt{\frac{mT}{(k-m)\log T}}\right)\right)\right]\\
    &\qquad = 4\sqrt{m(k-m)T\log T}+2(k-m)\sqrt{\gamma \cdot 4\log T}\left(1+\frac{1}{2}\log\left(\frac{mT}{(k-m)\log T}\right)\right).
    \end{align*}
    The third sum is bounded
    \[\sum_{j:\Delta_{m,j}\geq\epsilon} \left[\Delta_{1,j}(\gamma + 1) + \Delta_{1,j}\frac{4m}{T}\right]\leq k(\gamma + 1) + \frac{4mk}{T}.\]
    Finally we add the contribution from the events \(\{\Delta_{i^*_t,j}\leq c/(T-t)\}\),
    \[(c+1)(k-m) +\sum_{j>m}c\log\left(\frac{c}{\Delta_{m,j}}\right).\]
    Plugging in \(\Delta_{m,j}<\epsilon\),
    \[(c+1)(k-m) +(k-m)c\log\left(\frac{cmT}{(k-m)\log T}\right).\]
    Then combining terms gives the result.
\qed

\section{Proof of Instance-Dependent Regret Bound}\label{sec:delta-all-constants}
\begin{theorem}
    Denote \(\Delta:=\min_{i\leq m, j>m}\Delta_{i,j}\). If \(\Delta>0\), then the regret is bounded by
    \begin{align*}
        \Regret(T)&\leq \left(\sum_{j>m}\frac{4}{\Delta_{m,j}}\right) \left( 2\log T+(\bar{\omega}+c)m\sqrt{\frac{k\log T}{\gamma }}\right)\\
        &\quad +~ \sum_{j>m} \left[\left(2\sqrt{\gamma \cdot 4\log T}\right)\left(1+\log\left(\frac{\Delta_{1,j}}{\Delta_{m,j}}\right)\right)+c\log\left(\frac{c}{\Delta_{m,j}}\right)\right]\\
        & \quad + ~\sum_{j>m} \Delta_{1,j} \left[(\gamma + 1) + \frac{4m}{T}\right]\\
        &\quad +~2(\bar{\omega}+c)mk\left(1+\frac{1}{k}+\frac{k}{\gamma T}+\frac{k^2}{T^2}\right) + (c+1)(k-m)
    \end{align*}
    In particular
    \begin{align*}
        \limsup_{T\to\infty}\frac{\Regret(T)}{\log T}&\leq \sum_{j>m} \frac{8}{\Delta_{m,j}}.
    \end{align*}
\end{theorem}
\noindent{\bf Proof.}
    Combining the terms of Lemma \ref{lem:B-final}, Lemma \ref{lem:delta-transition-regret}, and Lemma \ref{lem:delta-sc-bound-final} with the decomposition derived in Appendix \ref{sec:proof-outline}, we obtain
    \begin{align*}
        \Regret(T)&\leq  \sum_{j>m} \left[\frac{8\log T}{\Delta_{m,j}}+ \left(2\sqrt{\gamma \cdot 4\log T}\right)\left(1+\log\left(\frac{\Delta_{1,j}}{\Delta_{m,j}}\right)\right)\right]\\
        & \quad + ~\sum_{j>m} \left[\Delta_{1,j}(\gamma + 1) + \frac{4m\Delta_{1,j}}{T}\right]\\
        &\quad +~ (c+1)(k-m) +\sum_{j>m}c\log\left(\frac{c}{\Delta_{m,j}}\right)\\
        &\quad + ~ 2\bar{\omega}mk\left(\left(\sum_{j>m}\frac{2}{\Delta_{m,j}}\right)\sqrt{\frac{\log T}{k\gamma}}+1+\frac{1}{k}+\frac{k}{\gamma T}+\frac{k^2}{T^2}\right)\\
        &\quad +~2cmk\left(\left(\sum_{j>m}\frac{2}{\Delta_{m,j}}\right)\sqrt{\frac{\log T}{k\gamma}}+1+\frac{1}{k}+\frac{k}{\gamma T}+\frac{k^2}{T^2}\right)
    \end{align*}
    Simplifying terms we arrive at the full constants bound,
    \begin{align*}
        \Regret(T)&\leq \left(\sum_{j>m}\frac{4}{\Delta_{m,j}}\right) \left( 2\log T+(\bar{\omega}+c)m\sqrt{\frac{k\log T}{\gamma }}\right)\\
        &\quad +~ \sum_{j>m} \left[\left(2\sqrt{\gamma \cdot 4\log T}\right)\left(1+\log\left(\frac{\Delta_{1,j}}{\Delta_{m,j}}\right)\right)+c\log\left(\frac{c}{\Delta_{m,j}}\right)\right]\\
        & \quad + ~\sum_{j>m} \Delta_{1,j} \left[(\gamma + 1) + \frac{4m}{T}\right]\\
        &\quad +~2(\bar{\omega}+c)mk\left(1+\frac{1}{k}+\frac{k}{\gamma T}+\frac{k^2}{T^2}\right) + (c+1)(k-m)
    \end{align*}
    When \(T>k\), there exist constants \(C_1, C_2, C_3, C_4\) such that
    \begin{align*}
        \Regret(T)&\leq C_1\sum_{j>m}\frac{1}{\Delta_{m,j}} \left( \log T+(\bar{\omega}+c)m\sqrt{\frac{k\log T}{\gamma }}\right)\\
        &\quad +~ C_2\sum_{j>m} \left[\sqrt{\gamma \log T\cdot \polylog\left(\frac{\Delta_{1,j}}{\Delta_{m,j}}\right)}+c\log\left(\frac{c}{\Delta_{m,j}}\right)\right]\\
        & \quad + ~C_3\sum_{j>m} \Delta_{1,j} \left[(\gamma + 1) + \frac{m}{T}\right]\\
        &\quad +~C_4(\bar{\omega}+c)mk
    \end{align*}
    The proof is thus complete. 
\qed

\section{Proof of Lower Bounds}\label{sec:lower_bounds}
In this section, we derive instance-dependent and instance-independent lower bounds for the replacement bandit problem. We begin with a general reduction showing that delayed and costly replacements cannot make a bandit problem easier. As a result, any lower bound for a base problem class transfers immediately to the corresponding class with delayed and costly actions.

\begin{lemma}[Reduction to the free and instantaneous class]
\label{lem:reduction-lower-bound}
    Let \(\mathcal{C}\) be a class of bandit environments, and let \(\mathcal{C}^D\) denote the corresponding class obtained by equipping each environment in \(\mathcal{C}\) with delayed actions and replacement cost \(c\ge 0\), while leaving the reward-generating process unchanged. Then, for every delayed environment \(\nu^D \in \mathcal{C}^D\) with corresponding base environment \(\nu \in \mathcal{C}\), and for every policy \(\pi^D\) for \(\nu^D\), there exists a policy \(\pi\) for \(\nu\) such that, for every horizon \(T\),
    \[
        \Regret_{\nu}^{\pi}(T)\le \Regret_{\nu^D}^{\pi^D}(T).
    \]
    Consequently,
    \[
        \inf_{\pi^D}\sup_{\nu^D\in \mathcal{C}^D}\Regret_{\nu^D}^{\pi^D}(T)
        \;\ge\;
        \inf_{\pi}\sup_{\nu\in \mathcal{C}}\Regret_{\nu}^{\pi}(T).
    \]
    In particular, any instance-dependent or instance-independent lower bound for \(\mathcal{C}\) also applies to \(\mathcal{C}^D\).
\end{lemma}

{\bf Proof.}
    Fix \(\nu^D\in\mathcal{C}^D\), let \(\nu\in\mathcal{C}\) be its corresponding base environment, and let \(\pi^D\) be any policy for \(\nu^D\). Write \(A_t^D\) for the active action selected by the delayed system in period \(t\), that is, the action that is actually implemented and generates rewards at time \(t\).
    
    We construct a policy \(\pi\) for the base environment \(\nu\) by having \(\pi\) play the same action \(A_t^D\) in each period \(t\). This clearly results in a feasible policy. Couple the two environments so that the reward generated by any implemented action is the same in \(\nu\) and \(\nu^D\). Under this coupling, the sequence of reward-generating actions is identical in the two systems, so the realized cumulative reward under \(\pi\) in \(\nu\) is exactly the same as the realized cumulative reward under \(\pi^D\) in \(\nu^D\). Since the reward model is unchanged, the oracle benchmark is also the same in both environments. Therefore, the sampling component of regret is identical.
    
    The delayed environment \(\nu^D\) may incur additional loss due to replacement costs \(c\geq 0\), whereas the base environment \(\nu\) does not. It follows that
    \[
        \Regret_{\nu}^{\pi}(T)\le \Regret_{\nu^D}^{\pi^D}(T).
    \]
    Because this holds for every \(\pi^D\), taking the supremum over \(\nu^D\in\mathcal{C}^D\) and then the infimum over policies yields the stated minimax inequality. The final claim follows immediately. \qed

With this reduction lemma, we can now prove the lower bounds for delayed replacement bandit policies.

\smallskip
\noindent {\bf Proof of Proposition \ref{prop:asymptoptic-lower}.}
    The unique optimal action is \(A^* = \{1,\dots,m\}\), and every suboptimal arm \(i > m\) has gap
    \[
    \Delta_i = \mu_m - \mu_i = \Delta.
    \]
    By the classical asymptotic lower bound for stochastic multiple-play bandits (Theorem 3.1 of \citealt{anantharam2003asymptotically}), any consistent policy satisfies
    \[
    \liminf_{T\to\infty}\frac{\Regret(T)}{\log T}
    \ge
    \sum_{i=m+1}^k \frac{\mu_m-\mu_i}{\kl(\mu_i,\mu_m)}.
    \]
    Applying this bound to the above instance yields
    \[
    \liminf_{T\to\infty}\frac{\Regret(T)}{\log T}
    \ge
    (k-m)\frac{\Delta}{\kl(1/2-\Delta,\,1/2)}.
    \]
    Now use the standard inequality
    \[
    \kl(p,q) \le \frac{(p-q)^2}{q(1-q)}.
    \]
    Setting \(p=1/2-\Delta\) and \(q=1/2\), we obtain \(\kl(1/2-\Delta,\,1/2) \le 4\Delta^2\).  Therefore,
    \[
    (k-m)\frac{\Delta}{\kl(1/2-\Delta,\,1/2)}
    \ge
    (k-m)\frac{\Delta}{4\Delta^2}
    =
    \frac{k-m}{4\Delta}.
    \]
    By Lemma \ref{lem:reduction-lower-bound}, this is a lower bound for the delayed and costly class, which proves the claim. \qed

\smallskip
\noindent{\bf Proof of Proposition \ref{888}.}
    The result is immediate from applying Lemma \ref{lem:reduction-lower-bound} to the lower bound from \cite{kveton2015tight}. \qed

\section{Analysis of the Rank-Matching Bijection} \label{sec:bijection_analysis}
Let \(A^- := A_t \setminus U_\ell\) and \(A^+ := U_\ell \setminus A_t\) denote the sets of workers to be removed and added, respectively. For a bijection \(\pi : A^- \to A^+\), replacing worker \(i \in A^-\) with \(\pi(i)\) incurs a delay \(\omega_{i,\pi(i)} \in \{0,\dots,\bar{\omega}\}\). During the transition, this replacement contributes reward from worker \(i\) for \(\omega_{i,\pi(i)}\) periods and from worker \(\pi(i)\) for up to \(\bar{\omega}-\omega_{i,\pi(i)}\) periods. Let \(\boldsymbol{\mu}^-\) and \(\boldsymbol{\mu}^+\) denote the (unknown) mean rewards of workers \(i \in A^{-}\) and \(j \in A^{+}\), respectively. Based on the information available at \(t\), we consider the rectangular confidence set
\[
    \Theta=\left(\prod_{i\in A^{-}} [\LCB_i(t), \UCB_i(t)]\right)
    \;\times\;\left(\prod_{j\in A^{+}} [\LCB_j(t), \UCB_j(t)]\right).
\]
For a bijection \(\pi:A^- \to A^+\), we define a surrogate of the realized worst-case transition reward
\begin{align*}
    Z(\omega, \pi):=
    \inf_{(\boldsymbol{\mu}^-,\boldsymbol{\mu}^+)\in\Theta}\sum_{i\in A^-}\left(\omega_{i,\pi(i)}\,\mu_i+(\bar{\omega}-\omega_{i,\pi(i)})\,\mu_{\pi(i)}\right)
\end{align*}
We begin by observing that this is coordinatewise increasing in \((\boldsymbol{\mu}^-,\boldsymbol{\mu}^+)\) and that the infimum is attained at the lower confidence bounds. Substituting yields
\[Z(\pi;\omega )=
    ~\sum_{i\in A^-}\left(\omega_{i,\pi(i)}\,\LCB_i+(\bar{\omega}-\omega_{i,\pi(i)})\,\LCB_{\pi(i)}\right),\]
where we suppress the common argument \(t\) for readability. 

The quantity \(Z(\pi;\omega)\) captures a notion of worst-case intermediate workforce output of a bijection under conservative estimates of worker productivity and a given realization of delays. 

For each potential replacement \((i,j)\in A^-\times A^+\), let \(\omega_{i,j}\in\{0,\dots,\bar{\omega}\}\) denote a random delay. We assume that the collection \(\{\omega_{i,j}\}_{(i,j)\in A^-\times A^+}\) is independent and identically distributed according to a common distribution \(P\) supported on \(\{0,\dots,\bar{\omega}\}\). Since the support is bounded, all moments of \(P\) are finite. In particular, we denote
\[\mathbb{E}[\omega_{i,j}] = \mu, \quad \text{ and } \quad \mathbb{V}(\omega_{i,j}) = \sigma^2, \]
which are identical for all \((i,j)\in A^{-}\times A^+\). For a bijection \(\pi:A^-\to A^+\), the realized worst-case transition reward \(Z(\pi;\omega )\) is a random variable due to the stochastic delays \(\{\omega_{i,\pi(i)}\}_{i\in A^-}\). Then, define the worst-case expected transition reward under stochastic delays as
\[ \Phi^{\mathrm{i.i.d.}}(\pi) \;:=\; \bbE{\omega^{\pi}}{Z(\pi;\omega)}.\]
We next establish two properties of the rank-matching bijection under this stochastic delay model.

\begin{proposition}[Constant value under i.i.d.\ delays]
\label{lem:iid-rank-matching-E}
In the case of i.i.d. delays, \(\Phi^{\mathrm{i.i.d.}}(\pi)\) is independent of the bijection \(\pi:A^-\to A^+\). That is, every bijection is optimal in terms of average worst-case transition reward.
\end{proposition}

\noindent {\bf Proof.} 
    Fix a bijection \(\pi:A^-\to A^+\). By the definition of \(Z(\pi;\omega)\), the objective \(\Phi^{\mathrm{I.I.D.}}(\pi) = \bE{Z(\pi;\omega)}\) can be written as 
    \[\bE{\sum_{i\in A^-}\omega_{i,\pi(i)}\LCB_i + (\bar{\omega}-\omega_{i,\pi(i)})\LCB_{\pi(i)}}\]
    Using linearity of expectation and \(\bE{\omega_{i,\pi(i)}}=\mu\) for all \(i\in A^-\),
    \[\mu\sum_{i\in A^-}\LCB_i + (\bar{\omega}-\mu)\sum_{i\in A^+}\LCB_{\pi(i)}\]
    Then because 
    \[\sum_{i\in A^-}\LCB_{\pi(i)}=\sum_{j\in A^+}\LCB_j\]
    it follows that \(\Phi^{\mathrm{I.I.D.}}(\pi)\) is constant over bijections, and every bijection attains the maximum.
\qed

While Proposition \ref{lem:iid-rank-matching-E} shows that all bijections attain the same \emph{expected} worst-case transition reward under i.i.d.\ delays, different bijections can induce different levels of variability in the realized reward during the transition period. There are many operational reasons why lowering the variability of output is desirable. For this objective, the rank-matching bijection is optimal.

\begin{proposition}[Variance minimization under i.i.d.\ delays]
\label{lem:iid-rank-matching-Var} 
In the case of i.i.d. execution delays, for each bijection \(\pi:A^-\to A^+\), let \(\pi^{\mathrm{R}}\) be the rank-matching bijection. Then \(\pi^{\mathrm{R}}\) minimizes the variance of the realized transition reward, i.e., 
    \[
    \pi^{\mathrm{R}} \in \argmin_{\pi:A^-\to A^+}\; \mathbb{V}_{\omega^{\pi}}\!\bigl(Z(\pi;\omega^{\pi})\bigr).
    \]
\end{proposition}

\noindent {\bf Proof.} 
    Using that 
    \[\sum_{i\in A^-}\LCB_{\pi(i)}=\sum_{j\in A^+}\LCB_{j}\]
    the expression \(Z(\pi;\omega)\) can be written as
    \[\bar{\omega}\sum_{j\in A^+}\LCB_{j} +\sum_{i\in A^-}\omega_{i,\pi(i)}\bigl(\LCB_i-\LCB_{\pi(i)}\bigr)\]
    The first term does not depend on \(\pi\) and is deterministic. Therefore, it does not affect the variance. Using independence of \(\{\omega_{i,\pi(i)}\}_{i\in A^-}\) and \(\mathbb{V}(\omega_{i,\pi(i)})=\sigma^2\),
    \begin{align*}
        \mathbb{V}\bigl(Z(\pi;\omega)\bigr) &= \mathbb{V}\!\left(\sum_{i\in A^-}\omega_{i,\pi(i)}(\LCB_i-\LCB_{\pi(i)})\right) \\
        &= \sum_{i\in A^-}(\LCB_i-\LCB_{\pi(i)})^2\,\mathbb{V}(\omega_{i,\pi(i)}) 
        = \sigma^2\sum_{i\in A^-}(\LCB_i-\LCB_{\pi(i)})^2.
    \end{align*}
    Therefore, minimizing \(\mathbb{V}(Z(\pi;\omega^{\pi}))\) over bijections \(\pi\) is equivalent to minimizing 
    \[\sum_{i\in A^-}(\LCB_i-\LCB_{\pi(i)})^2.\]
    Let \(x_1\le\cdots\le x_r\) be the sorted list of \(\{\LCB_i:i\in A^-\}\) and \(y_1\le\cdots\le y_r\) be the sorted list
    of \(\{\LCB_j:j\in A^+\}\). Any bijection \(\pi\) corresponds to a permutation \(\tau\) of \(\{1,\dots,r\}\), and the
    objective becomes \(\sum_{i=1}^r (x_i-y_{\tau(i)})^2\). Expanding,
    \[
    \sum_{i=1}^r (x_i-y_{\tau(i)})^2
    =
    \sum_{i=1}^r x_i^2+\sum_{i=1}^r y_i^2
    -2\sum_{i=1}^r x_i\,y_{\tau(i)}.
    \]
    The first two sums are independent of \(\tau\), so minimizing the squared-difference sum is equivalent to maximizing \(\sum_{i=1}^r x_i\,y_{\tau(i)}\). By the rearrangement inequality, this maximum is attained when the sequences are paired in the same order, i.e., \(\tau(i)=i\) for all \(i\). This corresponds exactly to the rank-matching bijection \(\pi^{\mathrm{R}}\). Hence \(\pi^{\mathrm{R}}\) minimizes \(\sum_{i\in A^-}(\LCB_i-\LCB_{\pi(i)})^2\), and thus
    minimizes \(\mathbb{V}_{\omega^{\pi}}(Z(\pi;\omega^{\pi}))\).
\qed 

\section{Proofs of Miscellaneous Lemmas}
\begin{lemma}\label{lem:Gc}
    Define the good event
    \[\mathcal{G}_j(t)=\bigcap_{i\leq m}\bigg\{\LCB_i(t)<\mu_i<\UCB_i(t),\LCB_j(t)<\mu_j<\UCB_j(t)\bigg\}\]
    Then 
    \[\bP{\mathcal{G}^c_j(t)}\leq \frac{2(m+1)}{T^2}\leq \frac{4m}{T^2}\]
\end{lemma}
\noindent {\bf Proof.}
    By definition we have that 
    \[\mathcal{G}^c_j(t) = \left\{|\mu_j - \bar{\mu}_j| \geq \sqrt{\frac{\log T}{T_j(t)}}\right\} \cup \left(\bigcup_{i\leq m} \left\{|\mu_i - \bar{\mu}_i| \geq \sqrt{\frac{\log T}{T_i(t)}}\right\}\right)\]
    Then using a union bound and Hoeffding's inequality we get
   \begin{eqnarray*}
    \bP{\mathcal{G}^c_j(t)}&\leq& \bP{|\bar{\mu}_j(t)-\mu_j|\geq \sqrt{\frac{\log T}{T_j(t)}}}+\sum_{i\leq m}\bP{|\bar{\mu}_i(t)-\mu_i| 
      \geq \sqrt{\frac{\log T}{T_i(t)}}} 
      \\
    &\leq& \frac{2(m+1)}{T^2}
      \leq \frac{4m}{T^2}.
\end{eqnarray*} 
The proof is complete. \qed  

\begin{lemma}\label{lem:arm-comp}
    Let \(i\leq m <j\) be two arms with means \(\mu_i>\mu_j\) such that \(i\) is optimal and \(j\) is not. Define \(\Delta_{i,j}=\mu_i-\mu_j\). Then under the event \(\mathcal{G}_j(t)\)
    \begin{equation}\label{eq:lem-comp-T}
        T_i(t) > \frac{4\log T}{\Delta_{i,j}^2} \quad \implies \quad \UCB_j(t) < \UCB_i(t)
    \end{equation}
\end{lemma}
\noindent {\bf Proof.}
    Define 
    \[b_i(t)=\sqrt{\frac{\log T}{T_i(t)}}\]
    We will prove the contrapositive -- if \(\UCB_i(t)\geq \UCB_j(t)\) then \(T_i(t) \leq  {4\log T}/{\Delta_{i,j}^2}\).  Assume the good event \( \mathcal{G}_i(t) \) holds. Then by definition
    \[
        \bar{\mu}_i(t) - \mu_i < b_i(t) \quad \text{and} \quad \mu_j-\bar{\mu}_j(t) < b_i(t).
    \]
    This implies
    \[
        \bar{\mu}_j(t) < \mu_j + b_i(t) \quad \implies \quad \bar{\mu}_j(t) + b_i(t) < \mu_j + 2b_i(t).
    \]
    By assumption,
    \[
        \UCB_j(t) \geq  \UCB_i(t) \quad \implies \quad \bar{\mu}_j(t) + b_i(t) \geq \bar{\mu}_i(t) + b_i(t).
    \]
    Combining the above inequalities,
    \begin{align*}
        \bar{\mu}_i(t) + b_i(t) &\leq \bar{\mu}_j(t) + b_i(t) 
        < \mu_j + 2b_i(t).
    \end{align*}
    On the other hand, from \( \mathcal{G}_i(t) \), we have \( \bar{\mu}_i(t) > \mu_i - b_i(t) \), so
    \[
        \mu_i =\mu_i - b_i(t) + b_i(t)\leq \bar{\mu}_i(t) + b_i(t) \leq \mu_j + 2b_i(t).
    \]
    Thus,
    \[
        \mu_i - \mu_j \leq 2b_i(t) \quad \implies \quad \Delta_{i,j} \leq 2b_i(t).
    \]
    Now plug in the definition \( b_i(t) = \sqrt{\frac{2\log T}{T_j(t)}} \),
    \[
        \Delta_{i,j} \leq 2 \sqrt{\frac{\log T}{T_j(t)}} \quad \implies \quad T_j(t) \leq \frac{4\log T}{\Delta_{i,j}^2}.
    \]
    which proves the contrapositive. Therefore, if \( \mathcal{G}_i(t) \) holds and 
    \[
        T_j(t) > \frac{4\log T}{\Delta_{i,j}^2} \quad \implies \quad \UCB_j(t)<\UCB_i(t)
    \]
The lemma is thus proved. \qed  

\begin{lemma}\label{lem:technical-1}
    \[\frac{1}{\Delta_{1,j}}+\sum_{i=2}^{m}\Delta_{i,j}\left(\frac{1}{\Delta^2_{i,j}}-\frac{1}{\Delta^2_{i-1,j}}\right)\leq \frac{2}{\Delta_{m,j}}\]
\end{lemma}
\noindent {\bf Proof.}
    Using simple algebra, we have 
    \begin{eqnarray*}
        &&\frac{1}{\Delta_{1,j}}+\sum_{i=2}^{m}\Delta_{i,j}\left(\frac{1}{\Delta^2_{i,j}}-\frac{1}{\Delta^2_{i-1,j}}\right)
        = \sum_{i=1}^{m-1}\left(\frac{\Delta_{i,j}}{\Delta^2_{i,j}}-\frac{\Delta_{i+1,j}}{\Delta^2_{i,j}}\right)+\frac{1}{\Delta_{m,j}} \quad \hbox {(Shift the index)}\\
        &&\qquad = \sum_{i=1}^{m-1}\left(\frac{\Delta_{i,j}-\Delta_{i+1,j}}{\Delta^2_{i,j}}\right)+\frac{1}{\Delta_{m,j}}
        \leq \sum_{i=1}^{m-1}\left(\frac{\Delta_{i,j}-\Delta_{i+1,j}}{\Delta_{i,j}\Delta_{i+1,j}}\right)+\frac{1}{\Delta_{m,j}} \quad  (\hbox {by } \Delta_{i,j}\geq \Delta_{i+1,j})\\
        && \qquad = \sum_{i=1}^{m-1}\left(\frac{1}{\Delta_{i+1,j}}-\frac{1}{\Delta_{i,j}}\right)+\frac{1}{\Delta_{m,j}}
        = \frac{2}{\Delta_{m,j}}-\frac{1}{\Delta_{1,j}}
        \leq \frac{2}{\Delta_{m,j}}.
    \end{eqnarray*}
The result thus follows. 
\qed  

\begin{lemma}\label{lem:sum-of-ratios}
    \[\sum_{i=1}^{m-1}\frac{\Delta_{i,i+1}}{\Delta_{i,j}}\leq \log\left(\frac{\Delta_{1,j}}{\Delta_{m,j}}\right)\]
\end{lemma}
\noindent {\bf Proof.}
    The result follows from some algebra.
\begin{eqnarray*}
    &&\sum_{i=1}^{m-1} \frac{\Delta_{i,i+1}}{\Delta_{i,j}}
        = \sum_{i=1}^{m-1} \frac{\Delta_{i,j}-\Delta_{i+1,j}}{\Delta_{i,j}}
        = \sum_{i=1}^{m-1} \left(1 - \frac{\Delta_{i+1,j}}{\Delta_{i,j}}\right)
        \leq  \sum_{i=1}^{m-1} \log\left( \frac{\Delta_{i,j}}{\Delta_{i+1,j}}\right)\\
       &&\qquad =\sum_{i=1}^{m-1}\left( \log \Delta_{i,j} - \log\Delta_{i+1,j}\right)
        =\left( \log \Delta_{1,j} - \log\Delta_{m,j}\right)
        =\log\left(\frac{\Delta_{1,j}}{\Delta_{m,j}}\right),
  \end{eqnarray*}
  where the inequality follows from $1-1/x\leq \log x$ for $x\ge 1$. 
  Throughout, we write \(\log\) instead of \(\ln\).
\qed  

\begin{lemma}\label{lem:exact-quadratic}
    Let \(\gamma \ge 0\) and define the sequence \(\{a_n\}_{n\ge 0}\) by
    \[
        a_0 \geq 0,
        \quad
        a_n := a_{n-1} + 2\sqrt{\gamma  a_{n-1}} + \gamma
        \quad\text{for } n\ge 1.
    \]
    Then for all \(n\ge 0\), we have \(a_n = \left(\sqrt{a_0}+n\sqrt{\gamma}\right)^2\).
\end{lemma}

\noindent {\bf Proof.} 
    We prove by induction on \(n\). For the base case \(n=0\), \(a_0=\left(\sqrt{a_0}+0\cdot \sqrt{\gamma}\right)^2\). Now assume the claim holds for some \(n\ge 0\). Using the recurrence,
    \[
    a_{n+1} = a_n+2\sqrt{\gamma a_n}+\gamma = \left(\sqrt{a_n}+\sqrt{\gamma}\right)^2.
    \]
    By the induction hypothesis,
    \[
    \sqrt{a_n}
    = \sqrt{\left(\sqrt{a_0}+n\sqrt{\gamma}\right)^2}
    = \sqrt{a_0}+n\sqrt{\gamma},
    \]
    where the last equality holds since \(\sqrt{a_0}+n\sqrt{\gamma}\ge 0\). Therefore,
    \[
    a_{n+1}
    = \left(\sqrt{a_0}+n\sqrt{\gamma}+\sqrt{\gamma}\right)^2
    = \left(\sqrt{a_0}+(n+1)\sqrt{\gamma}\right)^2.
    \]
    This completes the induction.
\qed

\section{Table of Related Multi-Armed Bandit Works}
\label{sec:related_works_table}
To facilitate direct comparison, Table~\ref{tab:bandit-contrast} summarizes the most closely related models and contrasts them with our setting.
\begin{table}[H]
   \caption{Comparison with selected related works on bandit learning}
   \label{tab:bandit-contrast}
   \renewcommand{\arraystretch}{3}
   \begin{tabular}{p{0.11\linewidth} p{0.17\linewidth} p{0.15\linewidth} p{0.13\linewidth} p{0.14\linewidth} p{0.20\linewidth}}
   \toprule
   
   \textbf{\makecell[c]} & \textbf{\makecell[c]{Action\\Space}} & \textbf{\makecell[c]{Reward\\Model}} & \textbf{\makecell[c]{Switching\\Costs}} & \textbf{\makecell[c]{Replacement\\Delays}} & \textbf{\makecell[c]{Instance\\Independent\\Regret}} \\
   \midrule

   \rowcolor{gray!10} 
   {\DRUCB} (Ours)& \makecell[c]{\(m\)-subset
   } & \makecell[c]{Stochastic} & \makecell[c]{\checkmark} & \makecell[c]{\checkmark} & \makecell[c]{\(\tilde{O}(\sqrt{mkT})\)}\\

   \cite{kveton2014matroid} & \makecell[c]{General matroids\\(incl. \(m\)-subset)} & \makecell[c]{Stochastic} & \makecell[c]{\(\times\)} &\makecell[c]{\(\times\)}& \makecell[c]{\(\tilde{O}(\sqrt{mkT})\)}\\

   \rowcolor{gray!10}
   \cite{agrawal1990switching} & \makecell[c]{\(m\)-subset
   } & \makecell[c]{Stochastic} & \makecell[c]{\checkmark} &\makecell[c]{\(\times\)} & \makecell[c]{Only instance\\dependent}\\

   \cite{amir2022better} & \makecell[c]{Single-arm} & \makecell[c]{Stochastic \&\\ Adversarial} & \makecell[c]{\checkmark} &\makecell[c]{\(\times\)} & \makecell[c]{\(O(\sqrt{ckT})\)\\ (Stochastic case)}\\

   \rowcolor{gray!10}
   \cite{huang2024adversarial} & \makecell[c]{Combinatorial} & \makecell[c]{Adversarial} & \makecell[c]{\checkmark} &\makecell[c]{\(\times\)} & \makecell[c]{\(O((\sqrt{mk}+cm)T^{2/3})\)}\\

   \cite{basu2019blocking} & \makecell[c]{Single-arm} & \makecell[c]{Stochastic} & \makecell[c]{\(\times\)} & \makecell[c]{\(\times\)} & \makecell[c]{n/a}\\

   \rowcolor{gray!10}
   \cite{joulani2013online} & \makecell[c]{Single-arm} & \makecell[c]{Stochastic \&\\ Adversarial} & \makecell[c]{\(\times\)} & \makecell[c]{\(\times\)} & \makecell[c]{\(\tilde{O}(\sqrt{kT})\)\\ (Stochastic case)}\\
   \bottomrule
   \end{tabular}
   \par\smallskip
   \noindent{\footnotesize \(\tilde{O}(\cdot)\) hides log terms.}
\end{table}

\section{Computation of Selection Rule} \label{sec:comp_of_selection_rule}
In this section, we provide a dynamic program that computes the solution \(U\) to the optimization problem
\begin{align*}
    \max_{U} \quad & \sum_{i\in U}\UCB_i(t)\\
    \text{subject to}\quad &\sum_{i\in A_t\setminus U}\delta_{i,\pi^\mathrm{R}(i)}(t)\geq c\cdot \frac{| A_t\setminus U|}{T-t}
\end{align*}
To solve the constrained selection problem, the policy must determine both which inactive arms should enter the target workforce and which currently active arms should be removed. Lemma~\ref{lem:monotonicity-selection} shows that the entering arms satisfy a useful monotonicity property: if \(\UCB_i(t)>\UCB_j(t)\) and \(\delta_{j,i}(t)\ge c/(T-t)\), then any feasible optimal solution containing \(j\) must also contain \(i\). As a result, for any fixed number \(r\) of incoming arms, the optimal entering set can be selected greedily by taking the \(r\) inactive arms with the largest upper confidence bounds. The remaining task is to determine how many replacements to make and which active arms to remove. For each \(r\in\{0,\dots,\min(m,k-m)\}\), this reduces to a cardinality-constrained knapsack problem on the active set. Algorithm~\ref{alg:ChooseTarget-Exact} summarizes this procedure.

\begin{algorithm}
\caption{\textsc{ChooseTarget-Exact}}
\label{alg:ChooseTarget-Exact}
\begin{algorithmic}[1]
    \Require current active workforce \(A_t\), current statistics at time \(t\)
    \State Compute \(\UCB_i(t)\) and \(\LCB_i(t)\) for all \(i\in[k]\)
    \State Set \(\tau_t \gets c/(T-t)\) and \(B_t \gets [k]\setminus A_t\)
    \State Sort \(B_t\) in decreasing order of \(\UCB\): \(b_1,\dots,b_{k-m}\)
    \State Compute \(I_r \gets \sum_{\ell=1}^r \UCB_{b_\ell}(t)\) for \(r=0,\dots,k-m\)
    \State Initialize \(U^* \gets A_t\) and \(V^* \gets \sum_{i\in A_t}\UCB_i(t)\)

    \For{$r=1,\dots,\min\{m,k-m\}$}
        \State Fix incoming set \(B_r \gets \{b_1,\dots,b_r\}\)
        \State Set budget \(K_r \gets I_r-r\tau_t\)
        \State Solve the cardinality constrained knapsack problem on the active set \(A_t\)
        \[
            \min_{S\subseteq A_t:\,|S|=r}
            \sum_{i\in S}\UCB_i(t)
            \qquad
            \text{s.t. }
            \sum_{i\in S}\LCB_i(t)\le K_r
        \]
        \If{a feasible minimizer \(S_r\) exists}
            \State Set \(U_r \gets (A_t\setminus S_r)\cup B_r\)
            \State Set \(V_r \gets \sum_{i\in U_r}\UCB_i(t)\)
            \If{$V_r>V^*$}
                \State \(U^*\gets U_r\), \(V^*\gets V_r\)
            \EndIf
        \EndIf
    \EndFor
    \State \Return \(U^*\)
\end{algorithmic}
\end{algorithm}

The cardinality-constrained knapsack problem in line 9 is solved using \textsc{FrontierConstrainedKnapsack} (Algorithm \ref{alg:frontier-constrained-knapsack}). A central challenge is that the subproblem is combinatorial: there are \(\binom{m}{r}\) candidate subsets, and the objective and constraint act in competing directions. Minimizing \(\sum_{i\in S}\UCB_i(t)\) favors removing workers with small upper confidence bounds, while the feasibility condition restricts the total lower confidence bound of the removed set. As a result, simple greedy rules generally fail to identify the optimum.

The key idea is to build the solution incrementally while retaining only the relevant tradeoffs. As the active workers are processed one at a time, the algorithm keeps track, for each cardinality \(q\), of the achievable pairs
\[
    \left(\sum_{i\in S}\LCB_i(t),\ \sum_{i\in S}\UCB_i(t)\right)
\]
generated by subsets \(S\) of size \(q\). Any pair that is dominated in both coordinates by another can be discarded, since it can never lead to an optimal solution. The remaining pairs form a Pareto frontier, which summarizes all nonredundant tradeoffs between feasibility and objective value. After all workers have been processed, the frontier \(\mathcal{F}_{m,r}\) contains exactly the nondominated tradeoffs for subsets of size \(r\). The optimal solution is then obtained by selecting, among the frontier points satisfying \(\sum_{i\in S}\LCB_i(t)\le K_r\), the one with minimum \(\sum_{i\in S}\UCB_i(t)\), and returning the associated subset \(S_r\). This idea is implemented in Algorithm \ref{alg:frontier-constrained-knapsack}.

Although the frontier can be exponential in the worst case, the algorithm is output-sensitive: its running time is polynomial in the number of active workers and in the size of the maintained frontier. In particular, it avoids any discretization of the confidence bounds and does not introduce dependence on the numerical magnitudes of \(\UCB_i(t)\), \(\LCB_i(t)\), or \(K_r\). Moreover, as the confidence radii shrink, the confidence bounds become increasingly aligned, so dominance pruning removes a larger fraction of candidate states. In the benchmark setting studied in Section \ref{numerical}, the frontier size is typically much smaller than \(m\). We present numerical experiments tracking the frontier size and overall runtime in Appendix \ref{sec:runtime}.

\begin{algorithm}
\caption{\textsc{FrontierConstrainedKnapsack}}
\label{alg:frontier-constrained-knapsack}
\begin{algorithmic}[1]
    \Require Active set \(A_t=\{a_1,\dots,a_m\}\), confidence bounds \(\UCB_i(t),\LCB_i(t)\), replacement count \(r\), budget \(K_r\)
    \State Initialize \(\mathcal{F}_{0,0}\gets \{(0,0)\}\) and \(\mathcal{F}_{0,q}\gets \emptyset\) for \(q=1,\dots,r\)

    \For{$j=1,\dots,m$}
        \For{$q=0,\dots,\min\{j,r\}$}
            \State Form all candidate pairs obtained by either
            excluding worker \(a_j\), or including worker \(a_j\):
            \[
                \mathcal{C}_{j,q}
                \gets
                \mathcal{F}_{j-1,q}
                \;\cup\;
                \left\{
                \bigl(L+\LCB_{a_j}(t),\,U+\UCB_{a_j}(t)\bigr)
                :
                (L,U)\in \mathcal{F}_{j-1,q-1}
                \right\}
            \]
            \State Prune dominated pairs from \(\mathcal{C}_{j,q}\), keeping only Pareto-optimal tradeoffs:
            \[
                \mathcal{F}_{j,q}
                \gets
                \left\{
                (L,U)\in\mathcal{C}_{j,q}:
                \nexists (L',U')\in\mathcal{C}_{j,q}
                \text{ with } L'\le L,\ U'\le U,
                \text{ and one strict}
                \right\}
            \]
        \EndFor
    \EndFor
    
    \State Let \((L^*,U^*)\) satisfy
    \[
        U^*=\min\{U:(L,U)\in\mathcal{F}_{m,r},\ L\le K_r\}
    \]
    \State \Return A set \(S_r\subseteq A_t\) with \(|S_r|=r\) such that
    \[
        \sum_{i\in S_r}\LCB_i(t)=L^*,
        \qquad
        \sum_{i\in S_r}\UCB_i(t)=U^*
    \]
\end{algorithmic}
\end{algorithm}

More generally, the proof of the regret bound requires only that the selected target workforce satisfy the monotonicity property in Lemma~\ref{lem:monotonicity-selection}. This leaves room for approximate or heuristic solution methods. To this end, we introduce \textsc{MonotoneClosure}, which takes any feasible target workforce and iteratively applies improving exchanges until the monotonicity property holds. 

Consequently, one may replace the exact solver in Algorithm~\ref{alg:frontier-constrained-knapsack} with an approximate solver or heuristic, then apply \textsc{MonotoneClosure} as a post-processing step. By Lemma~\ref{lem:monotonicity-selection}, each exchange preserves feasibility and strictly improves the objective, so the post-processing step can only improve the provided solution while enforcing the structural condition used in the regret analysis. This allows for a number of efficient implementations that retain the same theoretical guarantee.

\begin{algorithm}
\caption{\textsc{MonotoneClosure}}
\label{alg:MonotoneClosure}
\begin{algorithmic}[1]
    \Require feasible target workforce \(U\) at time \(t\)
    \While{there exist \(i\notin U\) and \(j\in U\) such that
    \[
        \UCB_i(t)>\UCB_j(t)
        \qquad\text{and}\qquad
        \delta_{j,i}(t)\ge \frac{c}{T-t}
    \]}
        \State \(U \gets (U\setminus\{j\})\cup\{i\}\)
    \EndWhile
    \State \Return \(U\)
\end{algorithmic}
\end{algorithm}

\section{Interview Benchmark}
\label{sec:interview-model}
We consider a pre-screening benchmark that models a firm selecting workers using an imperfect signal available before any work is observed. This policy is intended to capture common screening procedures such as resume review, interviews, or standardized assessments. Unlike learning-based policies, the pre-screening policy uses only its initial screening information and does not update its target workforce based on realized productivity feedback.

Let \(\mu_i\) denote the unknown mean productivity of worker \(i\in[k]\). At time \(0\), the firm obtains a noisy screening score \(S_i\) for each worker. We model this score as an imperfect signal of worker productivity,
\[
    S_i = \mu_i + \varepsilon_i,
\]
where the noise terms are generated so that the across-worker sample correlation between the screening scores and the true worker means is equal to a specified parameter \(\rho\):
\[
    \operatorname{Corr}\big((S_i)_{i=1}^k,(\mu_i)_{i=1}^k\big)=\rho.
\]
The parameter \(\rho\in[-1,1]\) therefore measures the effectiveness of the screening technology. When \(\rho=1\), screening perfectly ranks workers by their true mean productivity. When \(\rho=0\), screening contains no linear information about worker productivity. In our simulations we set \(\rho=0.3\), based on results on interviewing predictiveness in \cite{roulin2019conducting}. 

The pre-screening procedure also incurs an up-front cost. We denote this cost by \(c^{\mathrm{interview}}\). In the simulations, this cost is added to the regret accounting at the beginning of the horizon. In the absence of reliable sources for interview costs, we set \(c^{\mathrm{interview}}=c\) in our simulations.

Given the screening scores, the policy constructs a fixed target workforce by selecting the \(m\) workers with the largest screening scores:
\[
    U^{\mathrm{interview}}
    \in
    \argmax_{\substack{U\subseteq[k]\\ |U|=m}}
    \sum_{i\in U} S_i .
\]
Equivalently, \(U^{\mathrm{interview}}\) is the set of the top \(m\) workers according to the initial interview scores. Ties are broken randomly. After this initial target is computed, the policy never changes its target workforce. The policy therefore represents a typical screening-type rule: it chooses a workforce using pre-hiring information and does not adapt to realized worker performance.

\section{Additional Numerical Results}
In this section, we present several additional numerical illustrations, as well as tables of values from the simulations reported in Section \ref{numerical}.

\subsection{Tables from Section \ref{numerical:benchmarks}} \label{sec:benchmark_tables}

The table reports simulations for \(m \in \{5,10,20\}\) at 3-, 6-, 12-, and 24-month horizons. All other parameters are fixed (\(k=25, c=5,\bar{\omega}=3\)). Entries are reported as the sample mean plus or minus one standard deviation over 250 simulated instances.

\begin{table}[H]
    \centering

    \small
    \setlength{\tabcolsep}{4pt}
    
    \begin{tabular*}{\linewidth}{@{\extracolsep{\fill}} l r r r r r}
    \toprule
    Horizon & DR-UCB & A-AHT & A-OMM & InterviewScreen & WorkTrial \\
    \midrule
    3 months & $253 \pm 18$ & $299 \pm 23$ & $498 \pm 36$ & $369 \pm 39$ & $419 \pm 39$ \\
    6 months & $305 \pm 18$ & $417 \pm 32$ & $757 \pm 49$ & $581 \pm 79$ & $745 \pm 40$ \\
    12 months & $402 \pm 24$ & $649 \pm 46$ & $1{,}145 \pm 59$ & $1{,}007 \pm 158$ & $799 \pm 73$ \\
    24 months & $491 \pm 25$ & $1{,}056 \pm 60$ & $1{,}743 \pm 81$ & $1{,}857 \pm 318$ & $896 \pm 145$ \\
    \bottomrule
    \end{tabular*}
    \caption{Cumulative regret across policies and time horizons. \((m=10)\)}
\end{table}
\vspace{-.3in}

\begin{table}[H]
    \centering
    
    \small
    \setlength{\tabcolsep}{4pt}
    
    \begin{tabular*}{\linewidth}{@{\extracolsep{\fill}} l r r r r r}
    \toprule
    Horizon & DR-UCB & A-AHT & A-OMM & InterviewScreen & WorkTrial \\
    \midrule
    3 months & $34.15\% \pm 2.41\%$ & $40.45\% \pm 3.17\%$ & $67.32\% \pm 4.90\%$ & $49.90\% \pm 5.34\%$ & $56.68\% \pm 5.31\%$ \\
    6 months & $20.63\% \pm 1.21\%$ & $28.18\% \pm 2.19\%$ & $51.14\% \pm 3.33\%$ & $39.27\% \pm 5.33\%$ & $50.32\% \pm 2.70\%$ \\
    12 months & $13.55\% \pm 0.81\%$ & $21.85\% \pm 1.54\%$ & $38.57\% \pm 1.98\%$ & $33.94\% \pm 5.34\%$ & $26.91\% \pm 2.44\%$ \\
    24 months & $8.28\% \pm 0.43\%$ & $17.79\% \pm 1.01\%$ & $29.37\% \pm 1.37\%$ & $31.29\% \pm 5.35\%$ & $15.10\% \pm 2.45\%$ \\
    \bottomrule
    \end{tabular*}
    \caption{Normalized loss across policies and time horizons. \((m=10)\)}
\end{table}
\vspace{-.3in}

\begin{table}[H]
    \centering

    \small
    \setlength{\tabcolsep}{4pt}
    
    \begin{tabular*}{\linewidth}{@{\extracolsep{\fill}} l r r r r r}
    \toprule
    Horizon & DR-UCB & A-AHT & A-OMM & InterviewScreen & WorkTrial \\
    \midrule
    3 months & $265 \pm 16$ & $277 \pm 24$ & $494 \pm 36$ & $273 \pm 35$ & $298 \pm 35$ \\
    6 months & $358 \pm 20$ & $408 \pm 34$ & $805 \pm 51$ & $398 \pm 72$ & $467 \pm 42$ \\
    12 months & $479 \pm 27$ & $669 \pm 46$ & $1{,}294 \pm 65$ & $651 \pm 145$ & $540 \pm 77$ \\
    24 months & $596 \pm 29$ & $1{,}108 \pm 57$ & $1{,}991 \pm 84$ & $1{,}156 \pm 292$ & $681 \pm 154$ \\
    \bottomrule
    \end{tabular*}
    \caption{Cumulative regret across policies and time horizons. \((m=5)\)}
          
\end{table}
\vspace{-.3in}

\begin{table}[H]
    \centering
    
    \small
    \setlength{\tabcolsep}{4pt}
    
    \begin{tabular*}{\linewidth}{@{\extracolsep{\fill}} l r r r r r}
    \toprule
    Horizon & DR-UCB & A-AHT & A-OMM & InterviewScreen & WorkTrial \\
    \midrule
    3 months & $65.00\% \pm 3.95\%$ & $67.89\% \pm 5.91\%$ & $120.91\% \pm 8.86\%$ & $66.76\% \pm 8.69\%$ & $72.90\% \pm 8.64\%$ \\
    6 months & $43.87\% \pm 2.47\%$ & $49.97\% \pm 4.22\%$ & $98.63\% \pm 6.30\%$ & $48.78\% \pm 8.79\%$ & $57.17\% \pm 5.20\%$ \\
    12 months & $29.26\% \pm 1.65\%$ & $40.87\% \pm 2.78\%$ & $79.05\% \pm 3.99\%$ & $39.77\% \pm 8.87\%$ & $32.97\% \pm 4.68\%$ \\
    24 months & $18.21\% \pm 0.90\%$ & $33.85\% \pm 1.73\%$ & $60.81\% \pm 2.55\%$ & $35.29\% \pm 8.91\%$ & $20.79\% \pm 4.71\%$ \\
    \bottomrule
    \end{tabular*}
    \caption{Normalized loss across policies and time horizons. \((m=5)\)}
\end{table}

\vspace{-.3in}

\begin{table}[H]
    \centering

    \small
    \setlength{\tabcolsep}{4pt}
    
    \begin{tabular*}{\linewidth}{@{\extracolsep{\fill}} l r r r r r}
    \toprule
    Horizon & DR-UCB & A-AHT & A-OMM & InterviewScreen & WorkTrial \\
    \midrule
    3 months & $135 \pm 21$ & $163 \pm 18$ & $259 \pm 27$ & $285 \pm 38$ & $311 \pm 37$ \\
    6 months & $135 \pm 21$ & $232 \pm 19$ & $429 \pm 39$ & $424 \pm 76$ & $506 \pm 43$ \\
    12 months & $193 \pm 28$ & $364 \pm 22$ & $707 \pm 57$ & $703 \pm 154$ & $509 \pm 43$ \\
    24 months & $271 \pm 33$ & $590 \pm 28$ & $1{,}117 \pm 74$ & $1{,}260 \pm 310$ & $509 \pm 43$ \\
    \bottomrule
    \end{tabular*}
    \caption{Cumulative regret across policies and time horizons. \((m=20)\)}
\end{table}

\vspace{-.3in}

\begin{table}[H]
    \centering
    
    \small
    \setlength{\tabcolsep}{4pt}
    
    \begin{tabular*}{\linewidth}{@{\extracolsep{\fill}} l r r r r r}
    \toprule
    Horizon & DR-UCB & A-AHT & A-OMM & InterviewScreen & WorkTrial \\
    \midrule
    3 months & $12.20\% \pm 1.91\%$ & $14.73\% \pm 1.63\%$ & $23.44\% \pm 2.47\%$ & $25.85\% \pm 3.40\%$ & $28.12\% \pm 3.38\%$ \\
    6 months & $6.13\% \pm 0.96\%$ & $10.51\% \pm 0.86\%$ & $19.43\% \pm 1.79\%$ & $19.21\% \pm 3.45\%$ & $22.92\% \pm 1.93\%$ \\
    12 months & $4.36\% \pm 0.63\%$ & $8.23\% \pm 0.51\%$ & $15.97\% \pm 1.28\%$ & $15.88\% \pm 3.49\%$ & $11.50\% \pm 0.98\%$ \\
    24 months & $3.06\% \pm 0.37\%$ & $6.66\% \pm 0.32\%$ & $12.61\% \pm 0.83\%$ & $14.23\% \pm 3.51\%$ & $5.75\% \pm 0.49\%$ \\
    \bottomrule
    \end{tabular*}
    \caption{Normalized loss across policies and time horizons. \((m=20)\)}
    
\end{table}


\subsection{Design-Choice Comparisons for {\DRUCB}}\label{sec:dr-ucb-internal}
Because the external benchmarks are adapted policies, it is useful to separately examine the role of the design choices within {\DRUCB}. Table~\ref{tab:dr-ucb-internal} compares {\DRUCB} with three modified versions of the policy under the benchmark calibration from Section~\ref{numerical:benchmarks}. 

First, we replace the adaptive count-based replacement rule with a fixed calendar-time rule. This comparison isolates the value of adjusting replacement timing to the amount of information collected. Second, we replace the constrained selection rule with the unconstrained version of {\ChooseTarget}, namely the classic top-\(m\) UCB target. This comparison measures the benefit of accounting for replacement costs when selecting the target workforce. Third, we evaluate the effect of committing to incoming--outgoing worker pairings at the time of hire. To do so, we compare against an oracle delayed-decision pairing rule that waits until incoming workers become available and then removes the available outgoing workers with the lowest true means first. This gives a favorable benchmark for delayed removal decisions and allows us to assess whether committing to pairings at hire time materially affects performance.

\begin{table}[H]
    \centering
    \small
    \setlength{\tabcolsep}{4pt}
    \begin{tabular*}{\linewidth}{@{\extracolsep{\fill}} l r r}
    \toprule
    Variant & 3-month regret & 12-month regret \\
    \midrule
    DR-UCB & $238 \pm 16$ & $407 \pm 24$ \\
    DR-UCB + fixed calendar switching & $297 \pm 23$ & $485 \pm 42$ \\
    DR-UCB without horizon screening & $267 \pm 17$ & $431 \pm 21$ \\
    DR-UCB + oracle delayed-decision pairing & $244 \pm 16$ & $416 \pm 17$ \\
    \bottomrule
    \end{tabular*}
    \caption{Design choice comparisons for \textsc{DR-UCB} in the benchmark calibration of Section~\ref{numerical:benchmarks}. Each row averages 250 replications.}
    \label{tab:dr-ucb-internal}
\end{table}

The simulation yields a useful hierarchy. The adaptive replacement rule has by far the clearest effect in this calibration: replacing it with a fixed calendar schedule worsens both 3-month and 12-month regret substantially. Replacing the constrained selection rule by the unconstrained version of {\ChooseTarget} also worsens performance, but by a smaller amount. The final simulation reveals that pairing at hiring decision time does not degrade performance.


To make the constrained selection mechanism more visible, Table~\ref{tab:dr-ucb-internal-stress} reports a small stress test. In this test, replacement costs are increased to \(c=168\), which makes the constrained selection more consequential.

\begin{table}[h]
    \centering
    \small
    \setlength{\tabcolsep}{4pt}
    \begin{tabular*}{\linewidth}{@{\extracolsep{\fill}} l l r}
    \toprule
    Scenario & Variant & 12-month regret \\
    \midrule
    High switching cost (c=168) & DR-UCB & $4{,}065 \pm 160$ \\
    High switching cost (c=168) & DR-UCB without horizon screening & $4{,}749 \pm 237$ \\
    \bottomrule
    \end{tabular*}
    \caption{High cost regime simulations for \textsc{DR-UCB}}
    \label{tab:dr-ucb-internal-stress}
\end{table}

The high-cost stress test shows that constrained selection matters much more once late replacements become expensive: the gap between screened and unscreened {\DRUCB} becomes economically meaningful. 

\subsection{Choosing \(\gamma\) for Empirical Performance}
Algorithm \ref{alg:DelayedReplace-UCB} uses a tuning parameter \(\gamma\) to control the frequency of staffing adjustments. This section examines the empirical sensitivity of the policy to this choice. In Section \ref{theoretical}, we proposed several theoretically motivated values of \(\gamma\) based on the regret bounds.

The first choice, \(\gamma_1 = (\bar{\omega}+c)^2 m\), yields the \(O(\sqrt{kmT})\) regret-growth rate in Theorem \ref{thm:main}. The second choice, \(\gamma_2 = (\bar{\omega}+c)m\), derived in Corollary \ref{thm:delta-main}, gives a regret bound with at most linear dependence on the problem parameters.

These choices are motivated by upper bounds, which need not give the best finite-sample performance. We therefore also consider a parameter-dependent choice, \(\gamma^*\), obtained by minimizing the full regret bound. Specifically, let \(f(\gamma;\,T,k,m,\bar{\omega},c)\) denote the regret bound from Appendix \ref{sec:thm-all-constants}, with all constants retained, viewed as a function of \(\gamma\). We define
\[
    \gamma^* \in \argmin_{\gamma} f(\gamma;\,T,k,m,\bar{\omega},c).
\]
This choice depends only on problem parameters known to the decision maker.

To compare these tuning rules, we use a larger instance with a 5-year horizon, where the effect of \(\gamma\) is more visible. For each setting, we evaluate \(\gamma_1\), \(\gamma_2\), and \(\gamma^*\), together with a grid of 25 values ranging from \(1\) to \(2\max\{\gamma_1,\gamma_2,\gamma^*\}\). Each configuration is run 100 times using seeds distinct from those used in the other experiments. Figure~\ref{fig:gamma_curve} summarizes the resulting performance trends. 

\begin{figure}[h]
    \centering
    \includegraphics[width=0.75\linewidth]{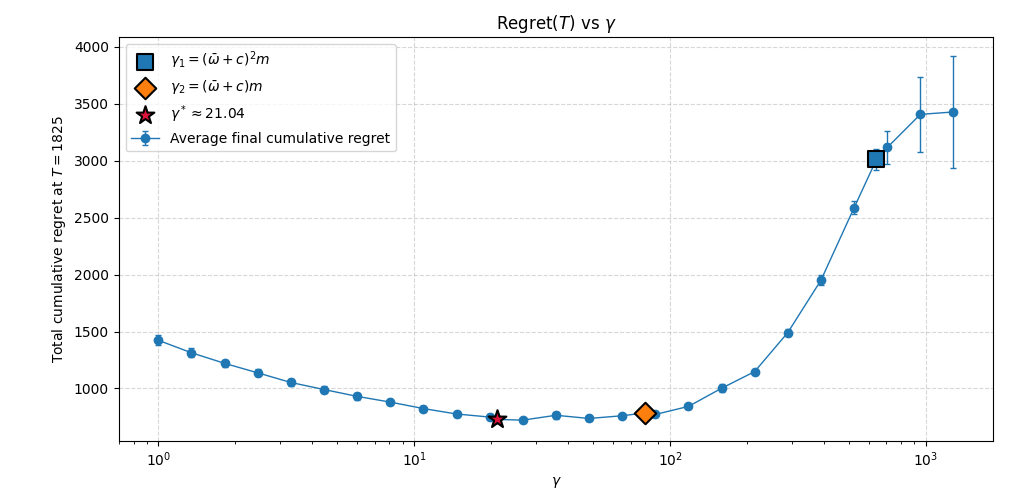}
    \caption{Sensitivity of cumulative regret over 5 years to the choice of \(\gamma\)}
    \label{fig:gamma_curve}
\end{figure}

The results in Figure \ref{fig:gamma_curve} yield two main observations. First, the theoretically optimal choice \(\gamma_1\) performs worst empirically among the three candidates. This discrepancy is driven by lower-order terms in the regret bound: while \(\gamma_1\) eliminates dependence on \(c\) and \(\bar{\omega}\) in the leading-order term in \(T\), it inflates additive terms that do not vanish with the horizon. As a result, \(\gamma_1\) may not achieve the best performance for specific problem instances over finite horizons. Second, the choice \(\gamma^*\) closely matches the empirically optimal value of \(\gamma\) across the tested range. This indicates that minimizing the full regret bound provides an effective and practically relevant strategy. Based on these results, we use \(\gamma^*\) in the benchmark studies.

\subsection{Runtime and Computational Tractability}
\label{sec:runtime}
This section examines the runtime and computational tractability of {\DRUCB}. Table~\ref{tab:runtime} reports runtime statistics for {\DRUCB}, \textsc{A-AHT}, and \textsc{A-OMM} in the benchmark simulation from Section~\ref{numerical:benchmarks} (\(T=730\), \(m=10\), \(k=25\)). {\DRUCB} is the fastest of the three policies, despite using a more involved selection rule. A likely explanation is that runtime is driven primarily by the number of replacements a policy initiates. Because {\DRUCB} makes the fewest replacements, it requires fewer replacement-related updates. This reduction offsets the additional computation required by its selection rule.

\begin{table}[h]
    \centering

    \small
    \setlength{\tabcolsep}{4pt}
    
    \begin{tabular*}{\linewidth}{@{\extracolsep{\fill}} l r r r}
    \toprule
    Metric & DR-UCB & A-AHT & A-OMM \\
    \midrule
    Mean (ms) & $15.3$ & $22.7$ & $38.0$ \\
    Standard Deviation (ms) & $0.1$ & $0.4$ & $0.2$\\
    \bottomrule
    \end{tabular*}
    \caption{Runtime statistics of learning policies for one full simulation \(T=730\) aggregated over 250 instances.}
    \label{tab:runtime}
\end{table}

We also directly address the question of computational tractability of the {\ChooseTarget} selection rule from Appendix \ref{sec:comp_of_selection_rule}. In Table \ref{tab:selection-runtime}, we report the average runtime of the selection rule aggregated over the benchmark simulations.

\begin{table}[h]
    \centering

    \small
    \setlength{\tabcolsep}{4pt}
    
    \begin{tabular*}{\linewidth}{@{\extracolsep{\fill}} l r r r}
    \toprule
     & Mean (ms) & Standard Deviation (ms) & Max (ms) \\
    \midrule
    {\ChooseTarget} (Algorithm \ref{alg:ChooseTarget-Exact})& $0.48$ & $0.03$ & $3.37$ \\
    \bottomrule
    \end{tabular*}
    \caption{Runtime statistics of the selection rule.}
    \label{tab:selection-runtime}
\end{table}

In Figure~\ref{fig:frontier-size}, we plot the observed frontier sizes for the simulation setup described in Section~\ref{numerical:benchmarks}. The smoothed trend is obtained via Gaussian kernel smoothing. The figure shows that the maintained frontier remains extremely small in practice: for \(m=10\), it never exceeds 13, and it decreases over time as the confidence bounds tighten.

\begin{figure}[h]
    \centering
    \includegraphics[width=0.75\linewidth]{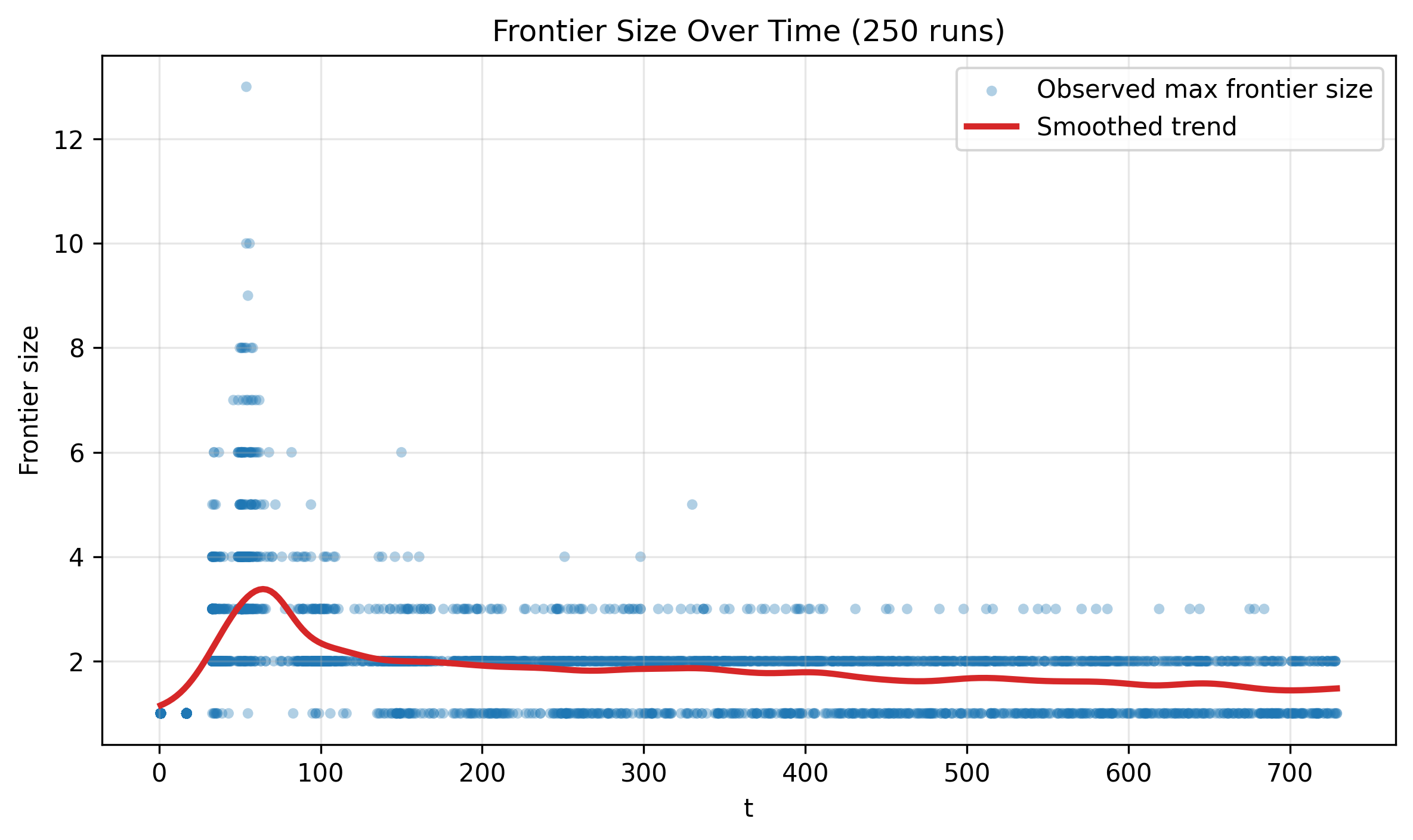}
    \caption{Observed frontier sizes for \textsc{FrontierConstrainedKnapsack} (Algorithm \ref{alg:frontier-constrained-knapsack}) over time.}
    \label{fig:frontier-size}
\end{figure}
\end{document}